\documentclass[preprint,11pt]{elsarticle}

\usepackage{amssymb,amsmath,amsthm}
\usepackage{graphicx}
\usepackage{float}
\restylefloat{figure}

\usepackage{fullpage}
\usepackage{float}
\restylefloat{figure}
\journal{Physica D}

\numberwithin{equation}{section}

\newtheorem{theorem}{Theorem}[section]
\newtheorem{lemma}[theorem]{Lemma}

\newtheorem{corollary}[theorem]{Corollary}
\newtheorem{example}{Example}
\newtheorem{remark}[theorem]{Remark}
\newtheorem{assumption}{Assumption}
\newcommand\be{\begin{enumerate}}
\newcommand\ee{\end{enumerate}}
\newtheorem{definition}[theorem]{Definition}

\begin{document}

\begin{frontmatter} 

\title%[Stabilization by noise with particular distributions]
{Global stabilization and destabilization by the state dependent  noise with particular distributions}
%\date{Nov 15, 2019 - EB}                                          

\author[label1]{Elena Braverman}
\author[label2]{Alexandra Rodkina}

\address[label1]{Dept. of Math.\& Stats., University of
Calgary,2500 University Drive N.W., Calgary, AB, Canada T2N 1N4; e-mail
maelena@ucalgary.ca, phone 1-(403)-220-3956, fax 1-(403)--282-5150 (corresponding author)}
\address[label2]{Dept. of Math., University of the West Indies,
Mona Campus, Kingston, Jamaica}

%\author[Elena Braverman and Alexandra Rodkina]
%{}

\begin{abstract}
Under natural assumptions, an unstable equilibrium of a difference equation  can  be stabilized 
by a bounded multiplicative noise, identically distributed at each step. This includes stabilization
of an otherwise unstable positive equilibrium of Ricker, logistic, and Beverton-Holt maps. 
Introduction of a multiplicative noise also allows to destabilize a stable equilibrium in a sense that all solutions 
stay away from this point, almost surely. 
In our examples a noise has symmetric, discrete or continuous,  distribution  with support $[-1,1]$, including  
Bernoulli and uniform continuous distribution.
We obtain conditions on the noise amplitudes in each case that allow to either stabilize or destabilize an equilibrium.  
Computer simulations illustrate our results.

%
%{\bf AMS Subject Classification:} 39A50, 37H10, 93D15, 39A30%
%
%{\bf Keywords:} stochastic difference equations; stochastic control; stabilization and destabilization; multiplicative noise; %Kolmogorov's Law of large numbers %
\end{abstract}


\end{frontmatter}

%\email{maelena@ucalgary.ca}
%\email{alexandra.rodkina@uwimona.edu.jm}

%\thanks{
%E. Braverman is a corresponding author. e-mail maelena@ucalgary.ca.
%The first author is supported by NSERC grant
%RGPIN-2015-05976}

%\begin{document}
%\maketitle

%\centerline{\scshape Elena Braverman}
%\medskip
%{\footnotesize
%% please put the address of the first author
% \centerline{Dept. of Math. and Stats., University of Calgary}
% \centerline{2500 University Drive N.W. Calgary, AB, Canada T2N 1N4}}
%
%\medskip
%
%\centerline{\scshape  Alexandra Rodkina}
%\medskip
%{\footnotesize
%% please put the address of the first author
% \centerline{Department of Mathematics,
%the University of the West Indies}
% \centerline{Mona Campus, Kingston, Jamaica}}

%%\medskip

%\bigskip

%\centerline{(Communicated by ...)}

%\centerline{\scshape Elena Braverman}
%\medskip
%{\footnotesize
% \centerline{Dept. of Math. and Stats., University of Calgary}
% \centerline{2500 University Drive N.W. Calgary, AB,  T2N 1N4, Canada}}

%\medskip

%\centerline{\scshape  Alexandra Rodkina}
%\medskip
%{\footnotesize
% \centerline{Department of Mathematics,
%the University of the West Indies}
% \centerline{Mona Campus, Kingston, Jamaica, India}}

%\maketitle

\section{Introduction}
\label{sec:intr}

Stabilizing effect of noise was observed already in the 1950ies, for instance, in the Kapica pendulum \cite{Kapica},
where a noise of certain amplitudes can stabilize an otherwise unstable equilibrium.
Theoretical justification for stabilizing an unstable equilibrium of a stochastic differential equation
by noise goes back to  \cite{Hasmin}, see also the recent monograph \cite{Khas}.
This was extended to linear systems \cite{Arnold} in 1983. For various types of  stochastic differential equations,
including equations with delays, stabilization results were further developed in \cite{AM,AMR1,Carab,Mao},
see also references therein.

In the present paper, we consider the possibility to stabilize an otherwise unstable equilibrium of a nonlinear difference equation by noise. For some earlier results on this topic see \cite{ABR,AMR06}, where the difference equation was a discretization of a corresponding differential equation.

Our research is inspired by population dynamics models of semelparous  populations which can be described by a 
difference equation
\begin{equation}
\label{eq:mainintr}
x_{n+1}=x_nf(x_n), \quad x_0>0, \quad n\in {\mathbb N}_0=\{0,1,2,\dots \},
\end{equation}
with a bounded nonnegative function $f$. %:\mathbb R\to [0, \infty)$. 
Here $x$ corresponds to the population density, $f(x)$ is a density-dependent per capita growth rate at the population density $x$, $n$ corresponds to the number of a season. 
As  particular cases, we consider Ricker, logistic and Maynard Smith (see \cite{Thieme}) equations.

Stochastic stabilization can be combined with deterministic control \cite{BKR,BR2c,BravRodk2,BR},
where stability effect is in fact achieved by non-stochastic methods, while sufficient conditions 
are established when noise introduction does not destroy stability. This includes the case of two-cyclic behaviour of the stabilized
system \cite{BR2c,BravRodk2} which is not in the framework of the present paper. 
However, in certain cases, stochastic perturbations
can eliminate the Allee effect for the non-perturbed system \cite{BRAllee}: unlike the original equation, whatever small initial value we choose, the solution converges to a blurred equilibrium, not to zero. This is stipulated by the form of the nonlinear function in the 
right-hand side of the equation and the ability of the stochastic perturbation to increase a solution with a positive probability.
The type of the function creates a trap for the stochastic solution: once an initial value is in a certain separated from zero interval, the solution stays there, independently of noise. A positive probability with which a solution increases by a positive value leads to a conclusion that it almost sure enters a trap and stays there \cite{BKR,BRAllee}. These ideas are applied in the present paper not to stabilization but to destabilization of an equilibrium. Also, compared to \cite{BKR,BR2c,BravRodk2,BR}, where mostly either an additive or a multiplicative noise is considered, here we apply more sophisticated noise forms.  

Since equation \eqref{eq:mainintr} is a population dynamics model, stabilization of the zero equilibrium with the help of the state dependent noise 
\begin{equation}
\label{def:addintr}
\sigma g(x_n) \xi_{n+1}
\end{equation}
can be treated as a stochastic control  leading to species eradication.  

Stabilization of a 
nonnegative equilibrium $K$ will be achieved with $g(x)=xf(x)-K$, while destabilization of 
the zero equilibrium with $g(x)=x$ and $\sigma=\sigma(x)$. 

By adding  the term of form \eqref{def:addintr} into the right-hand-side of \eqref{eq:mainintr}, we  stabilize or destabilize an equilibrium of \eqref{eq:mainintr}. 
This leads to the following  control equation 
\begin{equation}
\label{eq:addintr}
x_{n+1}=x_nf(x_n)+\sigma g(x_n) \xi_{n+1}, \quad n\in {\mathbb N}_0, \quad x_0>0,
\end{equation}
where   $\xi_{n}$ are independent identically distributed and $|\xi_n|\le 1$. In this paper, we deal only with three groups of distributions of $\xi$: (a) symmetric continuous (along with  uniform continuous) on $[-1, 1]$; (b) discrete uniformly distributed and  taking   $2l$ values (along with Bernoulli distributed for $l=1$); (c) piecewise continuous taking the value of $\frac{1}{2\delta(2l-1)}$ for $0<\delta<\frac{1}{2l-1}$ in a $\delta$-neighbourhood of each point $-1+\frac {2i}{2l-1}$, $i=0,  1, \dots, 2l-1$.

Overall, there are two main types of the noise applied:
with $g(x)=x$, see \cite{BKR,BR2c,BravRodk2,BRAllee,BR} which can be interpreted either as an additive 
noise applied to per capita growth rate, or as a perturbation {\em at the time of the reproduction event}, and 
with $g(x)=xf(x)$
which corresponds to a multiplicative noise, or a density-proportional perturbation {\em after the reproduction event} considered in the present paper.
For stabilization of the zero equilibrium  we use $g(x)=\sigma xf(x)$ and apply Kolmogorov's Law of Large Numbers. 
This approach goes back to H. Kesten, see \cite{FK} for the linear case 
and \cite {K} for convergence in probability. It was used in the proof of stability of the zero equilibrium for  
linear and nonlinear stochastic non-homogeneous equations  in \cite{BR0,BR1}, and for systems with  
square nonlinearities  in  \cite{Medv}.  

For any bounded function $f$ we prove that the  zero equilibrium of equation \eqref{eq:addintr} 
 is globally almost surely (a.s.) asymptotically stable  whenever
 \begin{equation}
\label{def:Eln1intr}
\eta:=-\mathbb E\ln |1+\sigma \xi_{n+1}| %>0,%\quad \forall n\in \mathbb N.
\end{equation}
is not only positive  but also big enough, $\eta >\ln H$, where $H$ is the maximum of $f$. The result holds for any distribution from group (b) and (c) with any $l\in \mathbb N$, where for (c) the value of $\delta$ should be  small enough, depending on $H$.  In the case of (a), the parameter of the continuous distribution has to be big enough, leading to concentration around the endpoints, depending again on $H$.

In \cite{AMR06} a special form of noise was designed to stabilize difference equations of a particular type, which can be treated as discretizations of differential equation. Note that  when dealing with bounded noises and constant step discretizations,  we can compare the result of  \cite{AMR06} with the present result on stabilization and  conclude that application of   Kolmogorov's Law of Large Numbers allows to simplify significantly the proofs and  remove restrictions on the bounds of function $f$, as well as some other assumptions, see  \cite{AMR06}.

Compared to the previous research, the present paper introduces the following novel elements.
\begin{enumerate}
\item
Stochastic {\em destabilization} is considered, to the best of our knowledge, for the first time. As simulations illustrate, this destabilization of an originally stable equilibrium leads to either almost all or a significant percent of solutions staying at a distance exceeding a perturbation, from this equilibrium.
\item
We have developed a library of explicit estimates for various types of perturbations, both discrete and continuous,
ensuring stabilization by noise.
\item
A wide class of both perturbations and perturbed equations of population dynamics is considered, without any limitations on a positive per capita production: monotonicity, concavity etc.
\end{enumerate}

However, there are also some limitations compared to previous studies on stochastic stabilization of difference equations: we consider bounded identically distributed perturbations, equations with a bounded per capita growth rate in the right-hand side, and some others. Nevertheless, these assumptions are satisfied for most common population dynamics models. Certainly, relaxing some of these restrictions can be a topic of future investigations. 

For destabilization of the zero equilibrium we use 
$g(x)=x$ and construct  a function $\sigma:\mathbb R\to [0, \infty)$, 
for an arbitrary bounded function $f$ and for random variables $\xi_n$ belonging to groups (a)-(c), 
%specially constructed $\sigma(x)$.%!!!$g(x)=x\sigma(x)$. 
%For an arbitrary bounded function $f$ and for random variables $\xi_n$ belonging to groups (a)-(c), we construct  a function $\sigma:\mathbb R\to [0, \infty)$ 
such that the conditional expectation satisfies
\begin{equation}
\label{def:destintr}
\mathbb E\left[ |f(x_n)+\sigma(x_n) \xi_{n+1}|^{-\alpha}\biggr| \mathcal F_n\right]<1, \quad \forall n\in  {\mathbb N}_0,
\end{equation}
where  $\alpha\in (0, 1]$, $\mathcal F_n$ is a $\sigma$-algebra generated by  $\{\xi_1,  \dots, \xi_n\}$, $x_n$ is a solution to equation \eqref{eq:addintr}.   Further, we prove that  $\displaystyle \mathbb P\left\{ \liminf_{n\to \infty}x_n>0\right\}=1$,  
for a solution  $x$ to  \eqref{eq:addintr} with any initial value $x_0>0$.  
However, to destabilize the equilibrium there is no need to add the term $x_n\sigma(x_n)\xi_{n+1}$ everywhere. When the function  $F(x):=xf(x)$ has a trap $[b, d]$, i.e. $F:[b, d]\to [b, d]$,  $0<b<d$, where $F(x)<x$ for $x \geq d$, we can truncate the noise term assuming that 
$\sigma(x)=0$ for $x\in (b, \infty)$. 
We prove  that a solution to \eqref{eq:addintr} with any initial value $x_0>0$ will reach the interval $[b, d]$ after a.s. finite number of steps and stay there.

The paper is organized as follows. In Section~\ref{sec:prelim} we introduce  all relevant definitions, assumptions and notations for equation~\eqref{eq:addintr}.  In Section \ref{sec:stab} we present  results on stabilization of the zero equilibrium of \eqref {eq:mainintr}, while in Section~\ref{sec:destab} we analyze destabilization of the zero equilibrium of \eqref {eq:mainintr}. Corresponding results  for non-zero equilibrium $K\neq 0$ are discussed  in 
Section \ref{sec:stabdestabK}. All  proofs are deferred to  the Appendix.
Examples of Ricker, logistic and Maynard Smith models along with  simulations are presented in Section~\ref{sec:sim}, 
while Section~\ref{sec:sum} contains discussion.

%%%%%%%%%%%%%%%%%%%%%%%%%%%
%%%%%%%%%%%%%%%%%%%%%%%%

\section{Preliminaries}
\label{sec:prelim}

\subsection{Assumptions}
\label{subsec:as}

Even though the form of equation \eqref{eq:mainintr} is inspired by population models, 
for convenience of our further  calculations, especially implementing a shift from the equilibrium $K >0$ to zero,  we 
assume that $f$ is defined on all $\mathbb R$ or its part and is bounded.  
\begin{assumption}
\label{as:fBound}
Assume that $f: \mathbb R\to  \mathbb R$ is a  uniformly bounded  piecewise continuous function: for some $H>0$,~
 $|f(x)| < H$, \quad $\forall x\in \mathbb R$.
\end{assumption}

In the present paper we stabilize or destabilize an equilibrium of equation \eqref{eq:mainintr}
by adding the state dependent noise term $\sigma g(x_n) \xi_{n+1}$ into the right-hand side,  
where $(\xi_n)_{n\in \mathbb N}$ is  a sequence of random variables. 
Thus deterministic equation \eqref{eq:mainintr} is  transformed into stochastic equation \eqref{eq:addintr}.

All stochastic sequences considered in the paper are supposed to be defined on a complete filtered probability space 
$(\Omega, {\mathcal{F}}$, $\{{\mathcal{F}}_n\}_{n \in \mathbb N}, {\mathbb P})$.
We use the standard abbreviation ``a.s." for either ``almost sure" or ``almost surely" 
with respect to a fixed probability measure $\mathbb P$  throughout the text. We use the standard abbreviation ``i.i.d.'' for  ``independent identically distributed", to describe random variables. 
A detailed discussion of stochastic concepts and notation can be found, for example, in \cite{Shiryaev96}. 

\begin{assumption}
\label{as:noise1}
Assume that  $(\xi_n)_{n\in \mathbb N}$ is  a sequence of i.i.d. random variables. 
\end{assumption}

Further, the filtration $\{{\mathcal{F}}_n\}_{n \in \mathbb N}$ is naturally generated by the sequence $(\xi_n)_{n\in 
\mathbb N}$, i.e. $\mathcal{F}_n$
is  a $\sigma$-algebra generated by $\{\xi_1, \dots, \xi_n\}$, for each $n\in \mathbb N$.

Equations in the present paper are motivated by populations  models,  we mainly deal with bounded~$\xi_n$, in 
particular, $|\xi_n|\le 1$ for all $n \in \mathbb N$ in all the applications. 
%\begin{assumption}
%\label{as:noise2}
%Assume that $(\xi_n)_{n\in \mathbb N}$ satisfy Assumption \ref{as:noise1} and, a.s.,  $|\xi_n|\le 1$ for all $n \in 
%\mathbb N$.
%\end{assumption}

%%%%%%%%%%%%%%

\subsection{Bounded noises and corollaries of the Borel-Cantelli Lemma }
\label{subsec:bnBC}

We are going to use the following distributions for bounded random variables $\xi_n$.
\begin{enumerate} 

\item [(a)] The  polynomial symmetrical continuously   distributed $\xi$  with the density $\phi_s$ defined by 
\begin{equation}
\label{def:xicontk}
\phi_s(x)=\frac{2s+1}2x^{2s},\quad x\in [-1, 1], ~~s \in {\mathbb N}_0.
\end{equation}
%Here $s\in \mathbb N_0=\mathbb N\cup \{0\}$. 
Note that for $s=0$, the density $\phi_0\equiv \frac{1}{2}$, i.e. it defines a continuous uniformly distributed on $[-1, 1]$ random variable.

\bigskip

\item [(b)]  The  discrete uniformly distributed on $[-1, 1]$ random variable $\xi$ with $2l$ states, $l\in \mathbb N$,  and the density function
\begin{equation}
\label{def:dud} 
\rho_l(x)=\left \{ 
  \begin{array}{cc}
\frac1{2l},& \quad  x=-1+\frac {2i}{2l-1}, \quad i=0, 1, \dots, 2l-1,
\\\\
  0,&  \mbox{\rm otherwise}.
  \end{array}
\right.
\end{equation}
Note that for $l=1$, the density $\rho_1$ defines  the Bernoulli random variable $\xi$.

\bigskip

\item [(c)] The piecewise continuous distribution on $[-1, 1]$ with $2l$ intervals, $l\in \mathbb N$, where the corresponding density function $\psi$ takes a nonzero constant value.  For each $l\in \mathbb N$ and $\delta\in \left(0,\frac 1{2l-1}\right)$, 
the density function  $\psi_{\delta, l}$ is defined as 
\begin{equation}
\label{def:contdud} 
\psi_{\delta, l}(x)=\left \{ 
  \begin{array}{cc}
 \frac1{2\delta (2l-1)},& \mbox{\rm for } x\in A_{l, \delta},\\
 0,& \quad  \mbox{\rm otherwise},
  \end{array}
\right.
\end{equation}
where
\begin{equation}
\label {def:Aldelta}
A_{l, \delta}:=\bigcup_{i=1}^{2l-2}\left[ -1+\frac {2i}{2l-1}-\delta,  -1+\frac {2i}{2l-1}+\delta\right]\bigcup[-1, -1+\delta]\bigcup [1-\delta, 1].
\end{equation}
\end{enumerate}
Obviously the function $\psi_{\delta, l}$ defined by \eqref{def:contdud} is a probability density function.

We assume $\cup_{i=j}^k S_i =\emptyset$ for $j>k$ and any sets $S_i$.

%\begin{lemma}
%\label{lem:defcontdud}
%Function $\psi_{\delta, l}$ defined by \eqref{def:contdud} is a density function.
%\end{lemma}

The Borel-Cantelli lemma (see, i.e. \cite{Shiryaev96}) is used in the proof of Lemma \ref{lem:topor}.  

\begin{lemma} \cite{Shiryaev96}
\label{lem:BC}
Let $A_1, \dots, A_n, \dots$ be a sequence of independent events. 

If ~~$~\displaystyle \sum_{i=1}^\infty \mathbb P\{A_i\}=\infty$ then~ $\mathbb P\{A_n \,\,  \text{\rm occurs infinitely often}\}=1$.
\end{lemma}

\begin{lemma} \cite{BKR,BRAllee}
\label{lem:topor}
Let  $(\xi_n)_{n\in\ \mathbb N}$ be a sequence of i.i.d. random variables,  
%a positive integer 
$J\in \mathbb N$, $([a_i, b_i])_{i=1, \dots, J}$ be a sequence of intervals 
%, $a_i <b_i$, $i=1, \dots, J$ 
such that $\mathbb P \left\{\xi_n\in [a_i, b_i]\right\}=~p_i>~0$, $i=1, \dots, J$. Then, for a random $\mathcal N$, 
$$ \displaystyle \mathbb{P}\left\{\text{$\exists$  %random
}\,  \mathcal N=
\mathcal N(J, a_i, b_i, i=1, \dots, J)<\infty \,:\,  \xi_{\mathcal N+i}\in [a_i, b_i], \, i=0, 1,  \dots, J\right\}=1.$$
\end{lemma}

The following corollaries of Lemma \ref{lem:topor} will be applied  in the proof of Theorem \ref{thm:destabtrunc} on destabilization with the help of a truncated noise. 

\begin{corollary}
\label{cor:s}
Let $(\xi_n)_{n\in \mathbb N}$ be a sequence of i.i.d. random variables with distribution  \eqref{def:xicontk}, \,  $\bigl((a_i, b_i)\bigr)_{i=1, \dots, J}$ be a sequence  of $J$ intervals each having a nonempty  intersection 
with the interval $[-1, 1]$. Then,  there exists an a.s. finite random variable  $\mathcal N=\mathcal N(J)\in \mathbb N$, such that
\[
\mathbb P\{\xi_{\mathcal N+i}\in (a_i, b_i), \quad  i=1, 2,  \dots, J\}=1.
\]
\end{corollary}

\begin{corollary}
\label{cor:2l}
Let $(\xi_n)_{n\in \mathbb N}$ be i.i.d. random variables with distribution  
\eqref{def:dud}, and $(a_i)_{i=1, \dots, J}$ be a  sequence 
with values from  the set $ \displaystyle  \left\{ -1+\frac {2k}{2l-1}, k=0, 1, \dots, 2l-1 \right\}$.
 Then,  there exists an a.s. finite random variable  $\mathcal N=\mathcal N(J) \in \mathbb N$, such that
\[
\mathbb P\{\xi_{\mathcal N+i}=a_i, \quad  i=1, 2,  \dots, J\}=1.
\]
\end{corollary}

\begin{remark}
\label{rem:distr23}
A statement similar to Corollary \ref{cor:s} can be formulated for distribution \eqref{def:contdud}, when  each interval $(a_i, b_i)$ has a nonempty intersection with the set $A_{l,\delta}$ defined by \eqref{def:Aldelta}.
\end{remark}

%%%%%%%%%%%%%%%%%%%%%%%%
\subsection{Martingales and convergence theory}\label{sec:mart}
We recall the following definitions.
\begin{definition}

A stochastic sequence $(M_n)_{n \in \mathbb N}$  is said to be an 
$\mathcal{F}_n$-martingale if ${\mathbb E}|M_n|<\infty$ and $\mathbb E\left[M_n |
\mathcal{F}_{n-1}\right]=M_{n-1}$ for all $n\in\mathbb N$ {\it a.s.}

A stochastic sequence $(\mu_n)_{n \in \mathbb N}$ is said to be an 
$\mathcal{F}_n$-martingale-difference  if $\mathbb E |\mu_n|~<~\infty$ and
$\mathbb E\left[\mu_n | \mathcal{F}_{n-1}\right]=0$ {\it a.s.} for all
$n\in\mathbb N$.
\end{definition}
The following construction can be found in \cite{ABR,KPR} and will be a key to proofs of results on destabilization.

\begin{lemma} \cite{KPR}
\label{lem:martprod}
Let $(Y_i)_{i\in\mathbb{N}}$ be a sequence of non-negative random variables adapted to the filtration $\{\mathcal{F}_n\}_{n\in\mathbb{N}}$, where each $Y_i$ satisfies 
\begin{enumerate}
\item[i)] $\mathbb{E}[Y_i]<\infty$;
\item[ii)] $\mathbb{E}[Y_i|\mathcal{F}_{i-1}]=1$. 
\end{enumerate}
Then the sequence $\{M_n\}_{n\in\mathbb{N}}$ given by ~
$\displaystyle
M_n=\prod_{i=1}^{n}Y_i,~~ n\in\mathbb{N},
$
is an $\mathcal{F}_n$-martingale.
\end{lemma}
We now present two convergence results required for the analysis in this article. The first one is a classical result on the convergence of non-negative martingales, which may be found, for example, in \cite[p. 508]{Shiryaev96}.
\begin{lemma}\label{lem:nnM} \cite{Shiryaev96}
If $(M_n)_{n\in\mathbb{N}}$ is a non-negative $\mathcal{F}_{n}$-martingale then $\lim\limits_{n\to\infty}M_n$ exists and is finite with probability one.
\end{lemma}
The second one  is the Kolmogorov's Law of Large Numbers, see \cite[p. 391]{Shiryaev96}.
\begin{lemma} \cite{Shiryaev96}
\label{thm:Kolm}
  Let $(v_{n})_{n\in\ \mathbb N}$ be a sequence of independent identically distributed random
  variables with $\mu:=\mathbf E|v_n|<\infty$.   Then ~
  %\begin{equation*}
 $\displaystyle   \frac{S_n}{n} \rightarrow \mu$ as $n \to \infty$, a.s.
    %\label{Kolm1}
  %\end{equation*}
\end{lemma}

%%%%%%%%%%%%%%%%%%%%%%%%

\section{Stabilization of the zero equilibrium}
\label{sec:stab}
Let $g(x)=xf(x)$ in the right hand side of \eqref {eq:addintr}, then we have
\begin{equation}
\label{eq:addstab}
x_{n+1}=(1+\sigma \xi_{n+1} )x_nf(x_n), \quad x_0>0, \quad n\in {\mathbb N}_0.
\end{equation}
We show that for any bounded function $f$, there exists a sequence of random variables $\xi_n$ satisfying Assumption~\ref{as:noise1}
%-\ref{as:noise2} 
and $\sigma\in (0, 1]$, such that the zero equilibrium of equation \eqref{eq:addstab} is a.s. globally asymptotically stable. The random variables $\xi_n$ can be chosen either discrete or continuous.

First we formulate  the general result on stabilization of the zero equilibrium of equation \eqref{eq:mainintr}.  
We assume that  $1+\sigma \xi_{i+1}\neq 0$  a.s., $\mathbb E\ln |1+\sigma \xi_{i+1}|$  is finite (see more details in Remark \ref{rem:welldefEln}) and 
  \begin{equation}
\label{def:Eln1}
\eta:=-\mathbb E\ln |1+\sigma \xi_{i+1}|>0.
\end{equation}

\begin{lemma}
\label{lem:genstab0}
Let Assumptions \ref{as:fBound} and  \ref{as:noise1} hold, 
$\mathbb E\ln |1+\sigma \xi_{i+1}|$  satisfy \eqref{def:Eln1} and
\begin{equation}
\label{est:H}
\ln H<\eta .
\end{equation}
Then, for a solution  $x$ to  \eqref{eq:addstab} with any initial value $x_0>0$, 
~ $\displaystyle \lim_{n\to \infty}x_n=0$ a.s.
\end{lemma}

\begin{remark}
\label {rem:welldefEln}
The value $\mathbb E\ln |1+\sigma \xi_{i+1}|$ is well defined for each $\sigma$ if $\xi$ is continuously distributed. In the next section, when $\xi$ is a discrete random variable which takes the value of $-1$ with a nonzero probability, we assume that $\sigma < 1$, which also guarantees that $\mathbb E\ln |1+\sigma \xi_{i+1}|$  is finite.
\end{remark}

\begin{remark}
\label {rem:Hlambda}
Without loss of generality, we can assume that the upper bound $H$ of $f$ from Assumption \ref{as:fBound} is greater than 1, 
i.e.  $H>1$. 
If $H<1$, we do not need any stabilization for original equation \eqref{eq:mainintr}, since the zero equilibrium is already globally asymptotically stable.
\end{remark}

The next lemma shows that each of distributions defined by \eqref{def:xicontk}, \eqref{def:dud} and \eqref{def:contdud} can stabilize the zero equilibrium of \eqref{eq:addintr} if $\sigma$ is chosen appropriately.

\begin{lemma}
\label{lem:Eln123}
Let Assumption \ref{as:fBound} hold.  
\begin{enumerate}
\item [(i)]  Let $\xi$ be   defined by \eqref{def:xicontk} and $\sigma=1$. Then,  for each $H>0$, there exists  $s\in \mathbb N_0$ from \eqref{def:xicontk} such that  \eqref{est:H} holds.
\item  [(ii)] Let $\xi$  be   defined by either \eqref{def:dud} or \eqref{def:contdud}. Then, for each $H>0$, there is a $\sigma\in (0,1)$ close enough to 1, such that  \eqref{est:H} holds. 
In the case of \eqref{def:contdud}, $\delta>0$ should also be small enough.
\end{enumerate}
 \end{lemma}
The following theorem is the   main result of this section. It  is a corollary of Lemmata \ref {lem:genstab0} and 
\ref{lem:Eln123}.
\begin{theorem}
\label{thm:stab123}
Let Assumptions \ref{as:fBound} and  \ref{as:noise1} hold, $x$ be a solution to  \eqref{eq:addstab} with any initial value $x_0>0$, and  one of the following conditions hold: 
\begin{enumerate}
\item [(i)]  $\xi$ is  defined by \eqref{def:xicontk}, $\sigma=1$ and $s$ from \eqref{def:xicontk} is big enough;
\item  [(ii)] $\xi$ is  defined by either \eqref{def:dud} or \eqref{def:contdud}, $\sigma$ is close enough to 1,
and, in the case of \eqref{def:contdud}, $\delta>0$ is small enough. 
\end{enumerate}
Then $\lim\limits_{n\to \infty}x_n=0$ a.s.
\end{theorem}

 %%%%%%%%%%%%%%

\section{Destabilization of the zero equilibrium}
\label{sec:destab}

In this section we destabilize the zero equilibrium of equation \eqref{eq:mainintr} using a stochastic perturbation of the type
\begin{equation}
\label{eq:sdnoise}
x_{n+1}=x_n f(x_n)+\sigma(x_n) x_n \xi_{n+1}, \quad x_0>0, \quad n\in {\mathbb N}_0.
\end{equation} 
Here $\xi_n$ satisfies Assumption \ref{as:noise1} and $\sigma:\mathbb R\to [0, \infty)$. 
The function $\sigma$ is chosen for each  $f$ in a way which guarantees that  solution $x_n$ to equation \eqref{eq:sdnoise} with any initial value $x_0>0$ does not converge to zero with probability 1. 
 
 A general result on destabilization of the zero equilibrium is given below.
 \begin{lemma}
\label{lem:destab}
Suppose that Assumption \ref{as:fBound} holds, and $x$ is a solution to equation  \eqref{eq:sdnoise}. 
Let $\sigma:\mathbb R\to [0, \infty)$ and random variables $\xi_n$ be such that, for some $\alpha\in (0, 1]$  and  for each  $n\in {\mathbb N}_0$, a.s., $f(x_n)+\sigma(x_n) \xi_{n+1}\neq 0$ and 
\begin{equation}
\label{cond:destab22}
\mathbb E\left[ |f(x_n)+\sigma(x_n) \xi_{n+1}|^{-\alpha}\biggr| \mathcal F_n\right]<1.
\end{equation}
Then $\mathbb P\left\{\liminf\limits_{n\to \infty}|x_n|>0\right\}=1$.
\end{lemma} 

Now we show  that each of the distributions defined by \eqref{def:xicontk}, \eqref{def:dud} and \eqref{def:contdud} 
can destabilize the equilibrium of %!!!!!!!!!!!! Lena - changed reference !!!!\eqref{eq:addintr}
\eqref{eq:sdnoise}
 when parameters of the distributions and $\sigma$ are  chosen appropriately and 
\begin{equation}
\label{est:sigmadest}
\sigma(x)>f(x), \quad \forall x\in \mathbb R.
\end{equation}
To this end, we prove that assumptions of Lemma \ref{lem:destab} are fulfilled for $\alpha\in (0, 1)$ 
in case  \eqref{def:xicontk} and $\alpha=1$ for \eqref{def:dud}, \eqref{def:contdud}. 
In all these cases  we have, a.s.,  $f(x_n)+\sigma(x_n) \xi_{n+1}\neq 0$, $\forall n\in \mathbb N$.

\begin{theorem}
%\begin{lemma}
\label{lem:Ealpha123}
Let Assumptions \ref{as:fBound} and \ref{as:noise1} hold.  
\begin{enumerate}
\item [(i)]  
Let $\xi$ be   defined by \eqref{def:xicontk} with any parameter  $s\in \mathbb N$. Then   there exists $\sigma:\mathbb R\to [0, \infty)$  such that
condition \eqref{cond:destab22} holds for some $\alpha=\alpha(s)\in (0, 1)$.

\item  [(ii)] Let $\alpha=1$ and $\xi$  be   defined by either \eqref{def:dud} or \eqref{def:contdud}. 
Then for any $l\in \mathbb N$  there exist $\sigma:\mathbb R\to [0, \infty)$ and, for \eqref{def:contdud}, $\delta>0$ such that 
condition \eqref{cond:destab22} holds. 
\end{enumerate}
 %\end{lemma}
\end{theorem} 

\begin{remark}
\label{rem:cond2122}
It will be shown in  the proof of Theorem~\ref{lem:Ealpha123} (see Appendix) that, under natural assumptions,
%$\sigma:\mathbb R\to [0, \infty)$    
for \eqref{def:dud} 
\begin{equation}
\label{ineq:sdud}
\sigma(x)> (2l-1)\left[\frac{1+\sqrt{1+4 f^2(x)}}{2}\right], \quad \forall x\in \mathbb R, 
\end{equation}
and for  \eqref{def:xicontk}  (with any $s\in \mathbb N$)
\begin{equation}
\label {cond:fe} 
\sigma(x)>\max \left\{f(x), e\right\}, \quad \forall x\in \mathbb R,
\end{equation}
imply condition \eqref{cond:destab22}. For \eqref{def:xicontk}, the function $\sigma(x)$ should be chosen continuous.
\end{remark}

To destabilize the equilibrium, there is no need to add the noise term to the equation everywhere, it is enough to apply a perturbation only
when the solution is in some  neighbourhood of the equilibrium.  We discuss  the  truncated version $\sigma_b$ of $\sigma$, 
which vanishes in %some right-hand side  neighborhood 
$(b, \infty)$. %of zero. 
Under some additional assumptions, we show that  a solution $x_n$ to \eqref{eq:addintr} with any initial value $x_0>0$, after a.s. finite number of steps, satisfies $x_n\ge b$.

Let  $\sigma$ be a noise coefficient constructed  in Lemma  \ref{lem:Ealpha123} and $b>0$, see \eqref{ineq:sdud} and \eqref{cond:fe}. Instead of applying  $\sigma(x)$  for each $x$,  we use a truncated coefficient
\begin{equation}
\label{def:modsigmal}
\sigma_{b}(x):=\left \{ 
  \begin{array}{cc}
 \sigma(x) ,& \quad x < b,\\\\
  0, & \quad x\ge b.
  \end{array}
\right.
\end{equation}

We will need an additional restriction for $f$.
Define 
\begin{equation}
\label{def:Fm}
\mbox{$F(x)=xf(x)$, for all $x\in \mathbb R$.}
\end{equation}
\begin{assumption}
\label{as:Fbd}
Let  Assumption \ref{as:fBound} hold, $F$  and  there exist $b, d$, $0<b<d$ such that
\begin{enumerate}
\item [(i)] $f(x)>0$ for $x\in(0, \infty)$  and $f$ is continuous on $[-b,0)\cup (0,b]$; %\\ %everywhere in $[-b, b]$ apart from maybe $x=0$;%\\
\item [(ii)] $F:[b, d]\to [b, d]$, where $F$ be defined by  \eqref{def:Fm}.
\end{enumerate}
\end{assumption}

Instead of equation \eqref{eq:sdnoise}, we consider 
\begin{equation}
\label{eq:sdnoiseb-b}
x_{n+1}=\max\left\{x_n f(x_n)+\sigma_b(x_n) x_n \xi_{n+1}, \,\, -b\right\}, \quad x_0>0, \quad  n\in {\mathbb N}_0.
\end{equation} 

\begin{theorem}
\label {thm:destabtrunc} 
Let Assumptions %\ref{as:fBound},  
\ref{as:noise1},\ref{as:Fbd} and \eqref  {est:sigmadest} hold, as well as one 
of the following conditions:

\begin{enumerate}
\item [(a)] $\xi$ is defined by \eqref{def:dud}, for some  $l\in \mathbb N$, and $\sigma$ satisfies  \eqref{ineq:sdud}; 

\item  [(b)] $\xi$ is defined by \eqref{def:xicontk}, for some  $s\in \mathbb N$, and $\sigma$ satisfies \eqref{cond:fe}.
\end{enumerate}
Let  also
\begin{equation}
\label{cond:Fbd}
d>\sup_{x\in [-b, b]} x\left[f(x)+ 
\sigma(x)\right]. 
\end{equation}
Then a solution of \eqref{eq:sdnoiseb-b} with any positive initial value $x_0\in (0, d)$ eventually reaches $[b, d]$  after an a.s. finite number of steps and stays there.
\end{theorem}

\begin{remark} 
\label{rem:delta}
We can prove the result of Theorem \ref{thm:destabtrunc}  when $\xi$ has distribution \eqref{def:contdud}, but, to save   space, we do not include it in the paper.
\end{remark}

%%%%%%%%%%%%%%%%%%

\section{Stabilization and destabilization of a positive equilibrium $K$}
\label {sec:stabdestabK}

%%%%%%%%%%%%%%%%

\subsection{Assumptions and shift}
\label{subsec:Kto0}

In this section we assume that,  in addition to zero, equation \eqref{eq:mainintr} also has a positive equilibrium $K>0$. 
To deal with this case, we need an additional assumption on $f$.
\begin{assumption}
\label{as:fK}
Assume that $f$ satisfies Assumption \ref{as:fBound}, $f(K)=1$, $f$ has a derivative $f'(K)$ at $K$. 
\end {assumption}

By Assumption~\ref{as:fK}, the function ${\rm f}:\mathbb R\to \mathbb R$ defined as

\begin{equation}
\label{def:rmf1}
{\rm f}(u)=\left\{\begin{array}{cc} {\displaystyle \frac{(u+K)f(u+K)-K}{u}},&  u\neq 0,\\\\
  Kf'(K)+1, &  \mbox{for} \quad u=0,
  \end{array}
  \right.
  \end{equation}
is continuous at $u=0$, since $f(K)=1$ and ~~
$\displaystyle
 \frac{(u+K)f(u+K)-K}{u}- Kf'(K) -1$ \\ $\displaystyle =f(u+K)-f(K)+K\left[ \frac{f(u+K)-f(K)}{u} -f'(K) \right] \to 0$
%tends to zero 
as $u\to 0$. 
Also, $\rm f$ is bounded: $|{\rm f}(x)| < {\mathcal H}$ for $x \in \mathbb R$, where 
\begin{equation}
\label{def:calH}
\mathcal H \le \max \left\{ \max_{|x-K|\leq 1} \left| \frac{xf(x)-K}{x-K} \right|, H+K(H+1) \right\}.
\end{equation}
In fact, for $|x-K|=|u|>1$, 
$$
\left| \frac{(u+K)f(u+K)-K}{u} \right|  \leq |f(u+K)| +K |f(u+K)-1| \leq H+K(H+1),
$$
which leads to the estimate in \eqref{def:calH}.

\begin{remark}
\label{rem:K}
Estimate \eqref{def:calH} is not optimal, for %particular cases of 
some $f$ sharper estimates of $\displaystyle \sup_{u\in (-K, \infty)}|{\rm f}(u)|$ can be obtained, 
due to  $H+K(H+1)$ in the bound. For the Ricker map with $f(x)=e^{r(1-x)}$ we have $H=1$, with $r=1$ the actual bound is 
$\mathcal H=1$, while  \eqref{def:calH} leads to $\mathcal H \leq 2H+1=3$. For $r=3$, we have %$H>e^2/3 \approx 2.463$, and 
the actual bound $\mathcal H \approx 2.4925$, which is less than $2H+1=3$.
\end{remark}

Thus, instead of dealing with the equilibrium $K$ for equation \eqref{eq:mainintr}, we stabilize (or destabilize) for $z_n=x_n-K$
the zero equilibrium of the equation 
\begin{equation}
\label{eq:oeqv}
z_{n+1}=z_n\rm f(z_n), \quad n\in {\mathbb N}_0.
\end{equation}
%%%%%%%%%%%%%%%%%%%
\subsection{Stabilization and destabilization}
\label{subsec:stabK1}

We start with stabilization and assume that  the equilibrium $K$ of \eqref{eq:mainintr} (and therefore 
the zero equilibrium of \eqref{eq:oeqv})  is unstable. 

%First, assume that $F'(K)=1+Kf'(K)<-1$, or that the equilibrium $K$ of \eqref{eq:mainintr}, as well as 
%the zero equilibrium of \eqref{eq:oeqv}, is unstable. 

We can stabilize the zero equilibrium of \eqref{eq:oeqv}
multiplying the right-hand side by $1+\sigma\xi_{n+1}$, which leads to the equation 
\begin{equation}
\label{eq:addstabK}
z_{n+1}=(1+\sigma\xi_{n+1})z_n\rm f(z_n), \quad z_0=x_0-K, ~~n\in {\mathbb N}_0.
\end{equation}
Returning to $f$ and $x_n$, with $x_n=z_n+K$, we arrive at
\begin{equation}
\label{eq:stabK}
x_{n+1}=x_nf(x_n)+\sigma \left[x_nf(x_n)-K \right]\xi_{n+1}, ~~n\in {\mathbb N}_0.
\end{equation}
\begin{remark}
\label{rem:Krmf}
Equation \eqref{eq:stabK} is a special case of \eqref{eq:addintr} with $g(x)=xf(x)-K$.
\end{remark}

Lemma \ref{lem:genstab0} and  Theorem \ref{thm:stab123} immediately lead to the following results.
\begin{lemma}
\label{cor:Kstabgen}
Let Assumptions  \ref{as:noise1} and \ref{as:fK} hold, $\mathcal H$ be defined by \eqref{def:calH}, and there exist  $\sigma>0$ such that
\begin{equation}
\label{cond:stabK}
\ln \mathcal H<-\mathbb E \ln |1+\sigma \xi_{n+1}|.
\end{equation}
Let $x$ be a solution to %equation 
\eqref{eq:stabK} with $\sigma$  and $\xi_n$ satisfying  \eqref{cond:stabK}. 
Then,  $\lim\limits_{n\to \infty}x_n~=~K$ a.s.
\end{lemma}

\begin{theorem}
\label{cor:Kstab123}
Let Assumptions %\ref{as:fBound}, 
\ref{as:noise1}  and \ref{as:fK}  hold, $x$ be a solution to  \eqref{eq:stabK} with any initial value $x_0>0$, 
and  one of the following conditions hold: 
\begin{enumerate}
\item [(i)]  $\xi$ is  defined by \eqref{def:xicontk}, $\sigma=1$ and $s$ from \eqref{def:xicontk} is big enough;
\item  [(ii)] $\xi$ is  defined by either \eqref{def:dud} or \eqref{def:contdud}, $\sigma$ is close enough to 1,
and, in the case of \eqref{def:contdud}, $\delta>0$ is small enough. 
\end{enumerate}
Then $\lim\limits_{n\to \infty}x_n=K$ a.s.
\end{theorem} 

Further, following the approach of Sections  \ref{sec:destab} and \ref{subsec:Kto0} for  the function ${\rm f}$ defined by \eqref{def:rmf1} and anyone of  distributions \eqref{def:xicontk}, \eqref{def:dud} or \eqref{def:contdud}, we construct the function $\sigma(z)$.  The  control equation which destabilizes a stable equilibrium $K$ with a stochastic perturbation is 
\begin{equation}
\label{eq:destabK_1}
z_{n+1}=z_n \left[{\rm f}(z_n) + \sigma (z_n) \xi_{n+1} \right], ~~n\in {\mathbb N}_0,
\end{equation}
where ${\rm f}$ is defined by \eqref{def:rmf1}.

Returning to $f$ and $x_n$, with $x_n=z_n+K$, we arrive at ~
$
x_{n+1}-K$ \\$=(x_n-K)({\rm f}(x_n-K)+\sigma (x_n-K)\xi_{n+1})=x_nf(x_n)-K+(x_n-K)\sigma (x_n-K)\xi_{n+1}$,
or 
\begin{equation}
\label{eq:destabK}
x_{n+1}=x_nf(x_n)+ (x_n-K)\sigma (x_n-K)\xi_{n+1}, ~~n\in {\mathbb N}_0.
\end{equation}
Similarly to stabilization, we can formulate results on destabilization of the equilibrium $K$ for equation \eqref{eq:destabK}.
In particular, Lemma~\ref{lem:destab} implies the following result.

 \begin{lemma}
\label{lem:destab_K}
Let Assumptions \ref{as:noise1} and \ref{as:fK} hold, $x$ be a solution to equation  \eqref{eq:destabK}, 
$\sigma:\mathbb R\to [0, \infty)$ and random variables $\xi_n$ be such that, for some $\alpha\in (0, 1]$  and  for each  $n\in \mathbb N$, a.s., $
{\rm f}(x_n-K)+\sigma(x_n-K) \xi_{n+1}\neq 0$  and 
\begin{equation}
\label{cond:destab_22K}
\mathbb E\left[ |{\rm f}(x_n-K)+\sigma(x_n-K) \xi_{n+1}|^{-\alpha}\biggr| \mathcal F_n\right]<1.
\end{equation}
Then $\displaystyle \mathbb P\left\{\liminf_{n\to \infty}|x_n-K|>0\right\}=1$.
\end{lemma} 

The results of Lemma~\ref{lem:Ealpha123} can be applied for $\rm f$ without any changes.

%%%%%%%%%%%%%%
\section{Examples and simulations}
\label{sec:sim}

For stabilization of zero, we consider the equation
\[
x_{n+1}=(1+\sigma \xi_{n+1} )F(x_n)= (1+\sigma \xi_{n+1} )x_nf(x_n).
\]
\begin{example}
\label{ex:forlena4}
\be
\item [(a)] Consider continuous distribution \eqref{def:xicontk} of $\xi$ % on $[-1, 1]$ 
for Ricker's model. Here $f(x)=e^{r(1-x)}$, $H=\max{f(x)}=e^r$. In order to get global stability, we should have 
\[
He^{\mathbb E\ln (1+\xi_{i+1})}<1, \quad
% e^r\le e^{-\mathbb E\ln (1+\xi_{i+1})}, \quad 
r\le -\mathbb E\ln (1+\xi_{i+1})=\frac{1}{2s+1}+\frac{1}{2s-1}+\dots +1-\ln 2.
\]
For $s=3$, the bound for $r$ is 
$\displaystyle 
\left| \frac{1}7+\frac 15 +\frac 13+1-\ln 2 \right| \approx 
%|0.6931-1.6761| \approx 
0.983
$, see Fig.~\ref{figure26}.

\begin{figure}[ht]
\centering
\vspace{-25mm}
\includegraphics[height=.3\textheight]{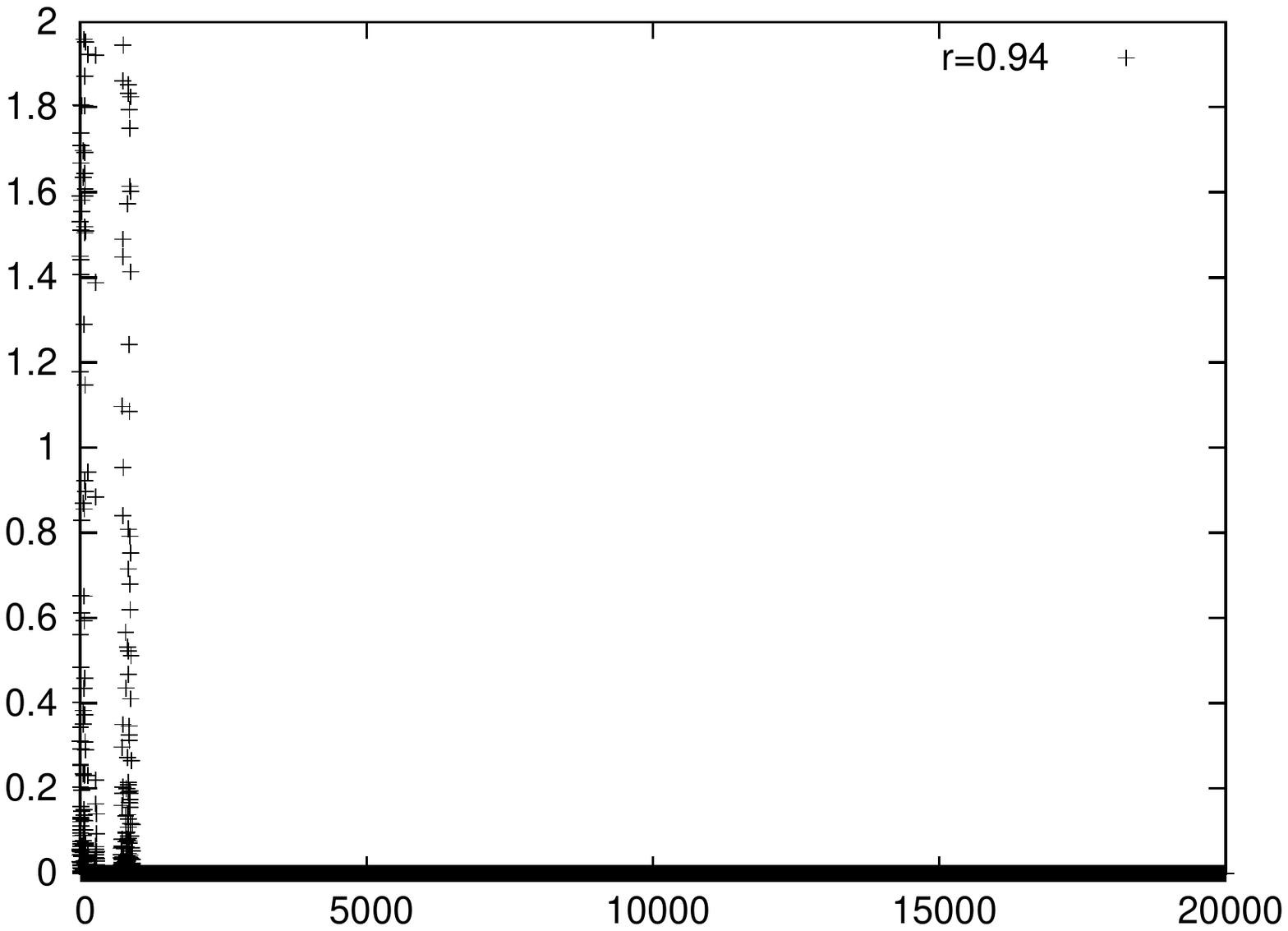}
\includegraphics[height=.3\textheight]{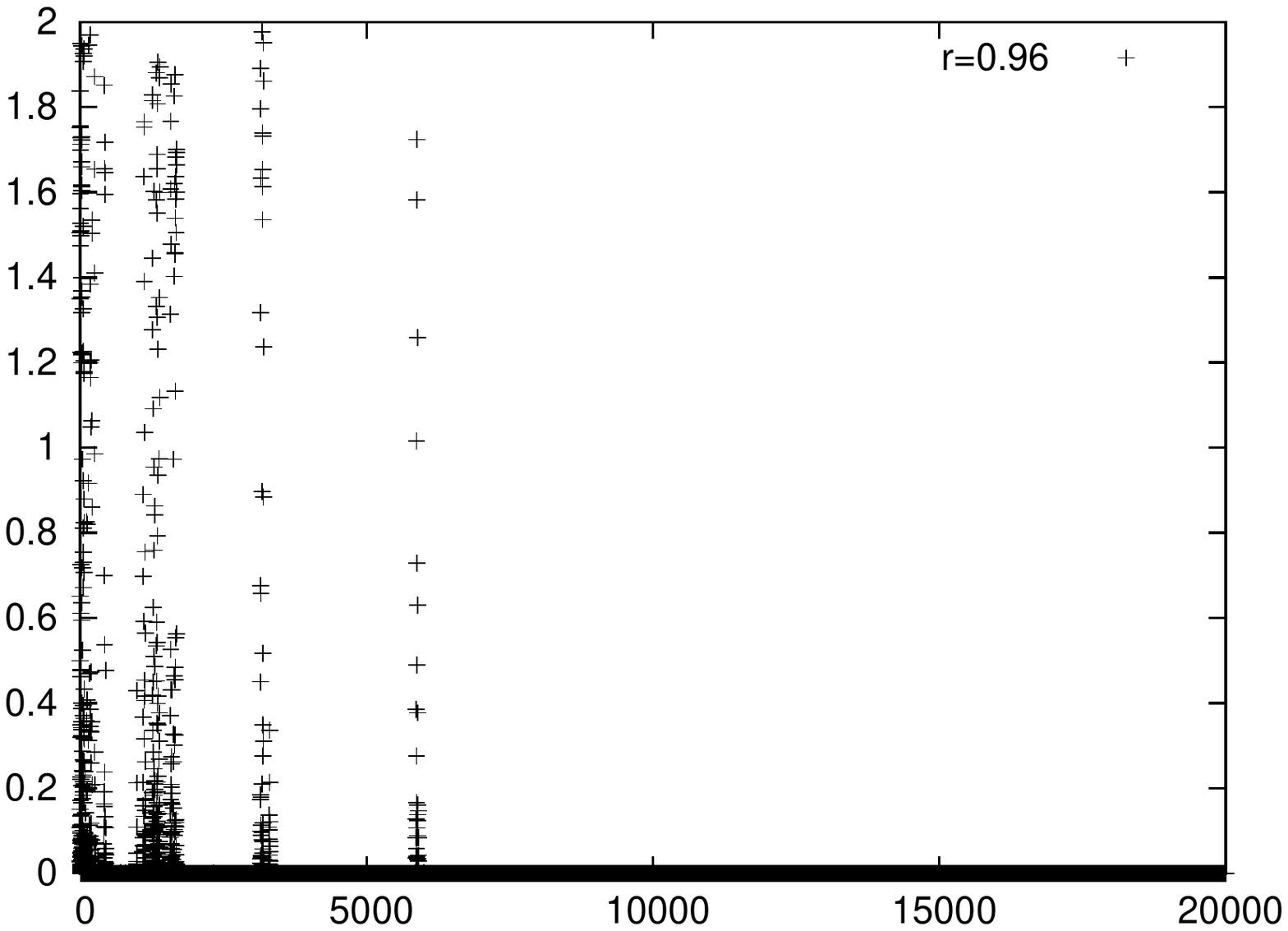}
\includegraphics[height=.3\textheight]{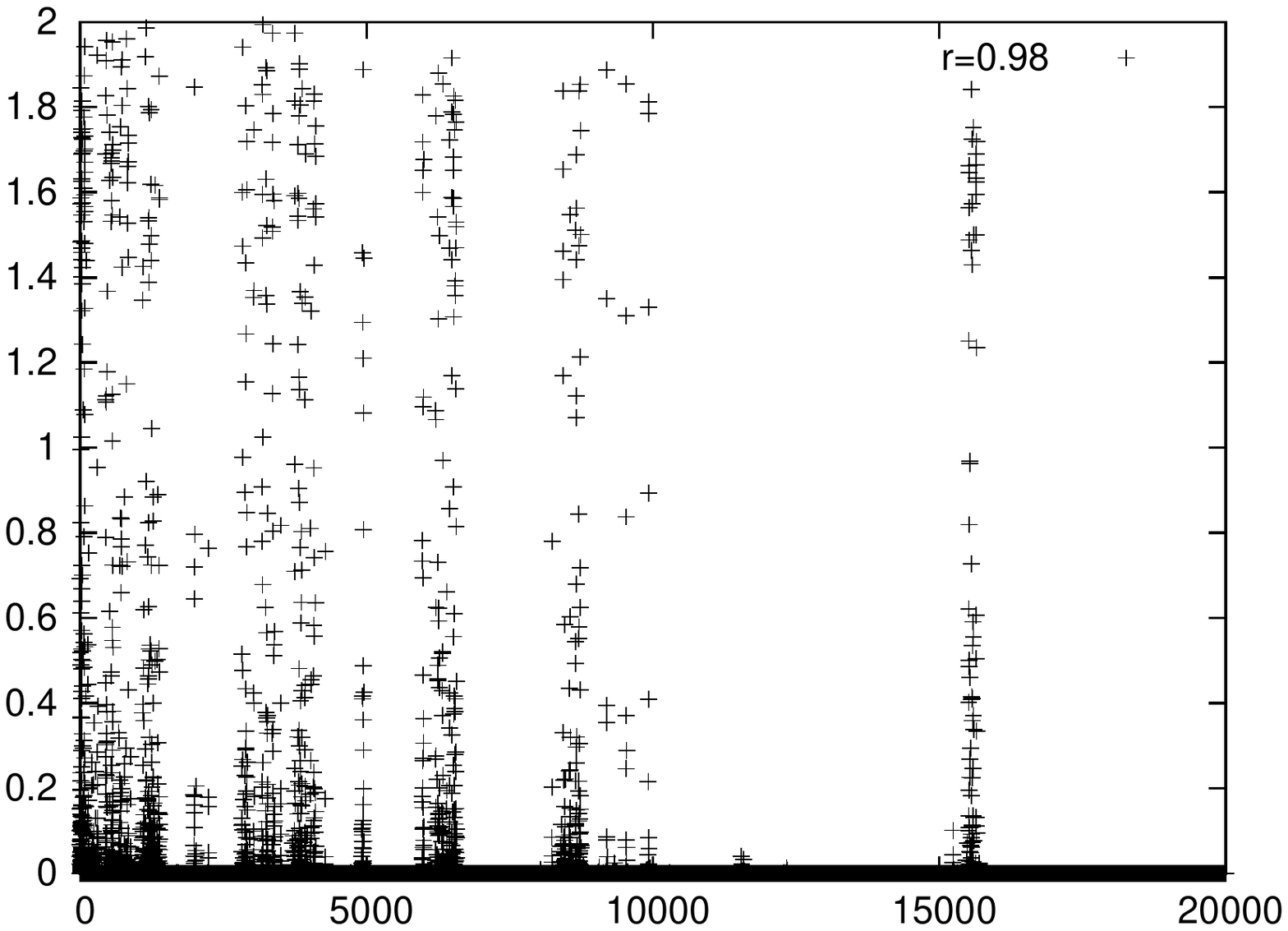}
\caption{Five runs for the Ricker equation and noise \eqref{def:xicontk} with  $s=3$, $\sigma=1$, $x_0=0.5$ and (from left to right) 
$r=0.94,0.96,0.98$.}
\label{figure26}
\end{figure}

\item [(b)] Consider discrete uniform distribution \eqref{def:dud}. The zero equilibrium of the logistic map for $r=2$ is 
unstable but can be stabilized for  $\sigma=1$,
however, in Fig.~\ref{figure4} we explore how for a smaller $\sigma=0.865$ stabilization depends on the choice of $l$.

\begin{figure}[ht]
\centering
\vspace{-25mm}
\includegraphics[height=.3\textheight]{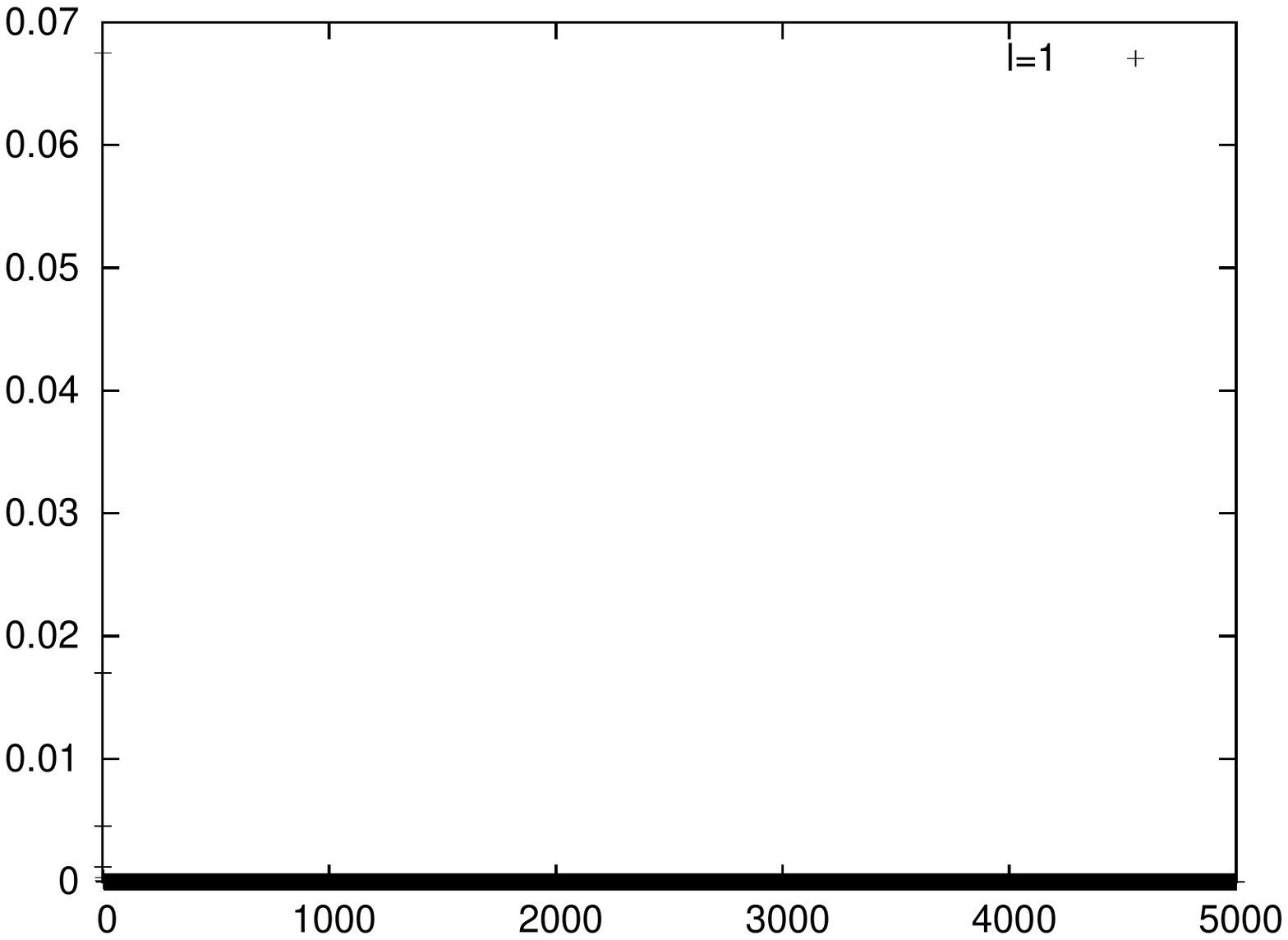}
\includegraphics[height=.3\textheight]{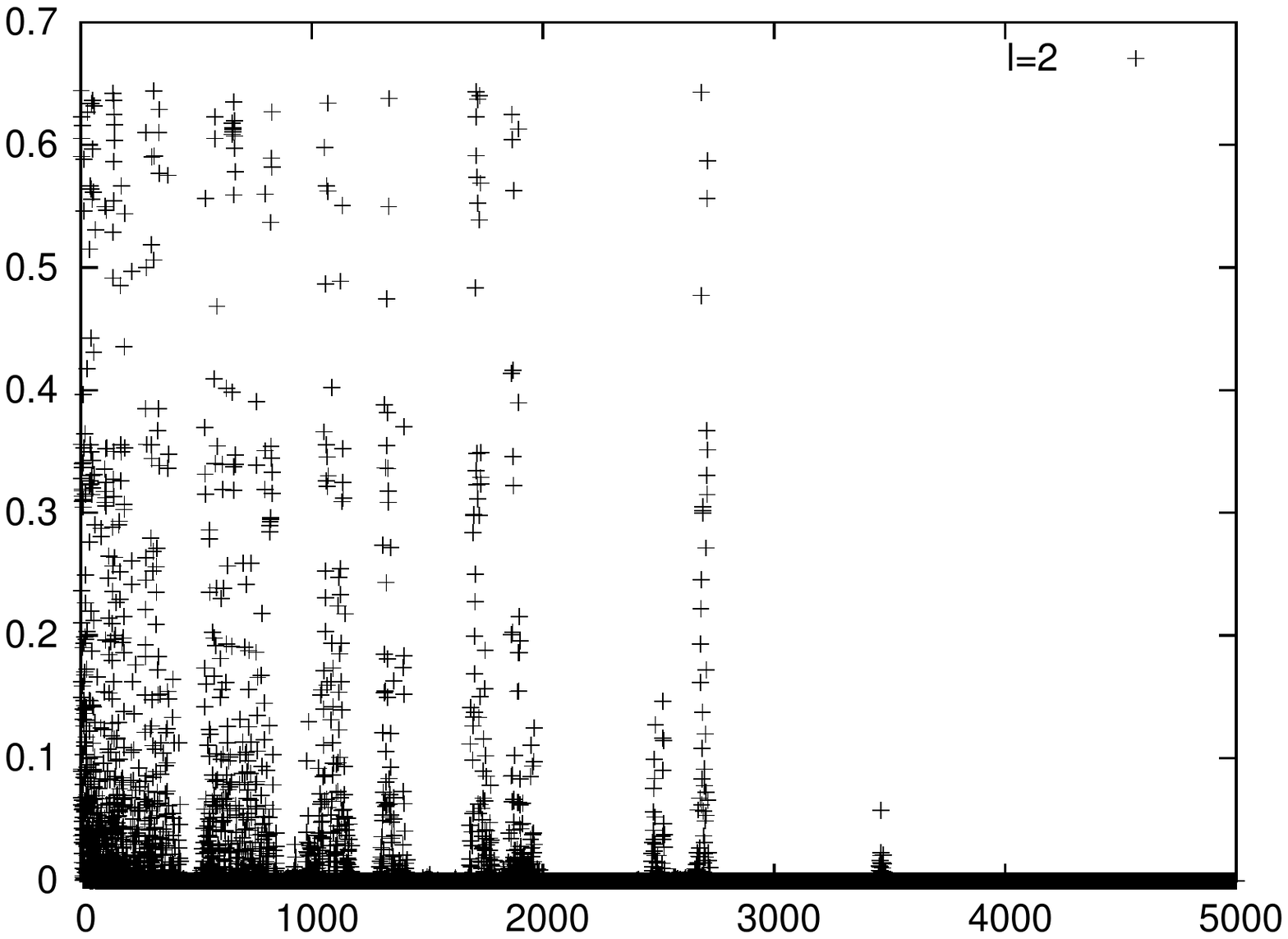}
\includegraphics[height=.3\textheight]{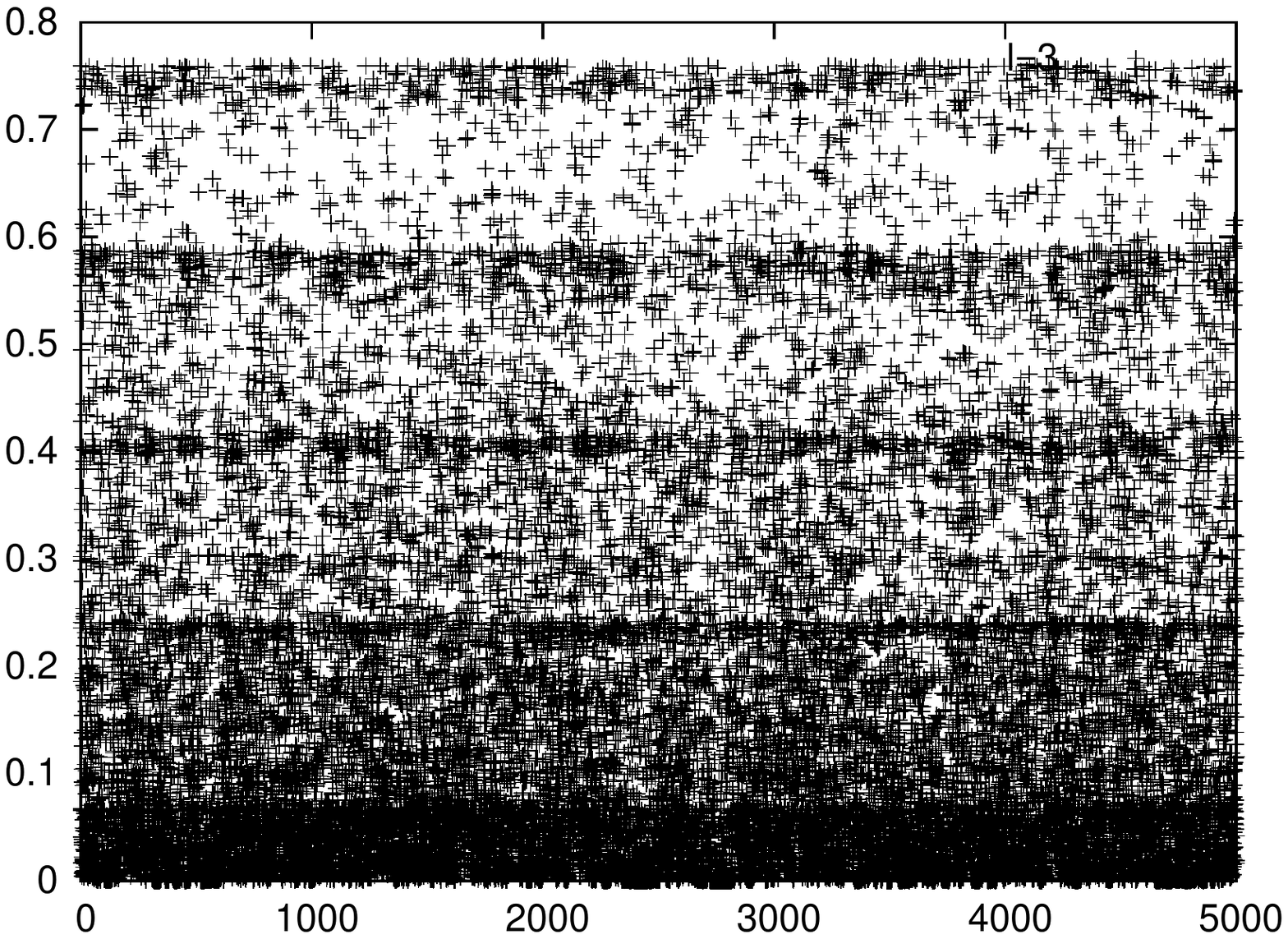}
\caption{Five runs for the perturbed logistic equation, $x_0=0.5$,
$\sigma=0.865$, $r=2$ and  \eqref{def:dud}: (left)  $l=1$ with fast
convergence to the zero equilibrium, (middle) $l=2$ where slower convergence is observed, and (right) $l=3$,
there is no convergence.
}  
\label{figure4}
\end{figure}
\ee

\end{example}

Stabilization of a positive equilibrium $K>0$ corresponds to the equation
\[
x_{n+1}=x_nf(x_n)+\sigma[x_nf(x_n)-K]\xi_{n+1} .
\]
\begin{example}
\label{ex:forlena5}
\be
\item [(a)] Consider   \eqref{def:xicontk} with $s=0$, i.e. the continuous uniform distribution on $[-1, 1]$.
\be
\item [(i)] For the Ricker model we can stabilize $K=1$ for $\sigma=1$ and $r \in (2,2.3068)$. 
Fig.~\ref{figure6} illustrates that stabilization for these and some  higher $r$ can be observed.

\begin{figure}[ht]
\centering
\vspace{-25mm}
\includegraphics[height=.3\textheight]{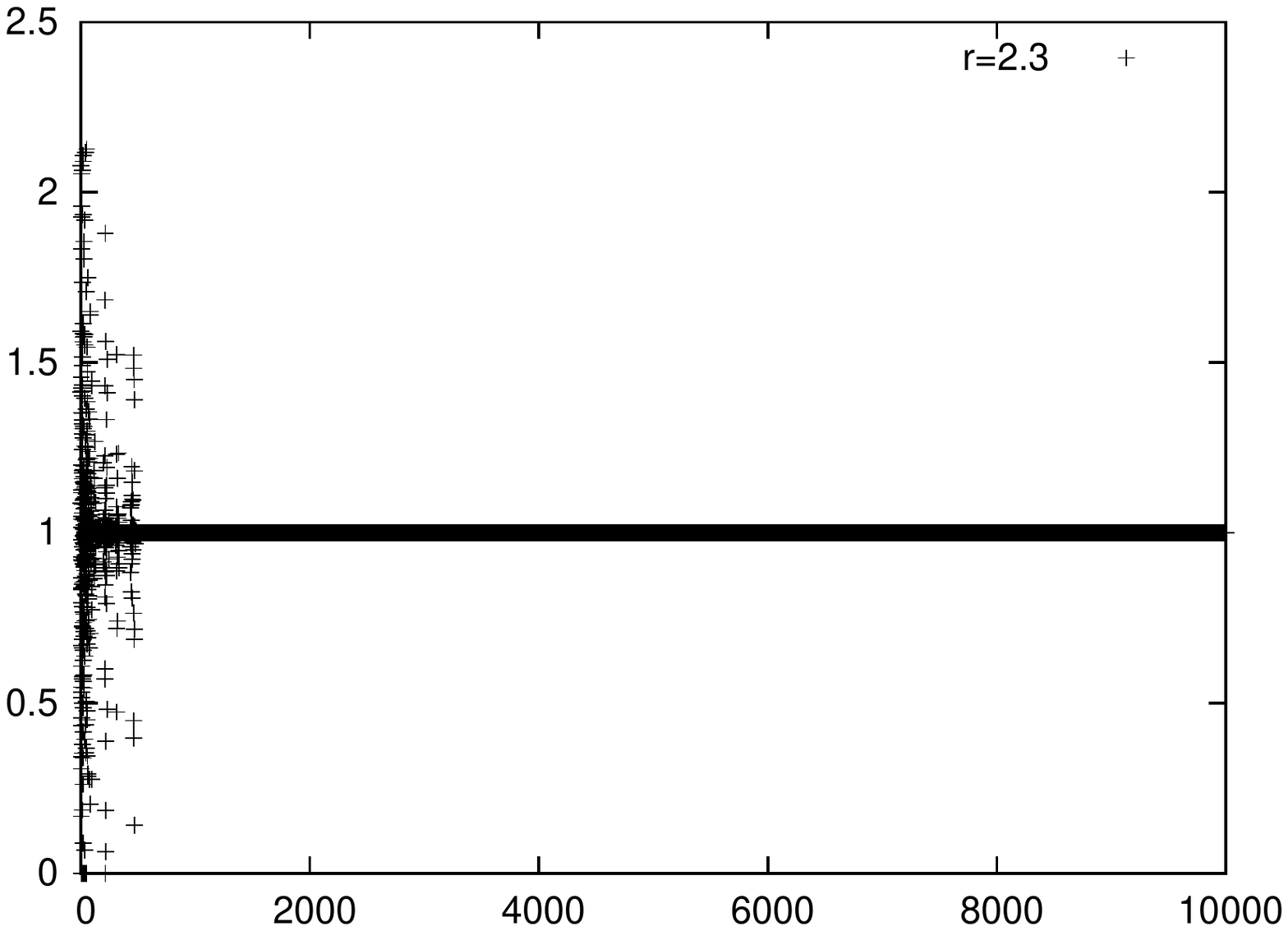}
\includegraphics[height=.3\textheight]{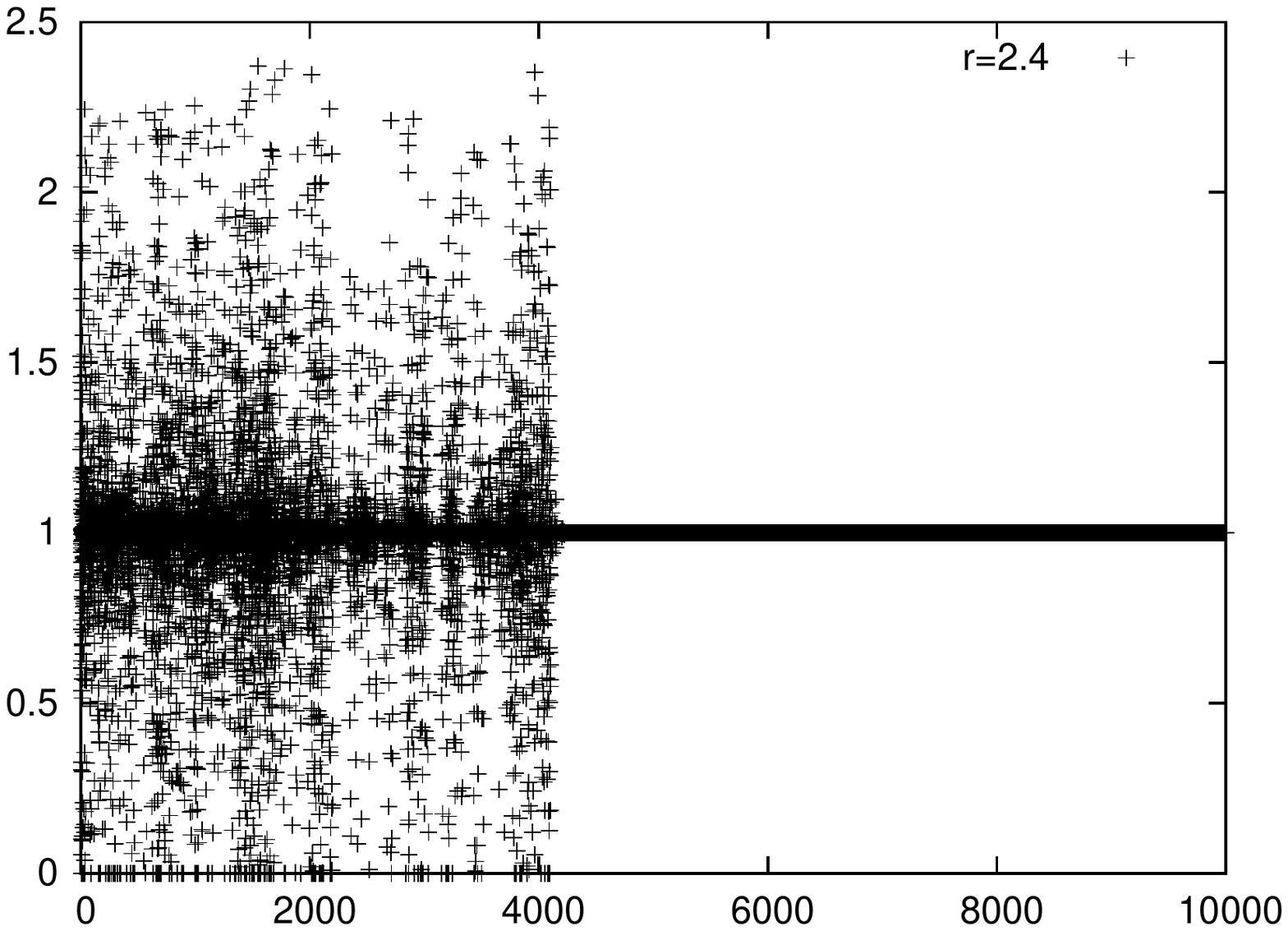}
\includegraphics[height=.3\textheight]{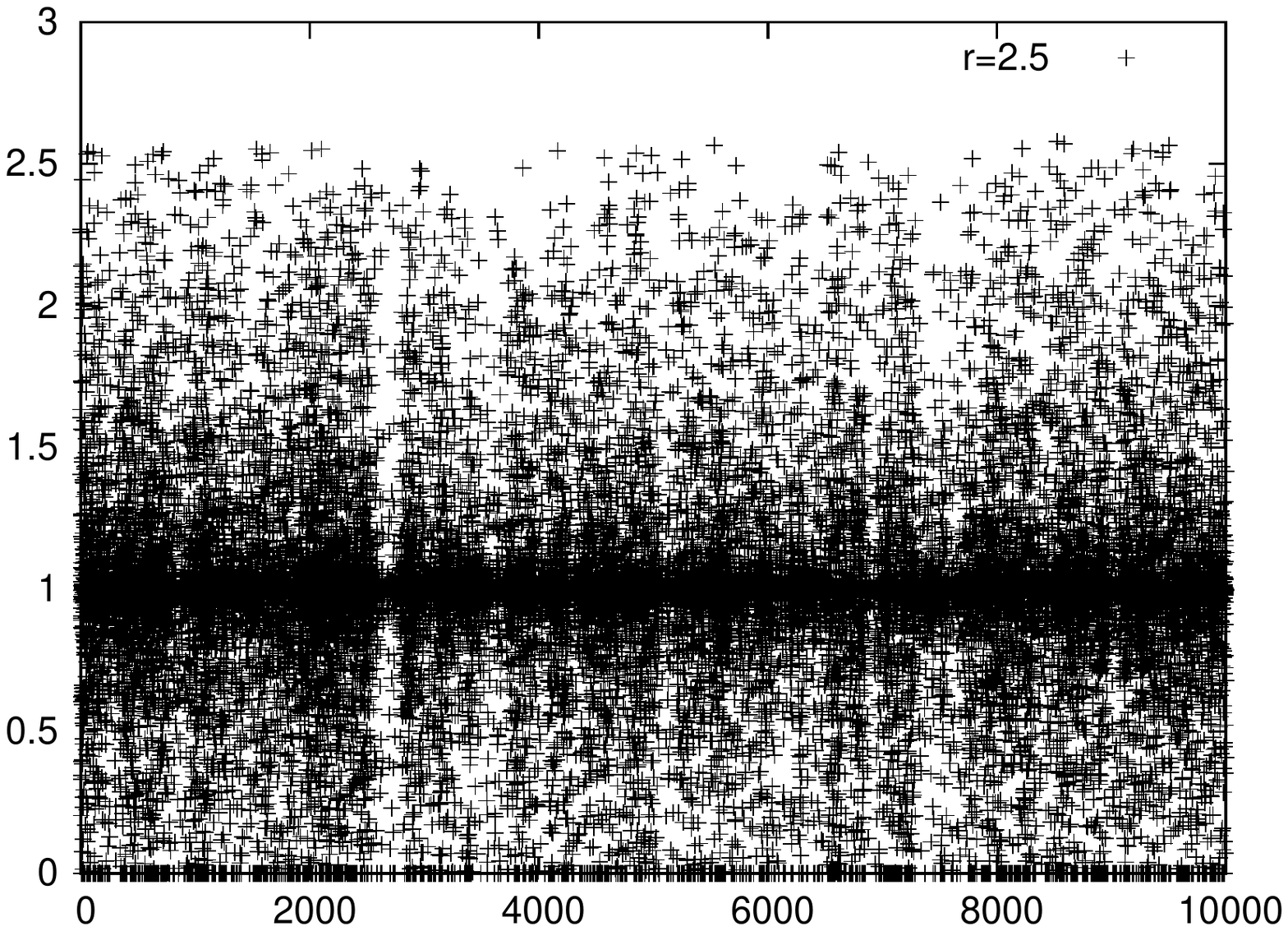}
\caption{Five runs for the perturbed Ricker equation, $x_0=0.5$, $\sigma=1$, \eqref{def:xicontk} with $s=0$ and (left)
$r=2.3$ with fast
convergence to the equilibrium $K=1$, (middle) $r=2.4$ where slower convergence is observed, and (right) $r=2.5$,
there is no convergence.
}  
\label{figure6}
\end{figure}

\item [(ii)] For the logistic map we consider stabilization of  $K=1-\frac 1r$ for $r\in (3, 3.35)$ and $\sigma=1$.
Fig.~\ref{figure7} shows that $r \approx 3.3484$ is an approximate stabilization bound.

\begin{figure}[ht]
\centering
\vspace{-25mm}
\includegraphics[height=.3\textheight]{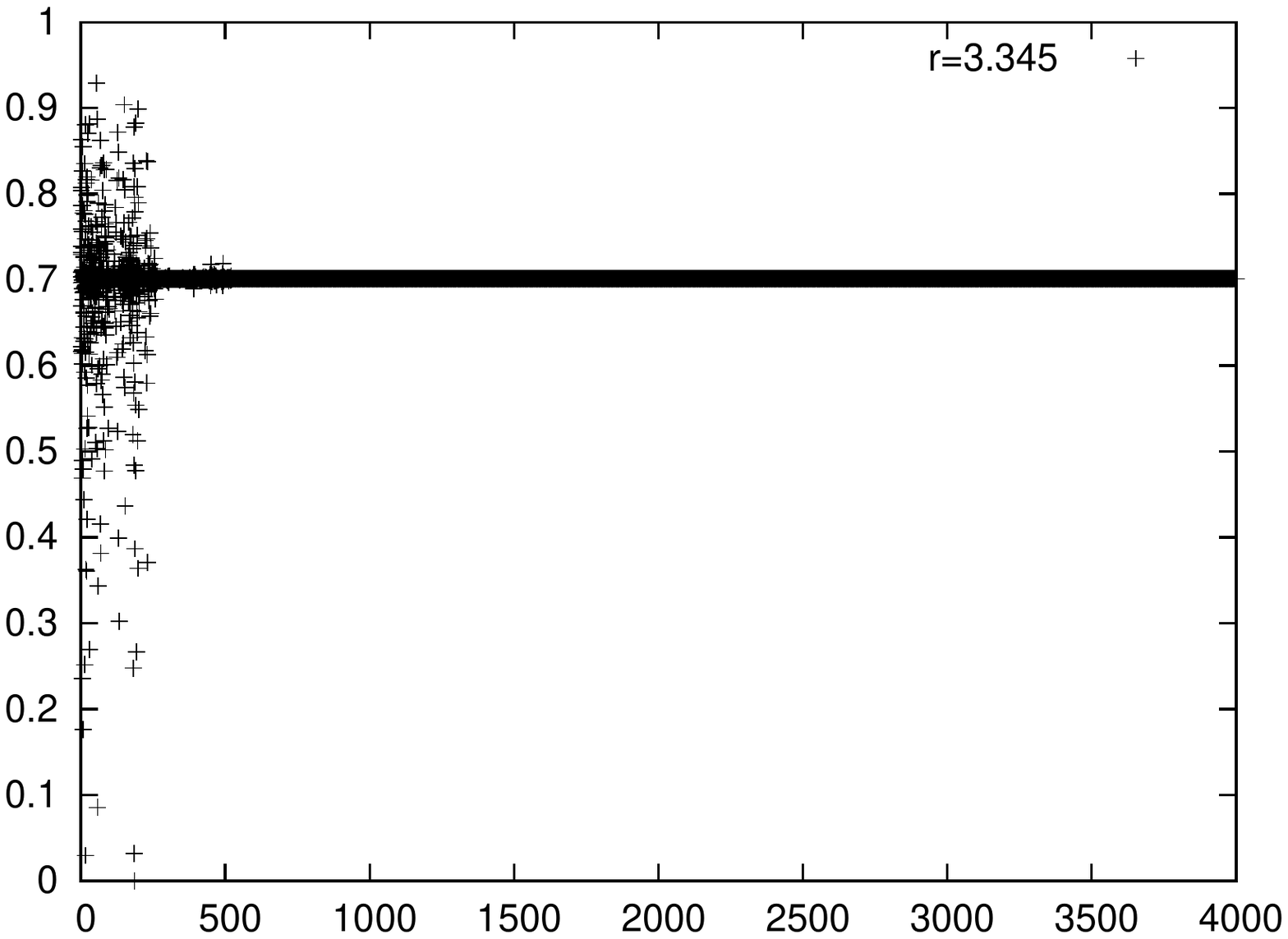}
\includegraphics[height=.3\textheight]{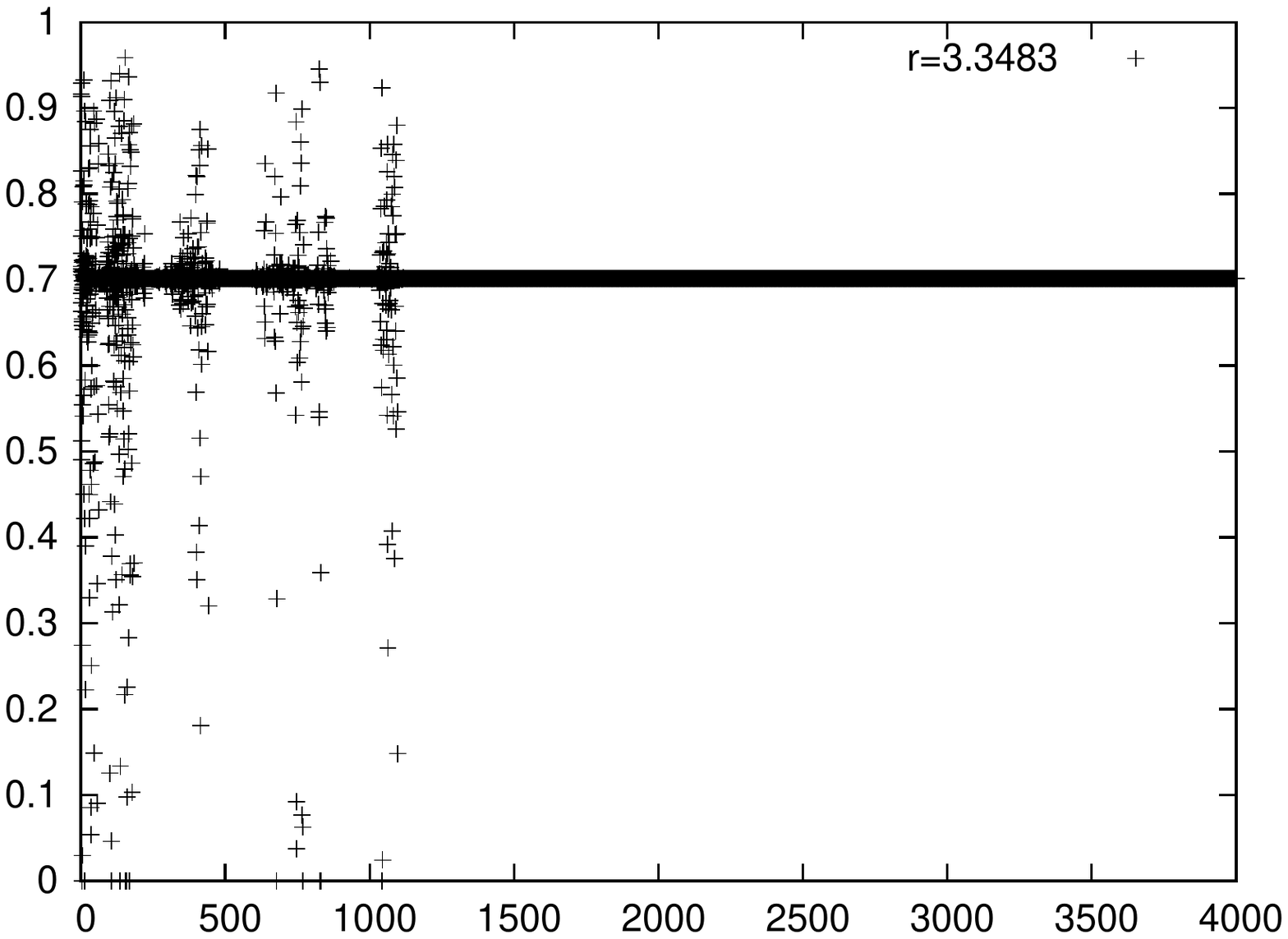}
\includegraphics[height=.3\textheight]{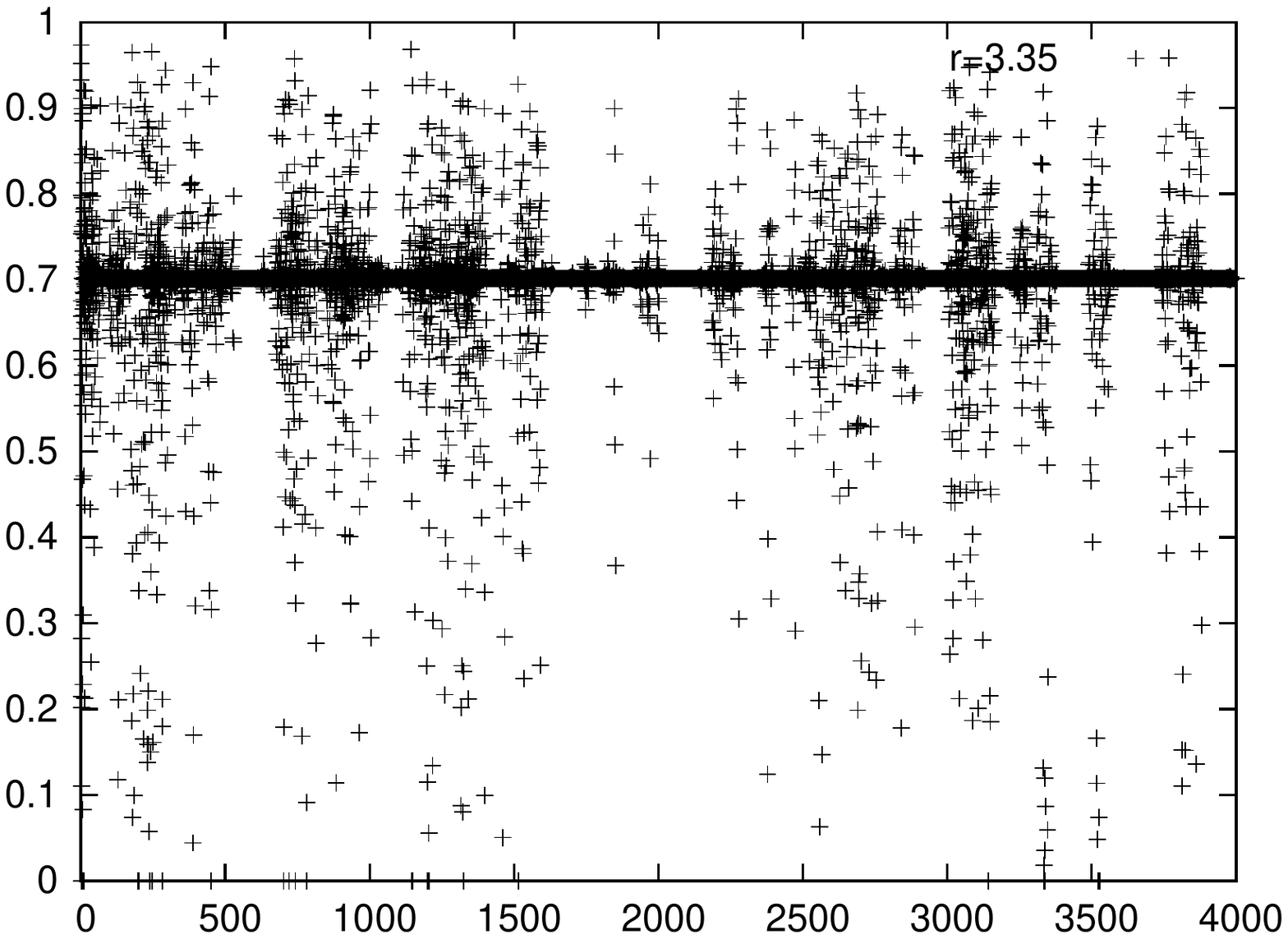}
\caption{Five runs for the perturbed logistic equation, $x_0=0.5$, $\sigma=1$, \eqref{def:xicontk} with $s=0$, and 
(left)
$r=3.345$ with fast
convergence to the equilibrium $K=1-\frac{1}{r}$, (middle) $r=3.3483$ where slower convergence is observed, and (right) $r=3.35$, 
there is no convergence.
}  
\label{figure7}
\end{figure}
\ee
\item [(b)]  Consider \eqref{def:dud}. 

\be
\item [(i)] The logistic map for $r=4$ is chaotic. We stabilize the positive equilibrium using a noise with 
$\sigma=1.05$ and various $l$, see Fig.~\ref{figure9}.

\begin{figure}[ht]
\centering
\vspace{-25mm}
\includegraphics[height=.3\textheight]{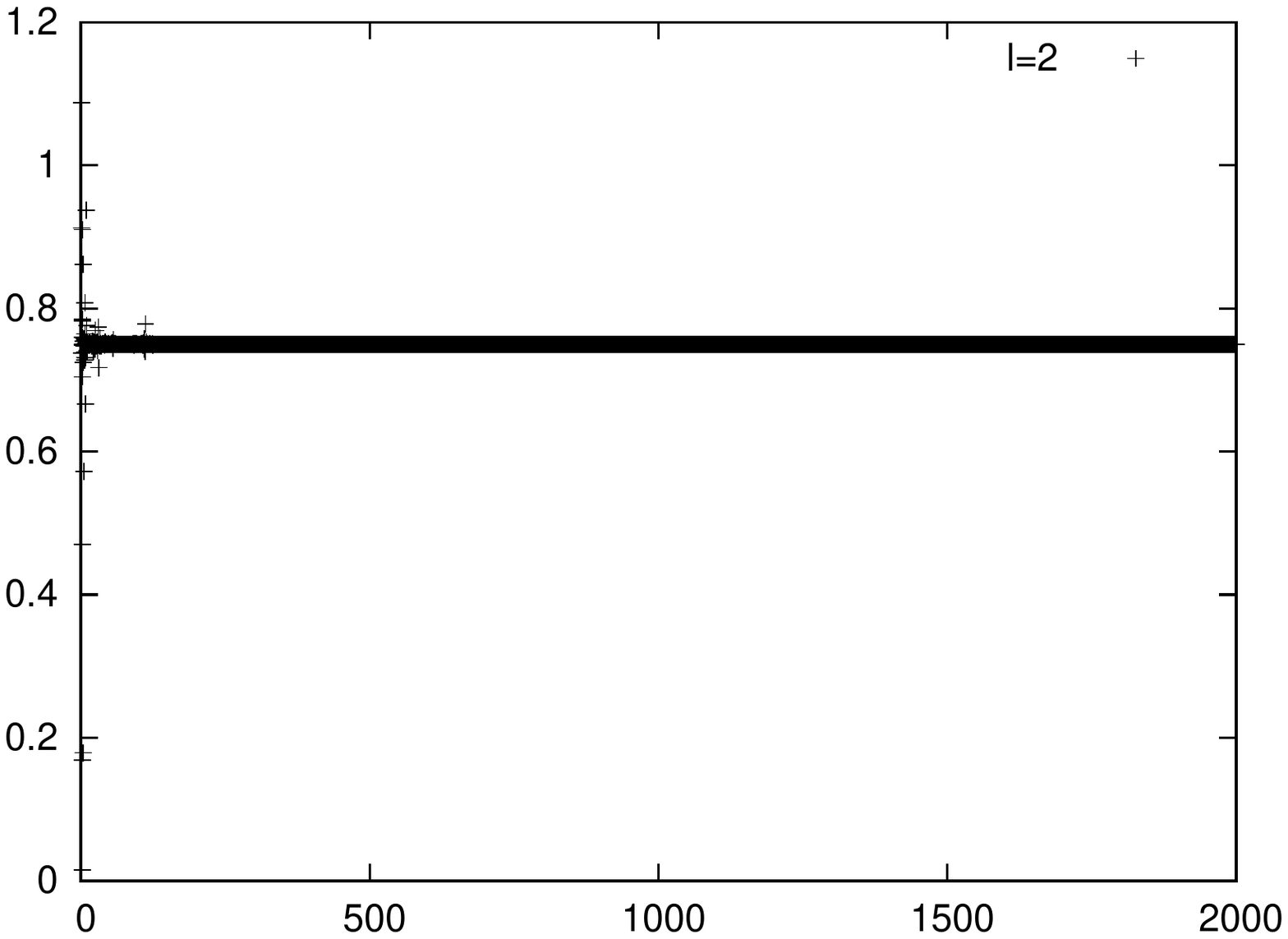}
\includegraphics[height=.3\textheight]{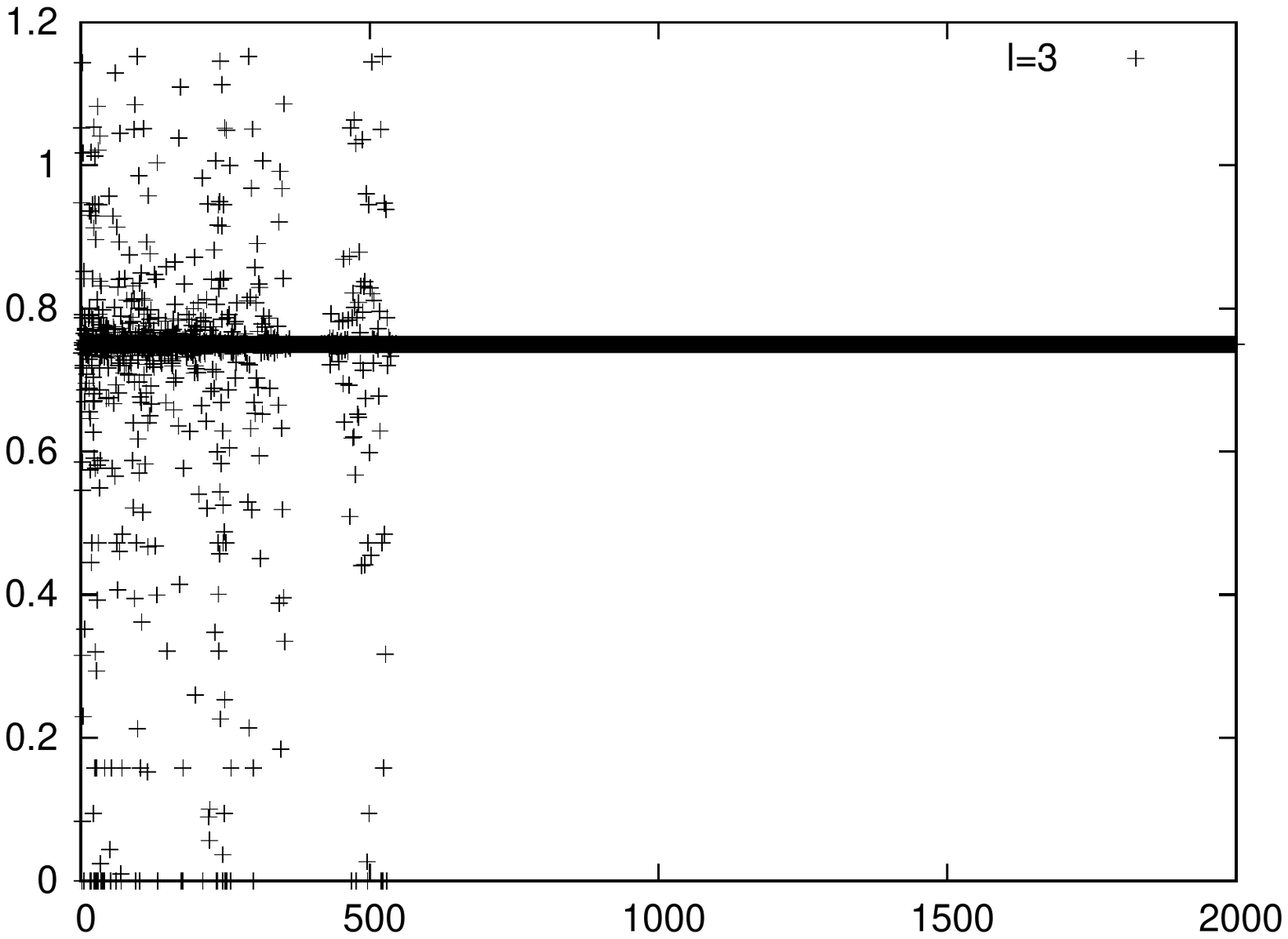}
\includegraphics[height=.3\textheight]{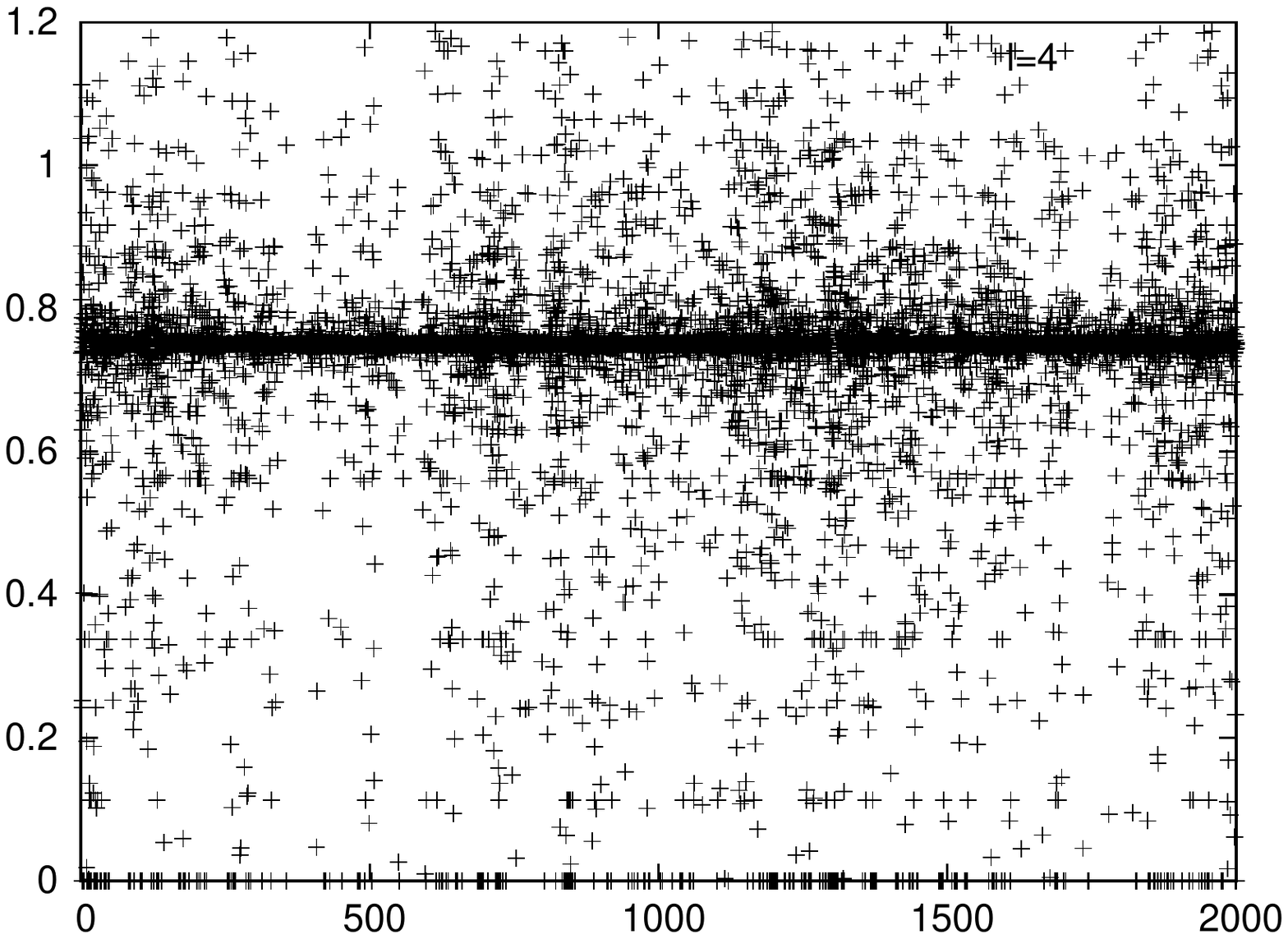}
\caption{Five runs for the perturbed logistic equation, $x_0=0.5$,
$\sigma=1.05$, $r=4$ and  \eqref{def:dud}: (left) $l=2$ with fast
convergence to the positive equilibrium, (middle) $l=3$ where slower convergence is observed, and (right) $l=4$,
there is no convergence.
}  
\label{figure9}
\end{figure}

\item[(ii)]  Function   $F(x)=\frac {3x}{2+(x-3)^2}$ has   two positive fixed points, $K_1=2$ and $K_2=4$.  
Also, $F'(4)= \frac{3(11-x^2)}{(2+(x-3)^2)^2}\biggr|_{x=4}=-\frac 53<-1$, so $K_2$ is an unstable equilibrium. 
To stabilize $K_2=4$  we consider the modified Beverton-Holt equation
\[
 x_{n+1}=\frac {3x_n}{2+(x_n-3)^2}+\left[\frac {3x_n}{2+(x_n-3)^2}-4\right]\xi_{n+1},\quad x_0\in \mathbb R,
\]
with $\xi$ defined by \eqref{def:dud}.  Stabilization  depending on $l$  is illustrated in Fig.~\ref{figure10}.

\begin{figure}[ht]
\centering
\vspace{-25mm}
\includegraphics[height=.3\textheight]{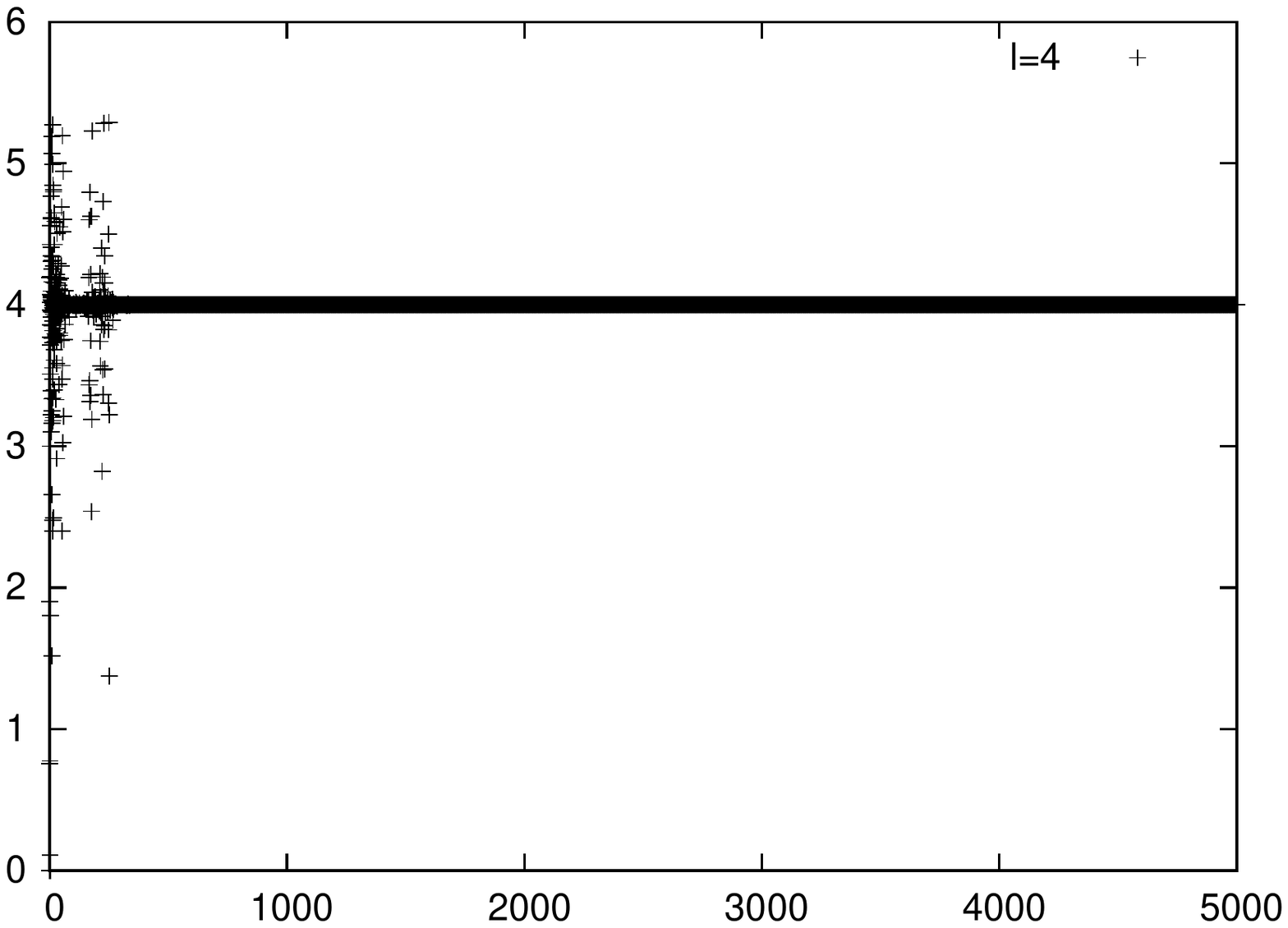}
\includegraphics[height=.3\textheight]{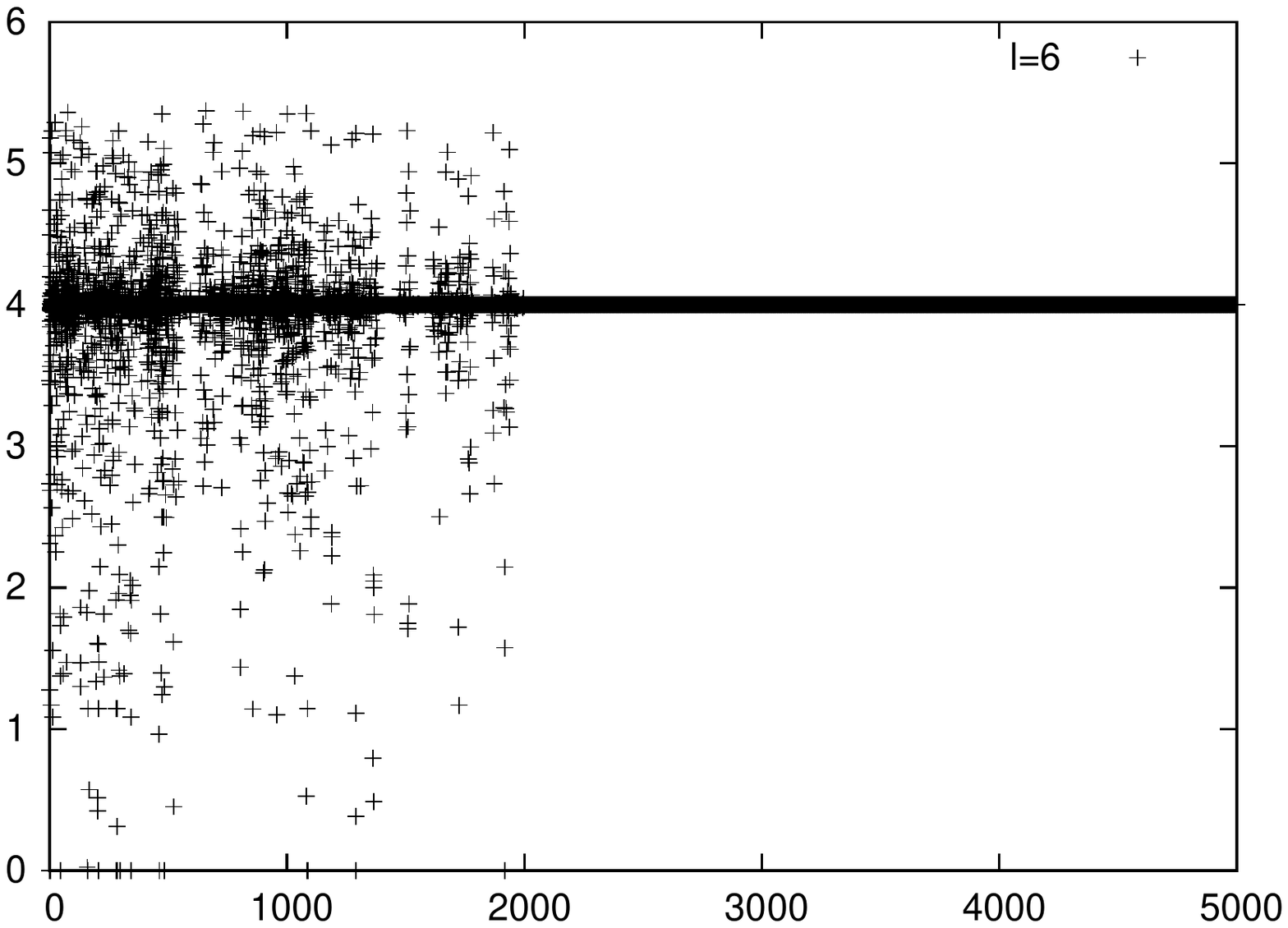}
\includegraphics[height=.3\textheight]{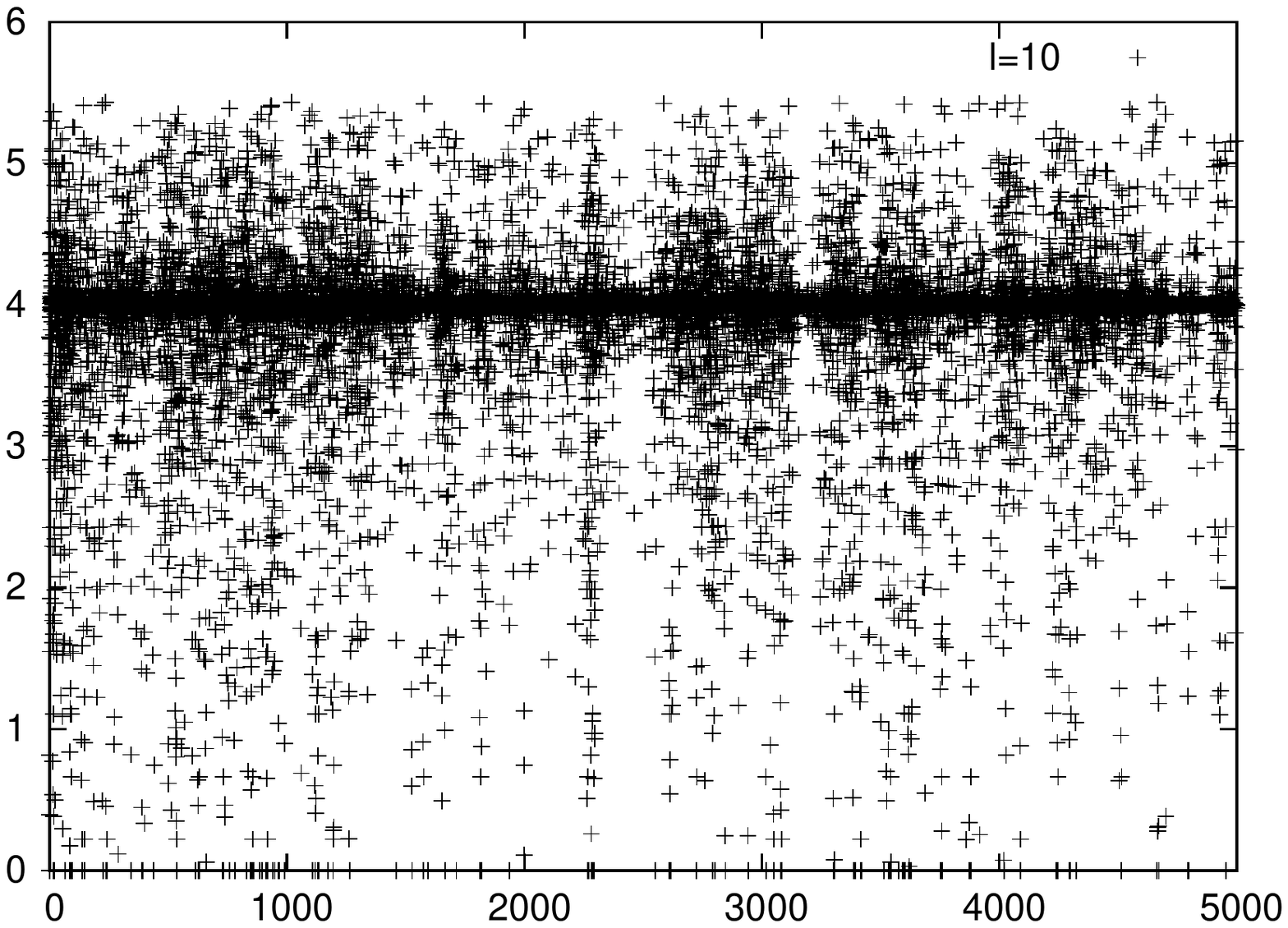}
\caption{Five runs for the perturbed modified Beverton-Holt equation, 
$x_0=0.5$, $\sigma=1.05$, and \eqref{def:dud}: (left) $l=4$ with fast
convergence to the positive equilibrium, (middle) $l=6$ where slower convergence is observed, and (right) $l=10$,
there is no convergence.
}  
\label{figure10}
\end{figure}

\ee
\ee
\end{example}

Finally, for stabilization we consider  the noise of type \eqref{def:contdud}.
\begin{example}
We consider stabilization \eqref{eq:stabK} of the positive equilibrium $K=1$ for the Ricker model with $r=2.1$, 
$\sigma=0.7$, noise \eqref{def:contdud} with $\delta=0.02$ and $l=2,6$, see Fig.~\ref{figure_add1}.

\begin{figure}[ht]
\centering
\vspace{-25mm}
\includegraphics[height=.3\textheight]{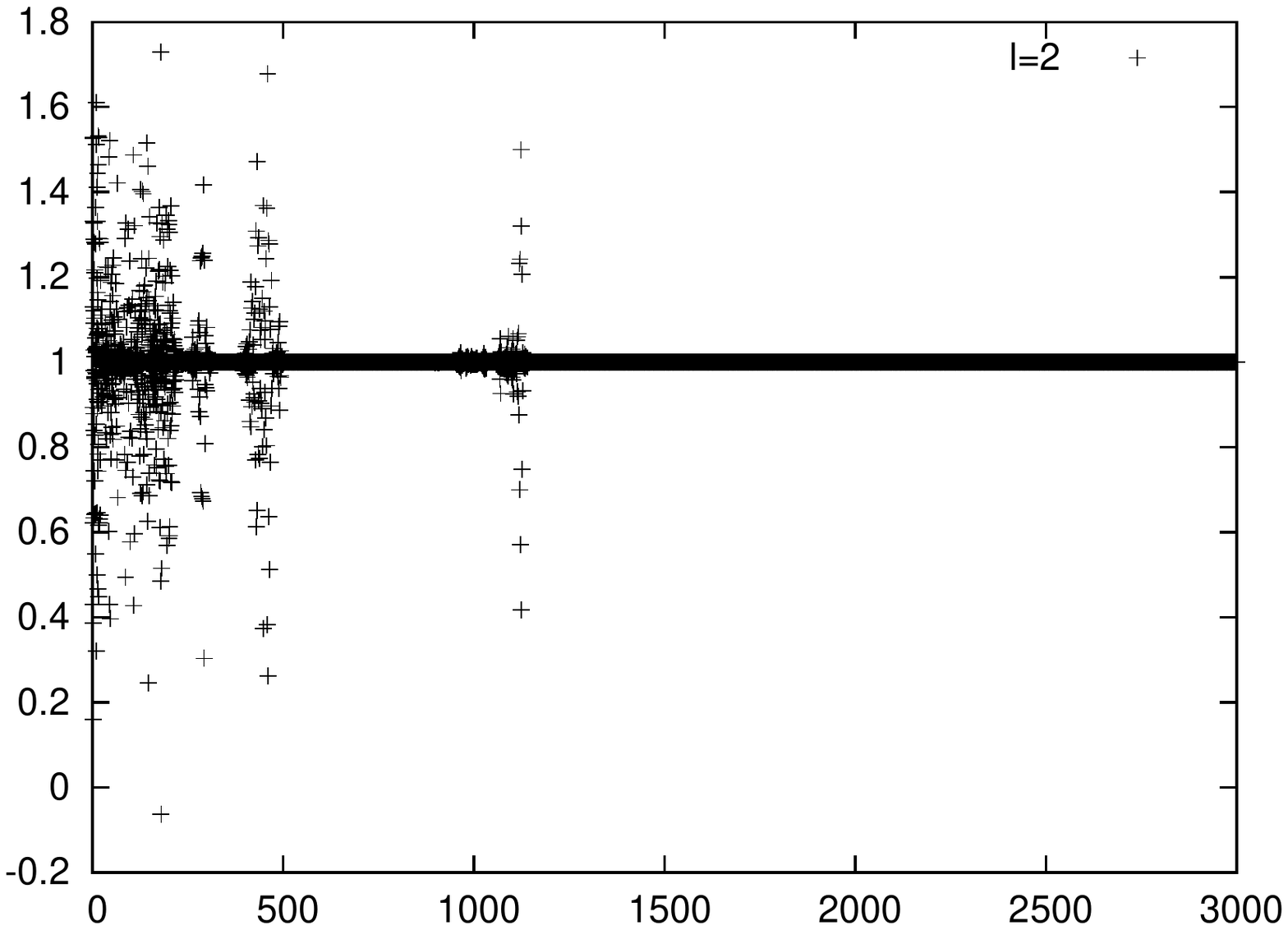}
\includegraphics[height=.3\textheight]{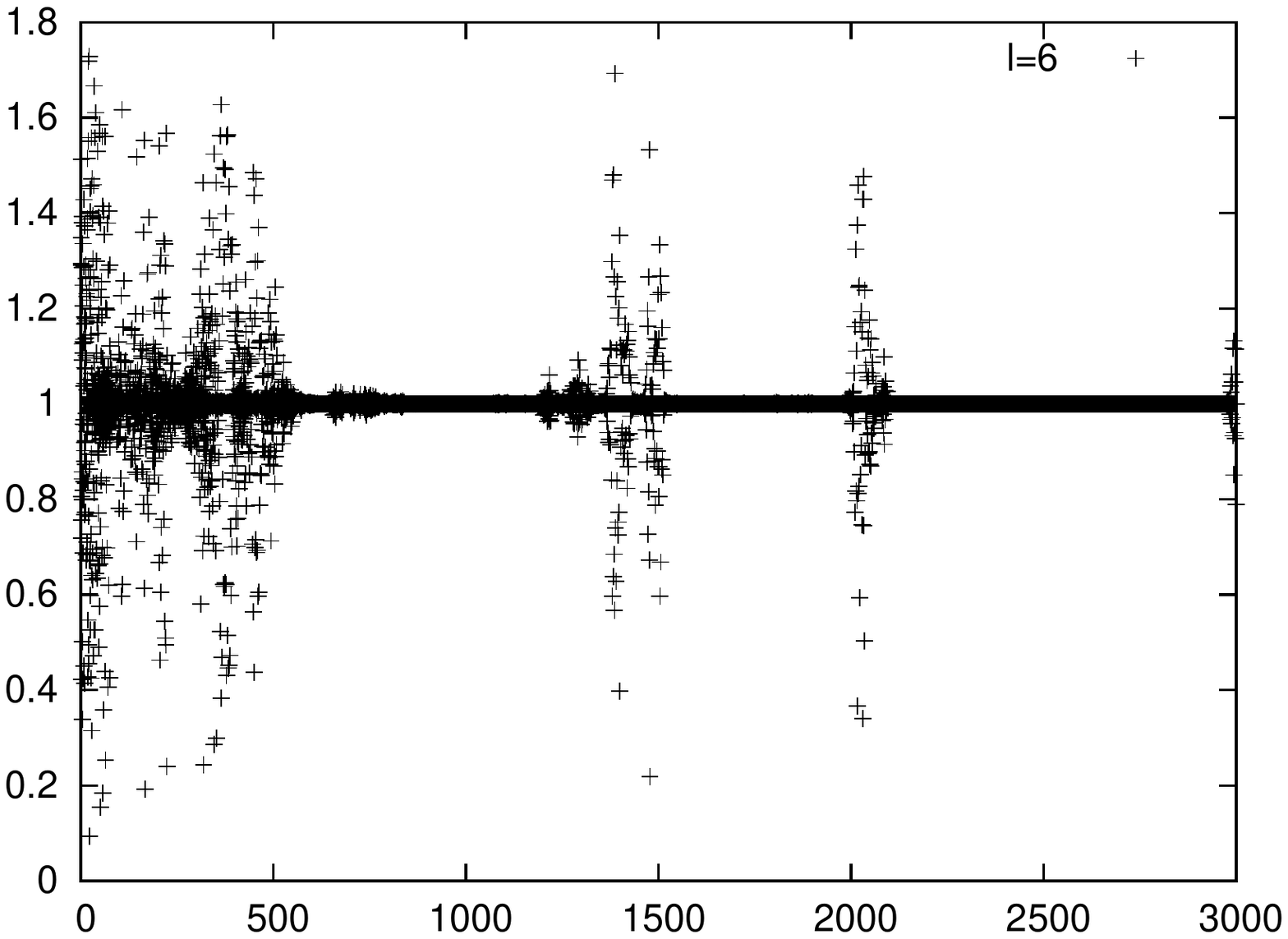}
\includegraphics[height=.3\textheight]{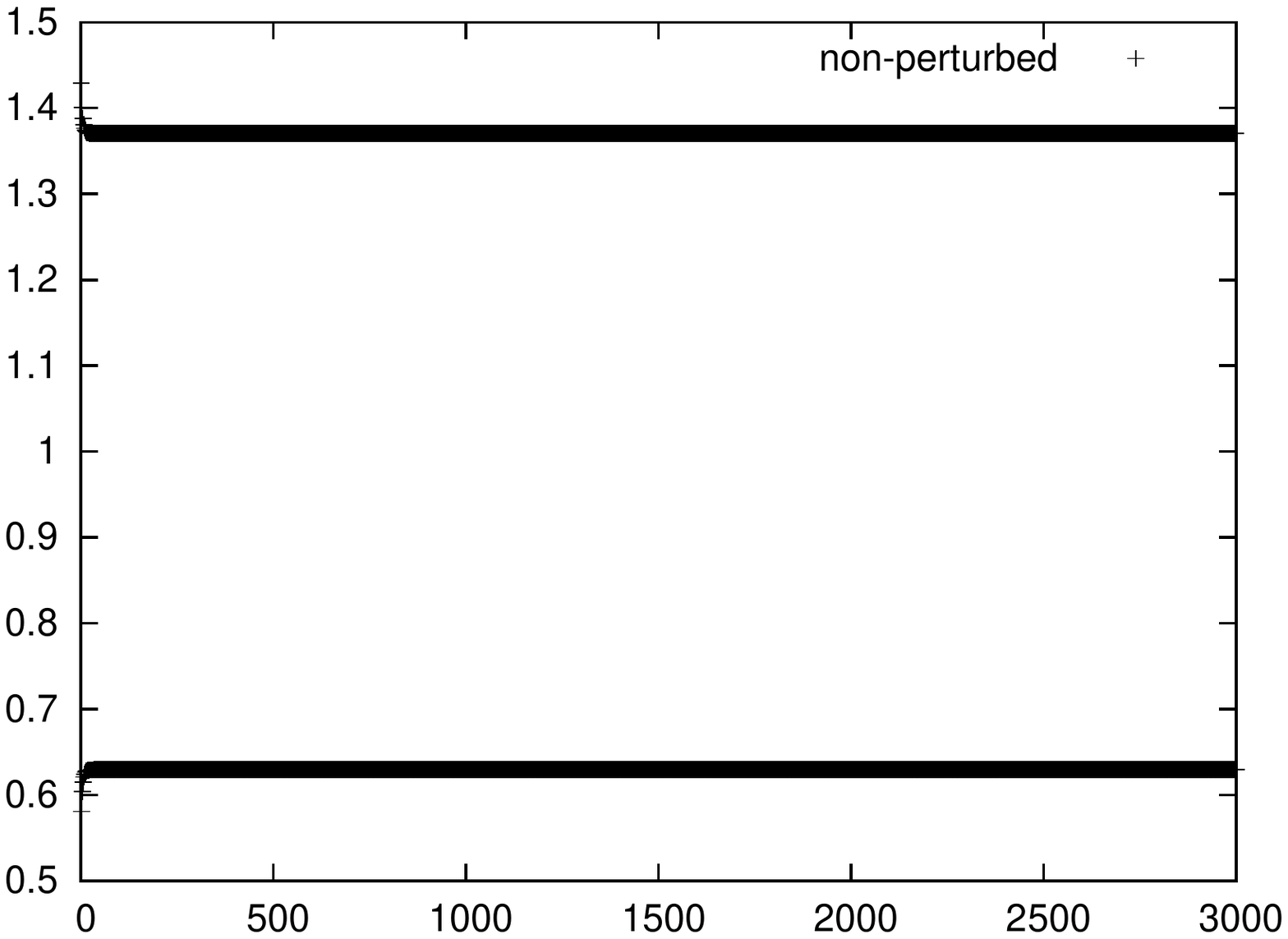}
\caption{The noise (\protect\ref{def:contdud}) with $\delta=0.02$ and 
%(upper left) $l=2$, (upper right) $l=6$ and 
five runs for the perturbed Ricker equation with $r=2.1$, $\sigma=0.7$, $x_0=0.5$ and
(left) $l=2$, (middle) $l=6$ and (right) for $\sigma=0$.  
}
\label{figure_add1}
\end{figure}
\end{example}

Next, let us proceed to destabilization. We start with destabilization of the zero equilibrium, which  corresponds to the equation
$\displaystyle
x_{n+1}=x_nf(x_n)+x_n\sigma_b(x_n) \xi_{n+1},
$
$\displaystyle \sigma_b(x) = \left\{ \begin{array}{ll} \sigma(x), & x \leq b, \\ 0, & x>b, \end{array} \right.$
where $\sigma(x)$ is chosen either by \eqref{ineq:sdud} or \eqref{cond:fe}, depending on the distribution of  $\xi$.
\begin{example}
\label{ex:destab0}
Consider $F$ as in  Example~\ref{ex:forlena5} (b)(ii). Since $F'(0)=\frac 3{11}<1$, zero is a stable equilibrium. 
We have 
$
F_{max}\approx 4.73736, \quad F(F_{max})=F(4.73736)\approx 2.8320\in (2, F_m).
$
We put
$
b=2, \quad \mbox{then}\quad  F:[b, F_m]\to [b, F_m],
$
so $[b, F_m]=[2, 4.73736] $ could be a trap if $\displaystyle x_{n+1}=x_n\left[f(x_n)+\sigma (x_n)\xi_{n+1} \right]$, with $x_n\in [0,b]$, will not jump over $F_m$. 
Below $\xi$ is defined by two types of distributions: \eqref{def:xicontk}  and \eqref{def:dud}. 
In each case, the theoretical results from Section \ref{sec:destab} are more restrictive than suggested by simulations.

\be

\item [(i)]  Fig.~\ref{figure29} illustrates the influence of $\sigma$ on destabilization for distribution \eqref{def:xicontk} in case $s=1$: still stability for $\sigma=1.2$,
partial destabilization for $\sigma=1.4$ and destabilization for $\sigma=1.6$.  
Fig.~\ref{figure30} considers the case $s=4$ with stability for  $\sigma=1.1$, partial destabilization
for $\sigma=1.15$ and destabilization for $\sigma=1.2$; here $s=4$, we observe that higher values of $s$ lead to smaller $\sigma$, sufficient for destabilization.
    
\begin{figure}[ht] 
\centering
\vspace{-25mm}
\includegraphics[height=.3\textheight]{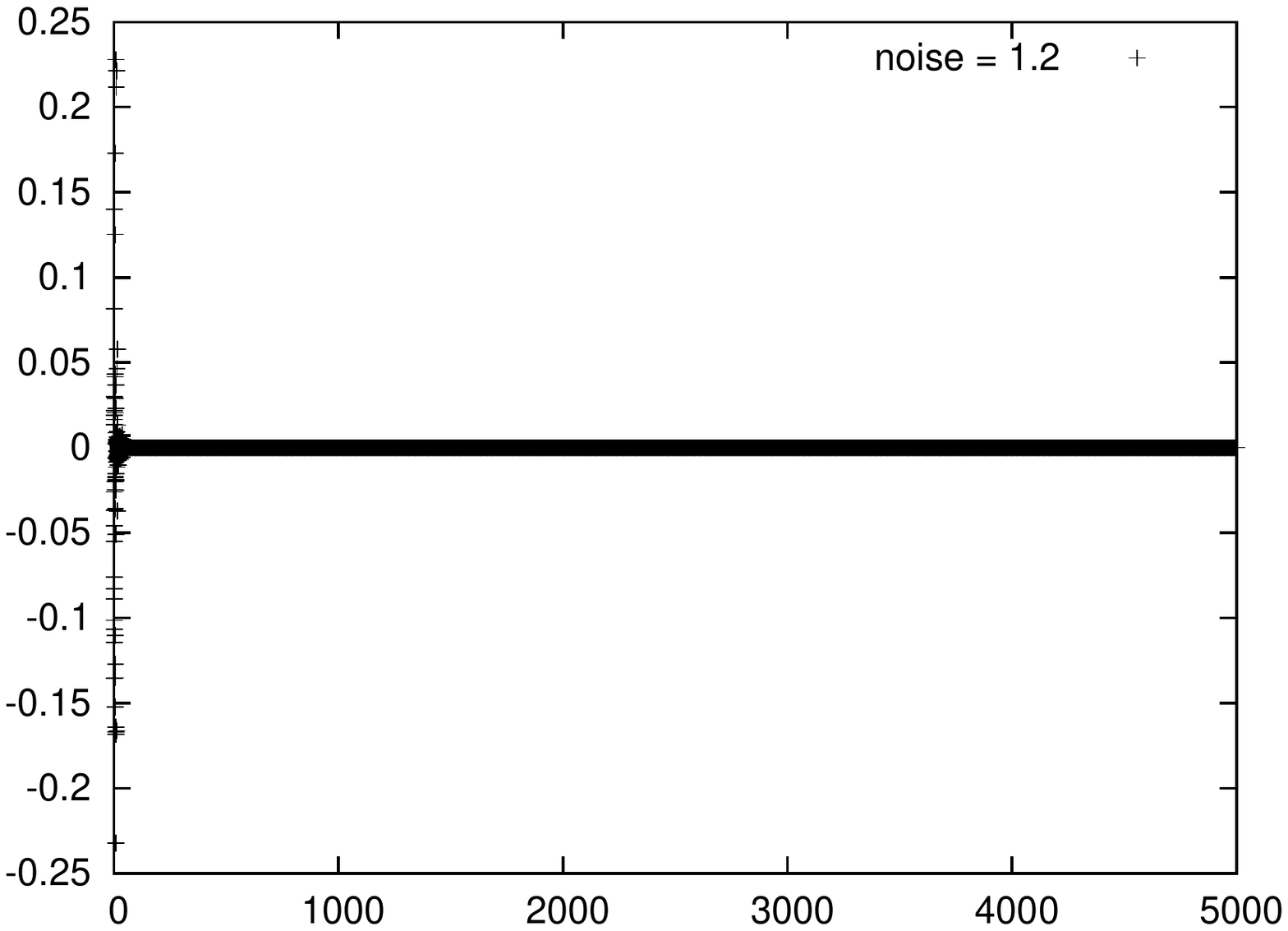}
\includegraphics[height=.3\textheight]{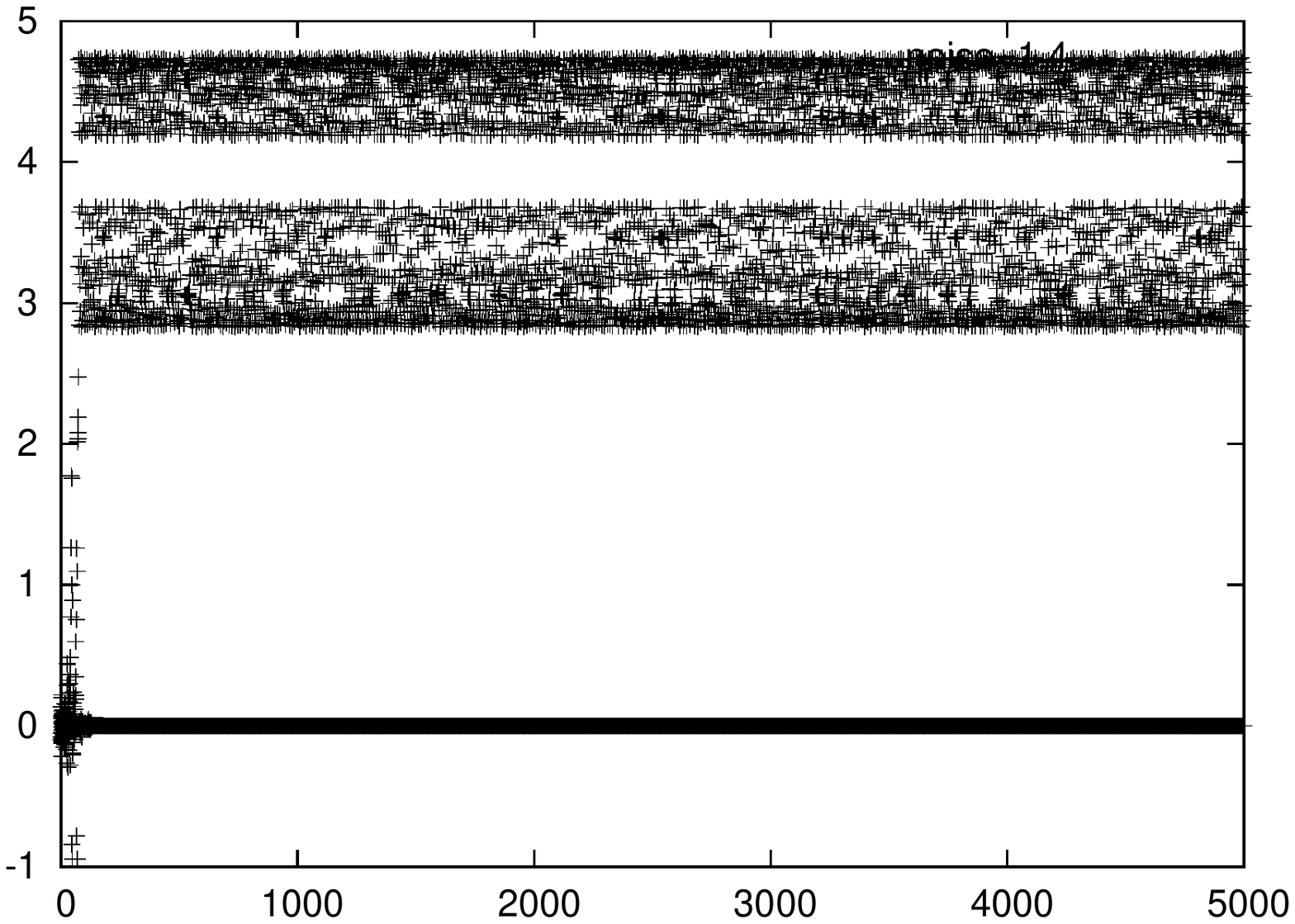}
\includegraphics[height=.3\textheight]{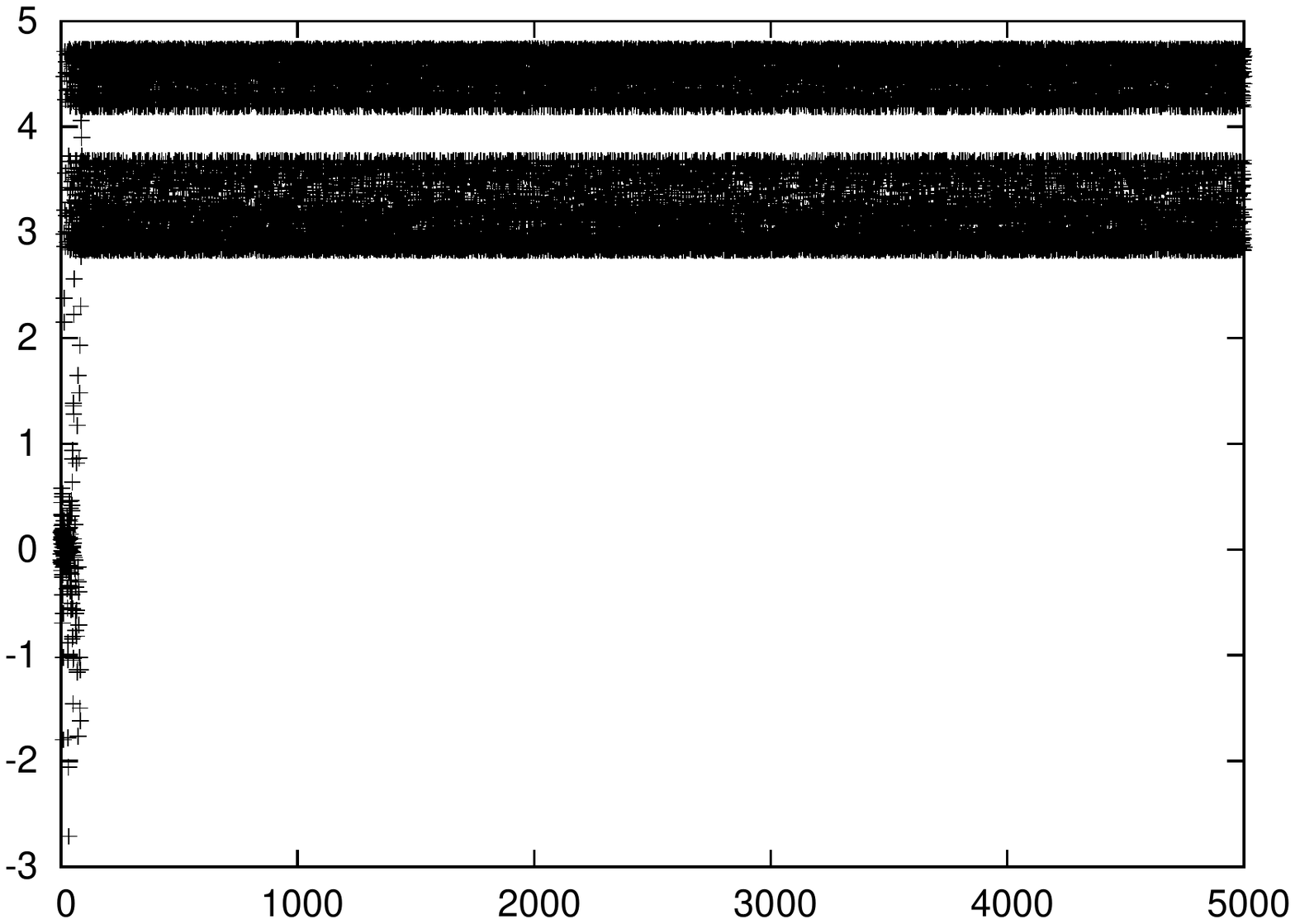}
\caption{Five runs for (\protect{\ref{def:xicontk}})  with $s=1$, $x_0=0.1$ and (from left to right) 
$\sigma=1.2,1.4,1.6$.}
\label{figure29}
\end{figure}

\begin{figure}[ht]  
\centering
\vspace{-25mm}
\includegraphics[height=.3\textheight]{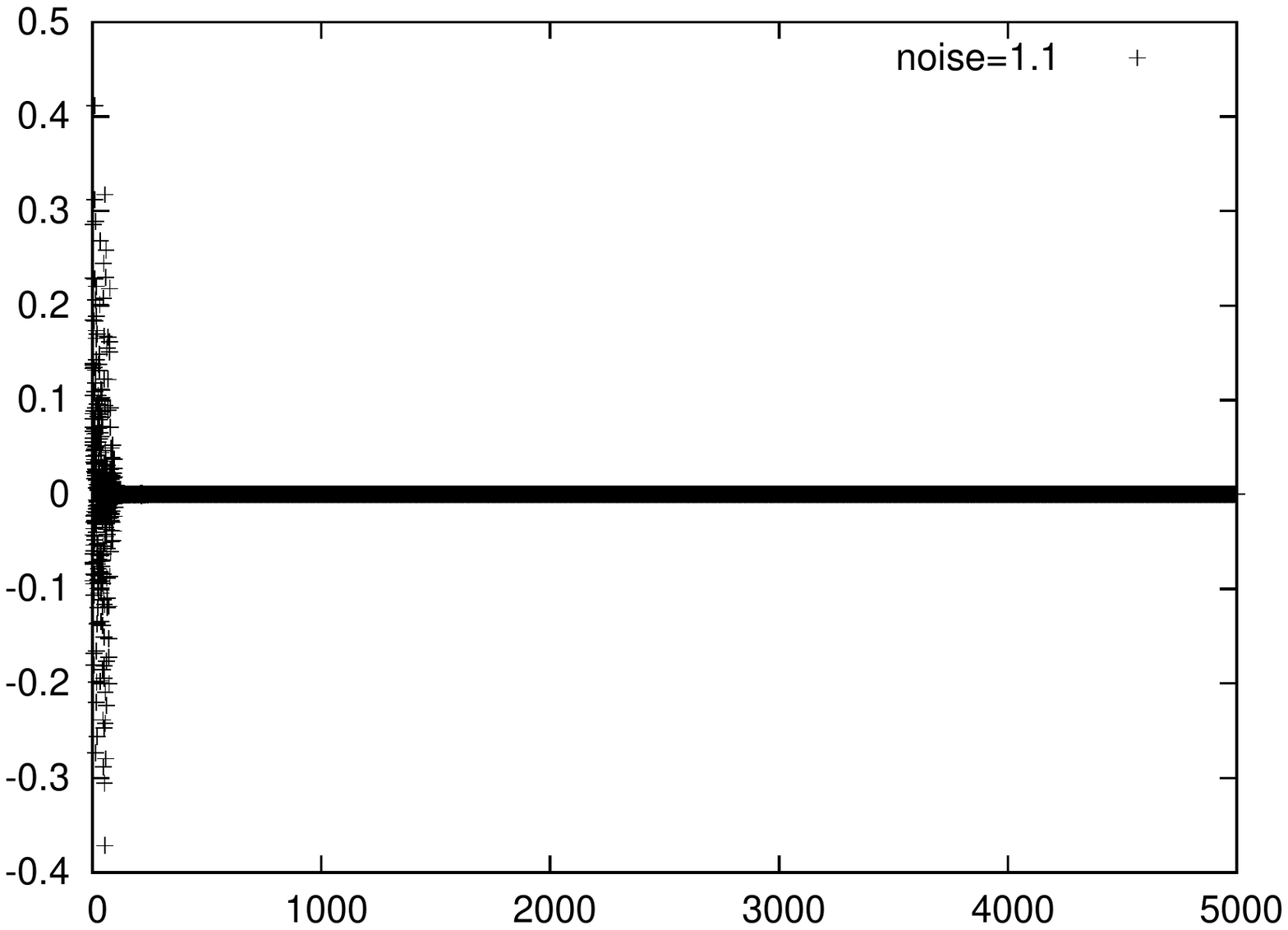}
\includegraphics[height=.3\textheight]{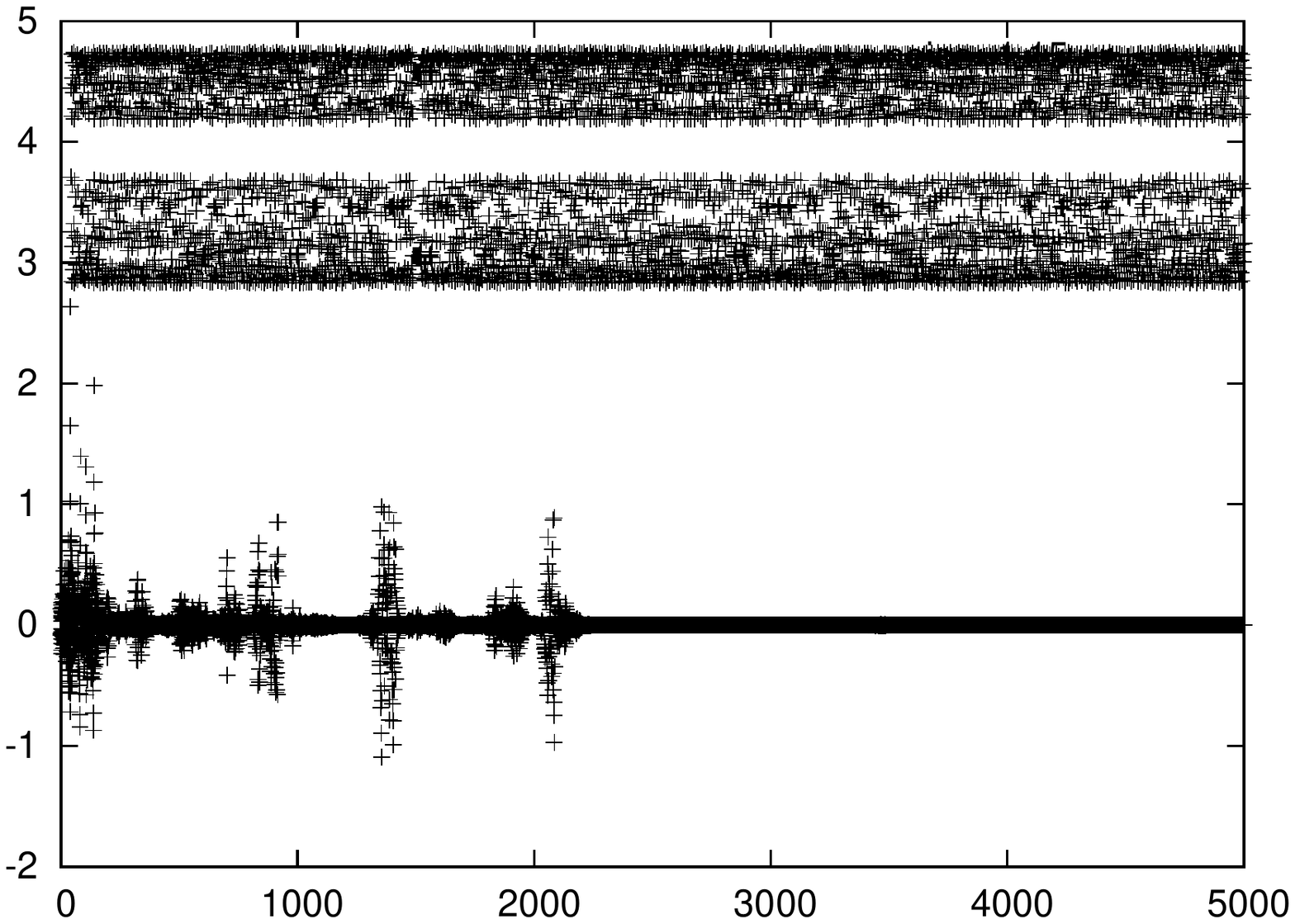}
\includegraphics[height=.3\textheight]{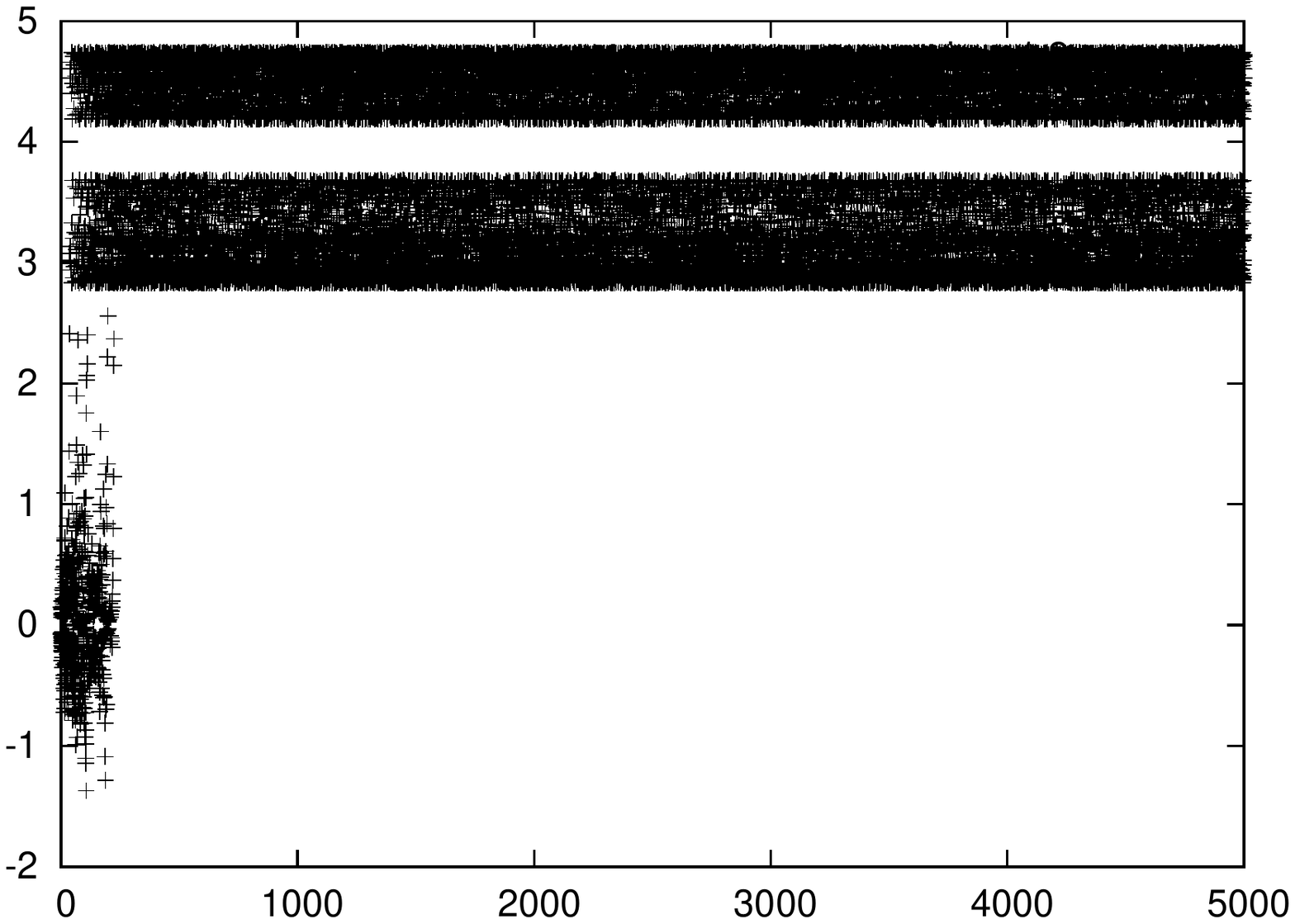}
\caption{Five runs for (\protect{\ref{def:xicontk}}) with $s=4$, $x_0=0.1$ and (from left to right) 
$\sigma=1.1,1.15,1.2$.}
\label{figure30}
\end{figure}

\item [(ii)]    Fig.~\ref{figure20} illustrates destabilization of zero for distribution \eqref{def:dud}, $l=1,3,4$, for an
initial value $x_0=0.3$ in the vicinity of the otherwise locally stable zero equilibrium, where higher
$l$ makes the transient period shorter.

\begin{figure}[ht]
\centering
\vspace{-25mm}
\includegraphics[height=.3\textheight]{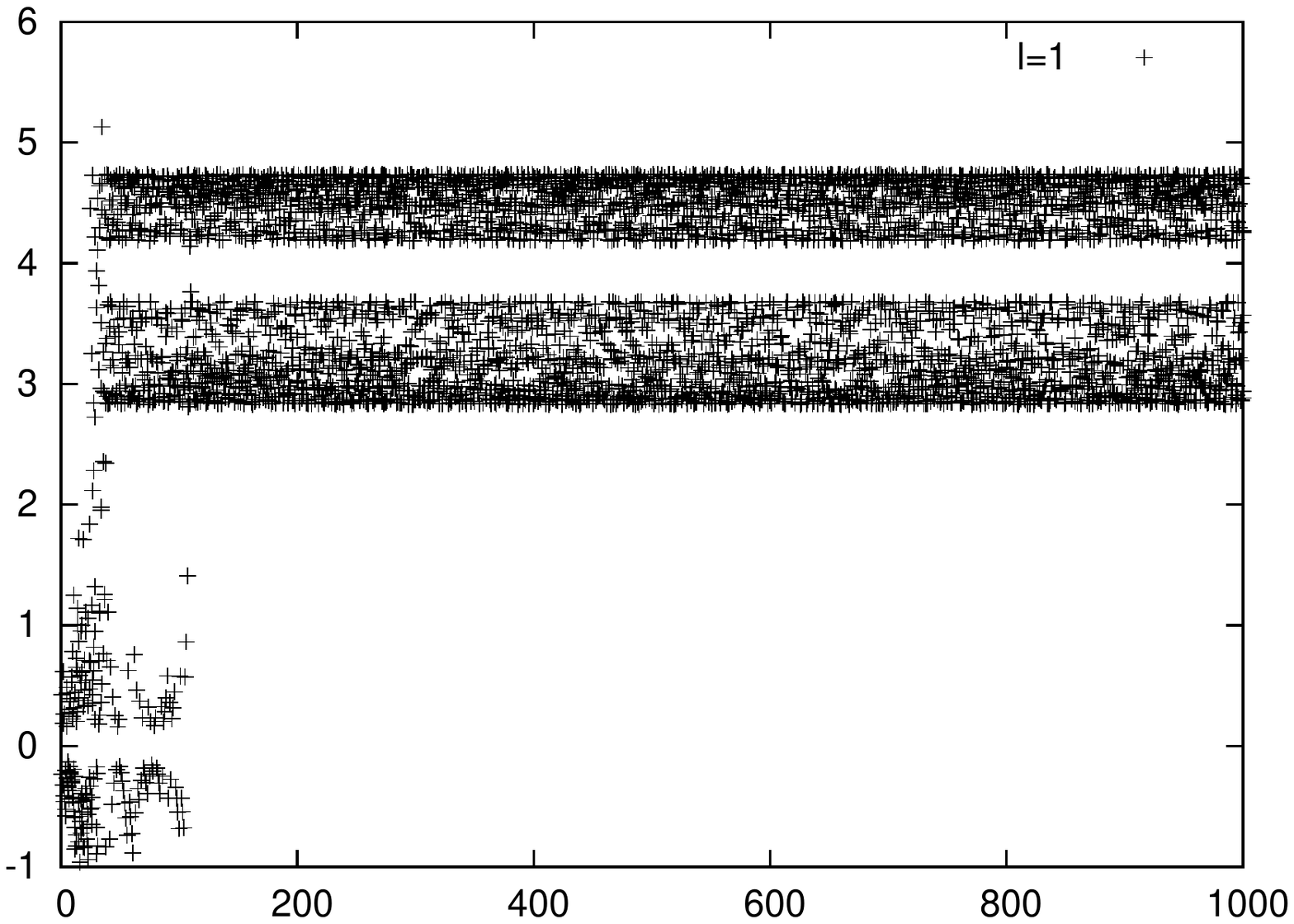}
\includegraphics[height=.3\textheight]{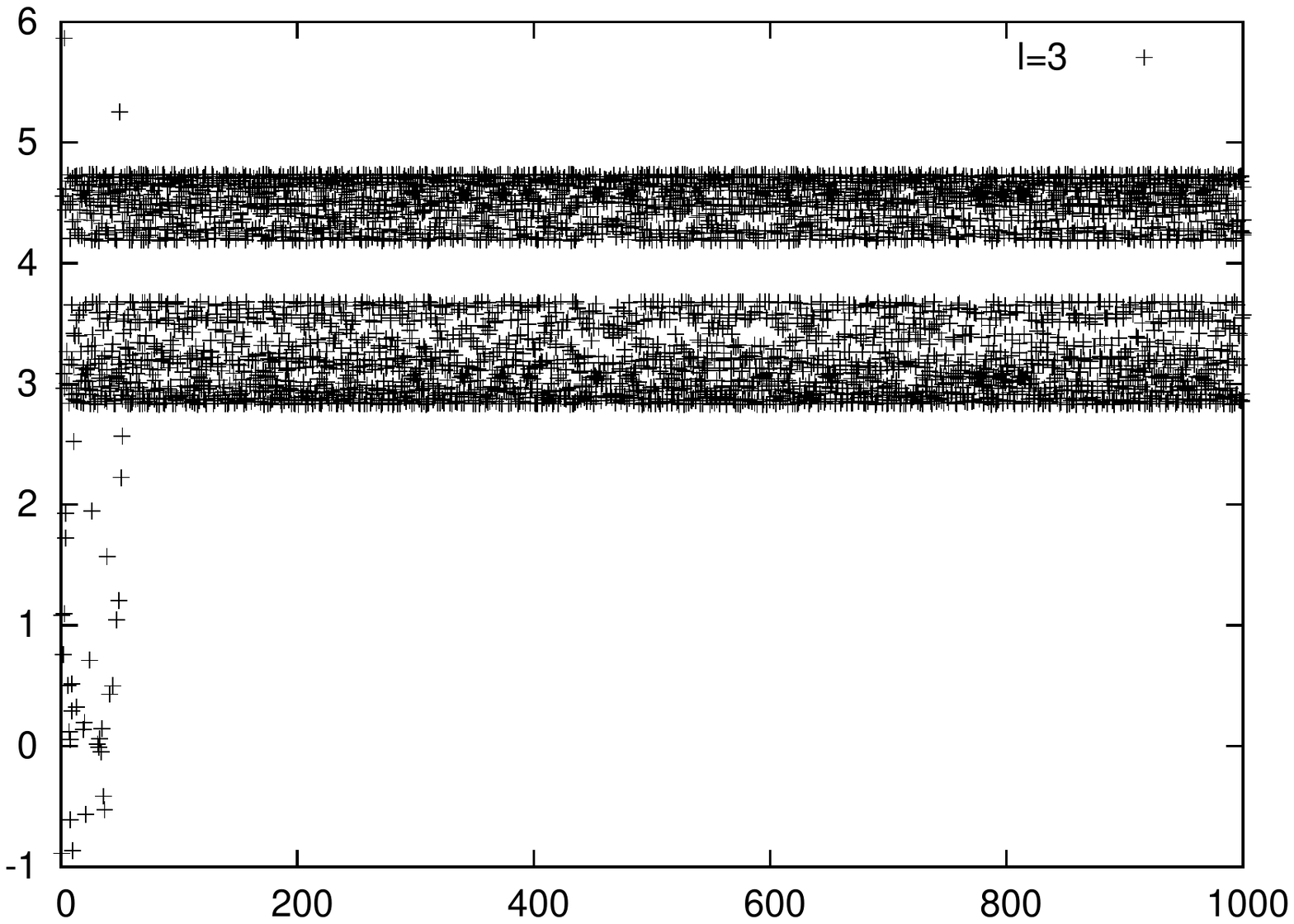}
\includegraphics[height=.3\textheight]{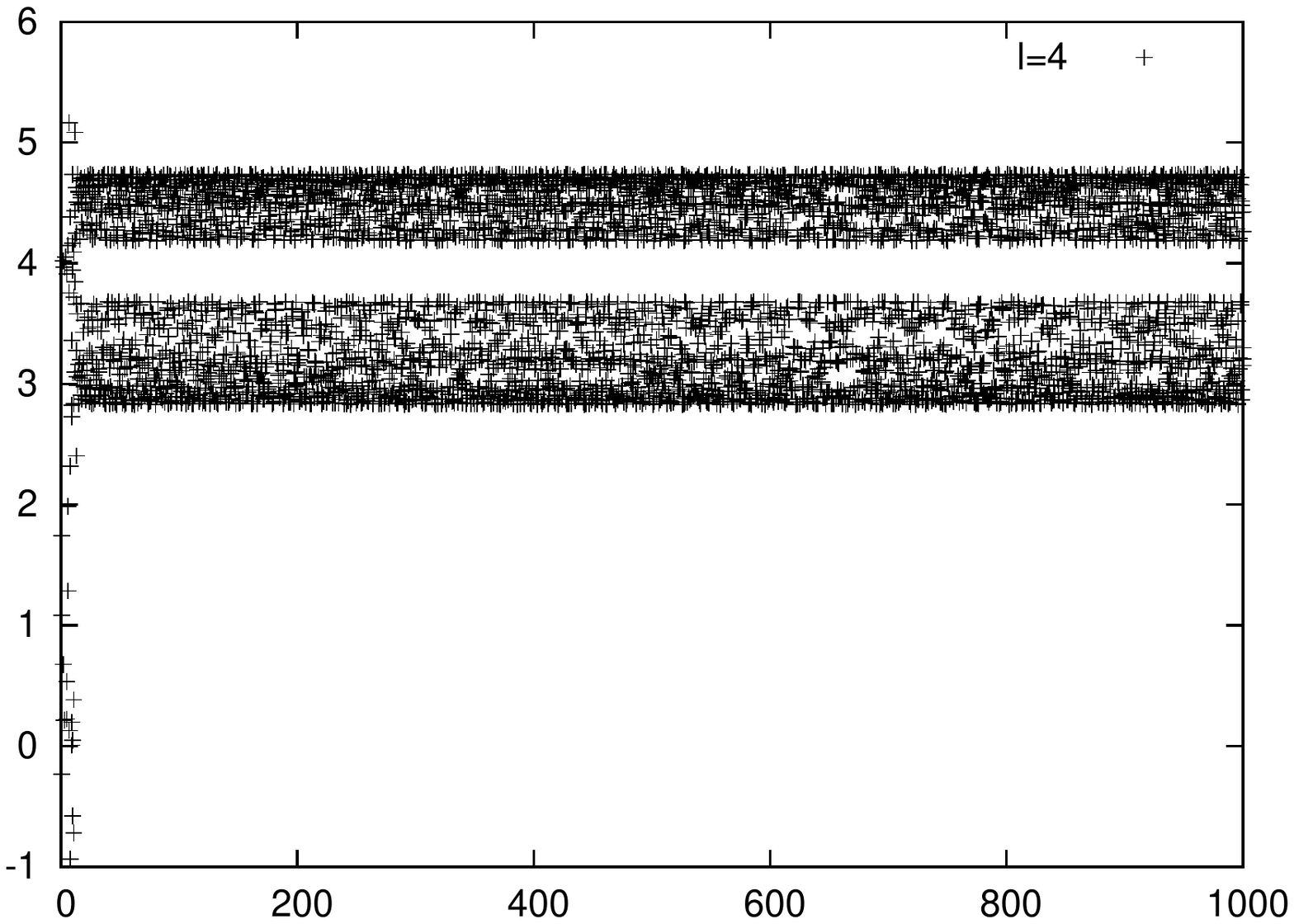}
\caption{Five runs for the discrete uniform distribution $l=1,3,4$, $x_0=0.3$.}
\label{figure20}
\end{figure}
\ee
\end{example}

Finally, we deal with destabilization of a positive equilibrium. 

\begin{example}
\label{ex:forlena7}
Consider the equation
 \[
     x_{n+1}=  F(x_n)+(x_n-3)\sigma_4(x_n-3)\xi_{n+1}, \quad x_0\in \mathbb R,
\]
where 
\begin{equation}
\label{def:FinstabK3}
F(x):=\left \{
  \begin{array}{cc}
 \frac {8.25x}{7.25+(x-2)^2} ,& \quad x < 3, \vspace{2mm} \\
   \frac {3(x-3)}{2+(x-6)^2}+3, & \quad x\ge 3.
  \end{array}
\right.
\end{equation}
The function $F$ in \eqref{def:FinstabK3} has five fixed points, $0,1,3,5,7$, where zero and $K=3$ are stable only.
It is easy to check that $F(x)$, $f(x)=F(x)/x$ and $F'(x)$ are continuous on $[0,\infty)$.
We destabilize $K=3$, where $\xi$ is as in \eqref{def:dud} with $l=2$, so, by Remark \ref{rem:cond2122} destabilization is  achieved if $\sigma_4(x)$  satisfies \eqref{ineq:sdud}. However,  Fig.~\ref{figure11} illustrates that  \eqref{ineq:sdud} is not sharp:
%We apply \eqref{ineq:sdud} with the maximum in the right-hand side. 
%$\sigma_b(x)=\sigma_4(x)$.
%We apply \eqref{ineq:sdud} with the maximum in the right-hand side. Obviously $f_{\max}=f(2)=15/14$,
%and for $l=2$, $\sigma>3 \times 1.682353 \approx 5.047$ should work.
%In Fig.~\ref{figure11} we illustrate that this estimate is not sharp: 
for $\sigma = 1.2$ the equilibrium $K=3$ is still stable (left),  for $\sigma=1.8$ we observe bistability and destabilization of $K=3$ for $\sigma=1.9$.

%set term postscript portrait
%set size 0.8,0.4
%set output "BH_combo_sig_1_2_x0_0_8_l_2.ps"
%plot [0:5000] [-5:15] "BH_combo_sig_1_2_x0_0_8_l_2"
%set output "BH_combo_sig_1_8_x0_0_8_l_2.ps"
%plot [0:5000] [-5:15] "BH_combo_sig_1_8_x0_0_8_l_2"
%set output "BH_combo_sig_1_9_x0_0_8_l_2.ps"
%plot [0:5000] [-5:15] "BH_combo_sig_1_9_x0_0_8_l_2"
%set output "BH_combo_sig_1_9_x0_2_0_l_2_cut_2.ps"
%plot [0:5000] [-2:10] "BH_combo_sig_1_9_x0_2_0_l_2_cut_2"
%set output "BH_combo_sig_1_9_x0_2_0_l_2_cut_0_8.ps"
%plot [0:5000] [-2:10] "BH_combo_sig_1_9_x0_2_0_l_2_cut_0_8"
%set output "BH_combo_sig_1_9_x0_2_0_l_2_cut_0_5.ps"
%plot [0:5000] [-2:10] "BH_combo_sig_1_9_x0_2_0_l_2_cut_0_5

%set output "BH_combo_sig_1_9_x0_2_0_l_2_cut_1_to_4.ps"
%plot [0:5000] [-2:10] "BH_combo_sig_1_9_x0_2_0_l_2_cut_1_to_4"
%set output "BH_combo_sig_1_9_x0_2_0_l_2_cut_2_5_to_3_5.ps"
%plot [0:5000] [-2:10] "BH_combo_sig_1_9_x0_2_0_l_2_cut_2_5_to_3_5"
%set output "BH_combo_sig_1_9_x0_2_0_l_2_cut_2_2_to_3_8.ps"
%plot [0:5000] [-2:10] "BH_combo_sig_1_9_x0_2_0_l_2_cut_2_2_to_3_8"
%set output "BH_combo_sig_1_9_x0_2_0_l_2_cut_1_5_to_3_5.ps"
%plot [0:5000] [-2:10] "BH_combo_sig_1_9_x0_2_0_l_2_cut_1_5_to_3_5"
%set output "BH_combo_sig_1_9_x0_2_0_l_2_cut_2_to_4.ps"
%plot [0:5000] [-2:10] "BH_combo_sig_1_9_x0_2_0_l_2_cut_2_to_4"

\begin{figure}[ht]
\centering
\vspace{-25mm}
\includegraphics[height=.3\textheight]{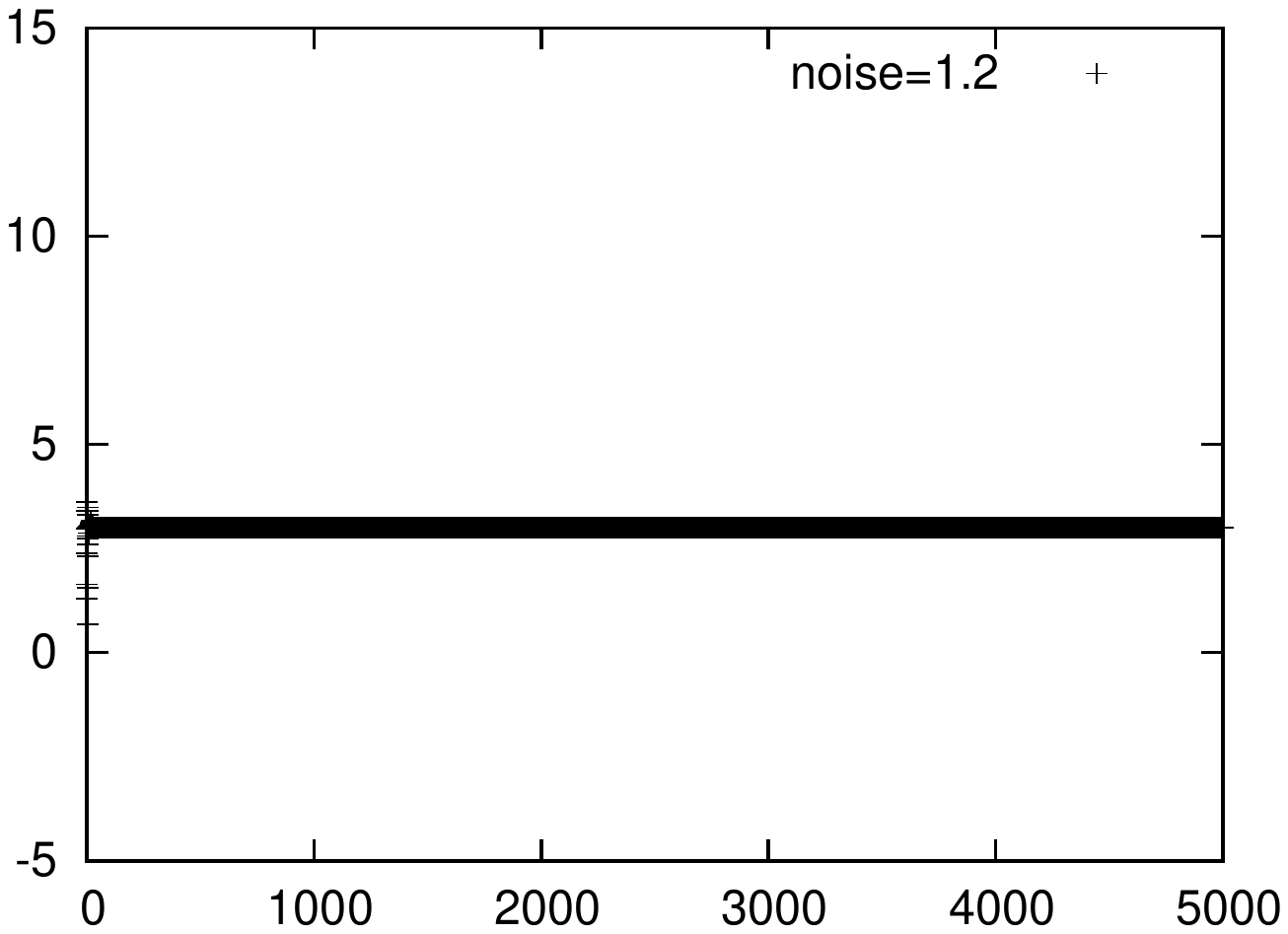}
\includegraphics[height=.3\textheight]{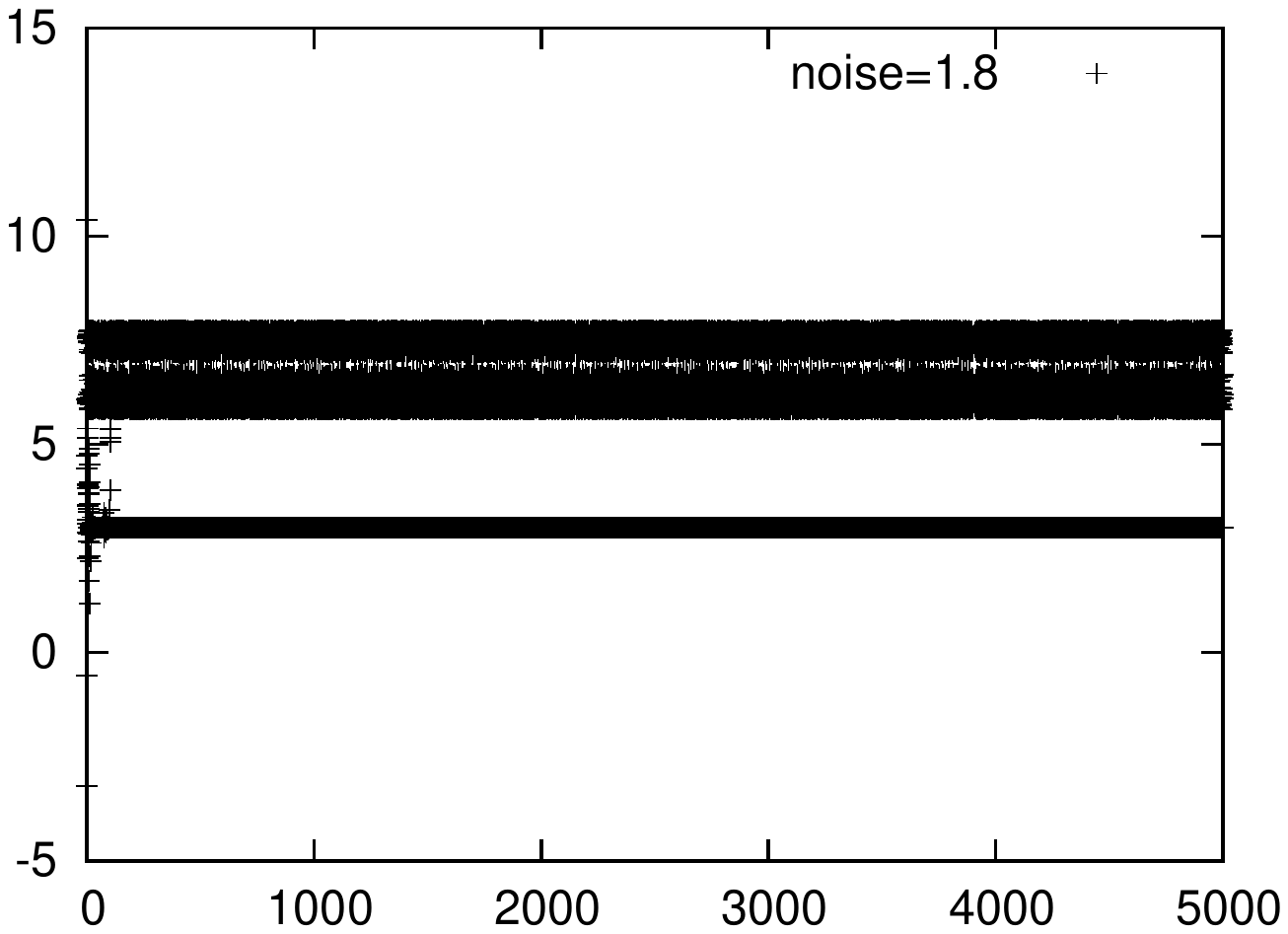}
\includegraphics[height=.3\textheight]{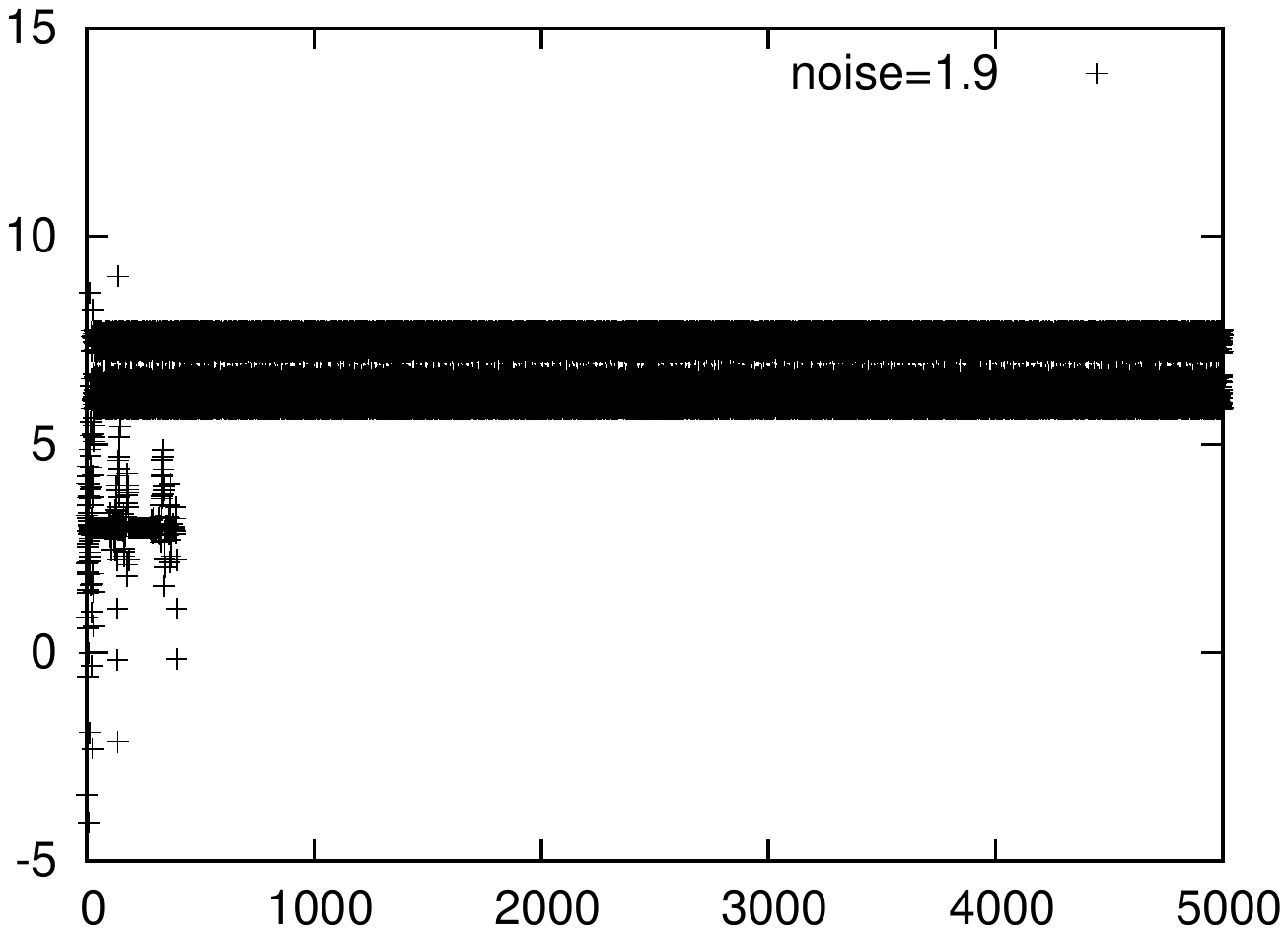}
\caption{
Five runs for the perturbed modified Beverton-Holt equation with $F$ as in \eqref{def:FinstabK3}, noise  \eqref{def:dud} with $l=2$,
$x_0=0.8$, and (left) $\sigma=1.2$ with fast
convergence to the positive equilibrium $K=3$, (middle) $\sigma=1.8$  with bistability and (right) $\sigma=1.9$ with destabilization.
}  
\label{figure11}
\end{figure}

In Fig.~\ref{figure11a}, we explore the possibility of introducing noise for $\sigma=1.9$ not everywhere on $(-\infty,4]$ but, first,
on $[1,4]$, then on 
%$[2,4]$, 
$[1.5,3.5]$, $[2.2,3.8]$ and, finally, on $[2.5,3.5]$ only.  
Here $x_0=2.0$, to avoid immediate attraction to zero for 
$x_0=0.8$.

\begin{figure}[ht]
\centering
\vspace{-25mm}
\includegraphics[height=.25\textheight]{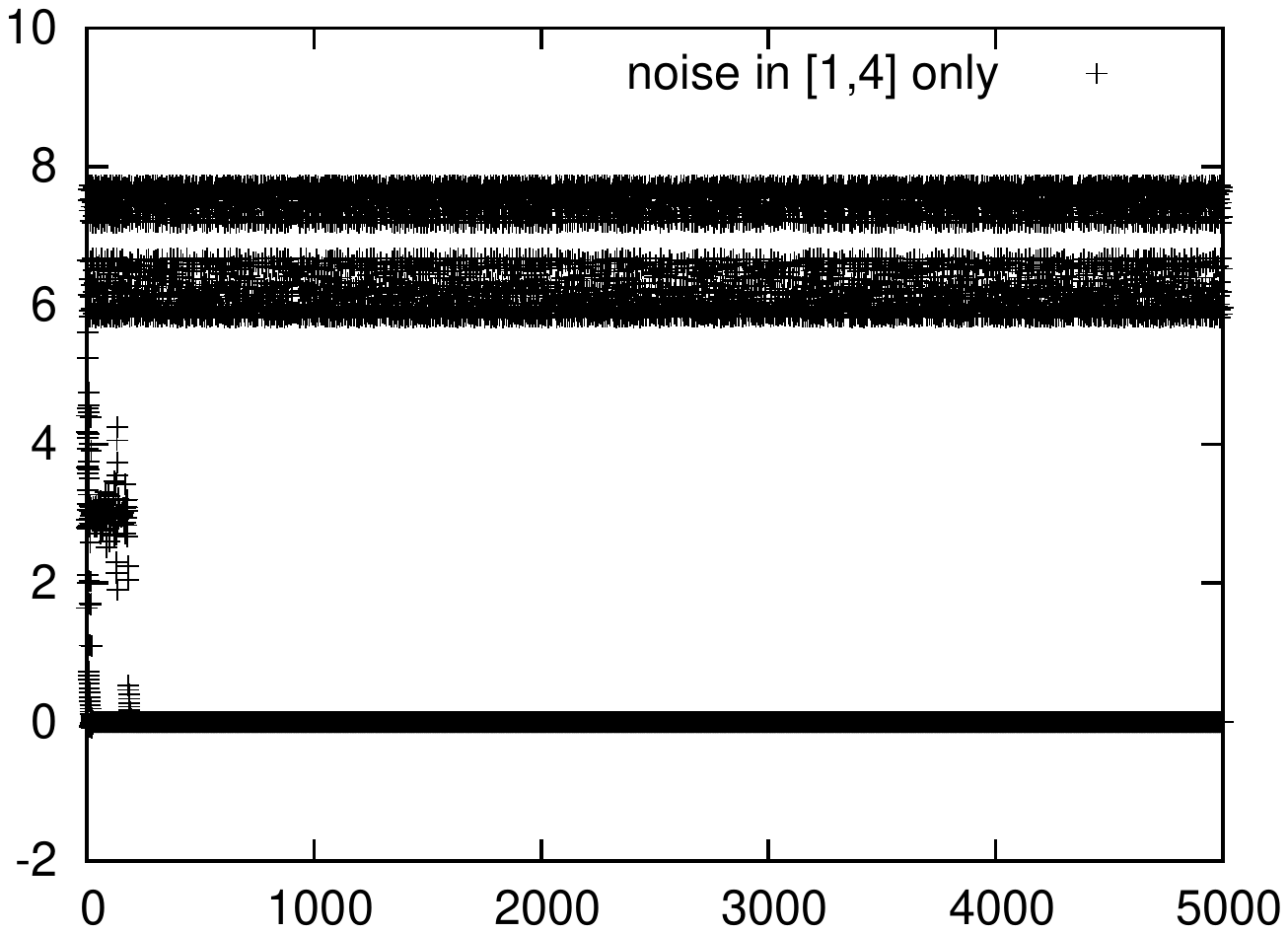}
\includegraphics[height=.25\textheight]{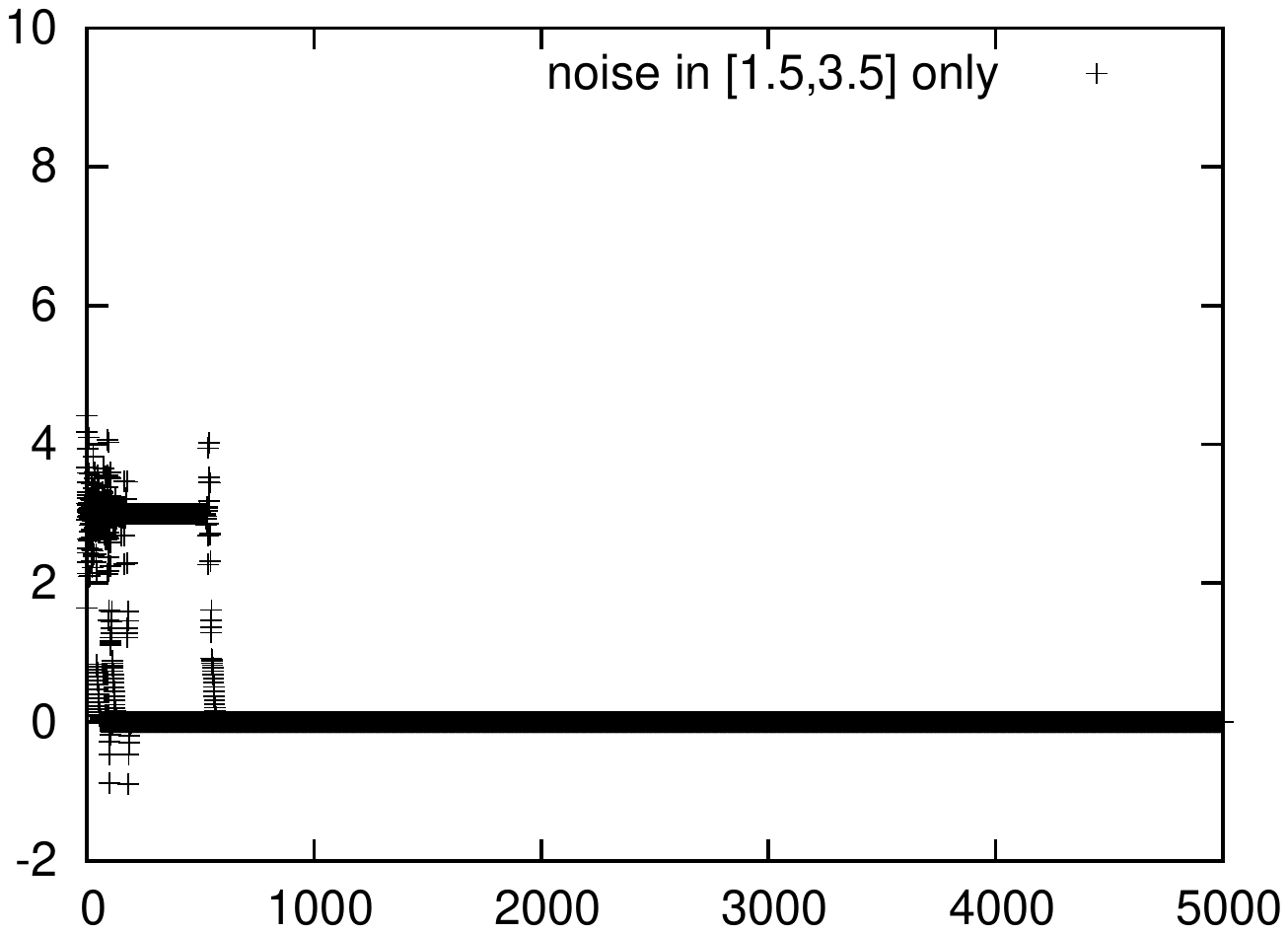}
\includegraphics[height=.25\textheight]{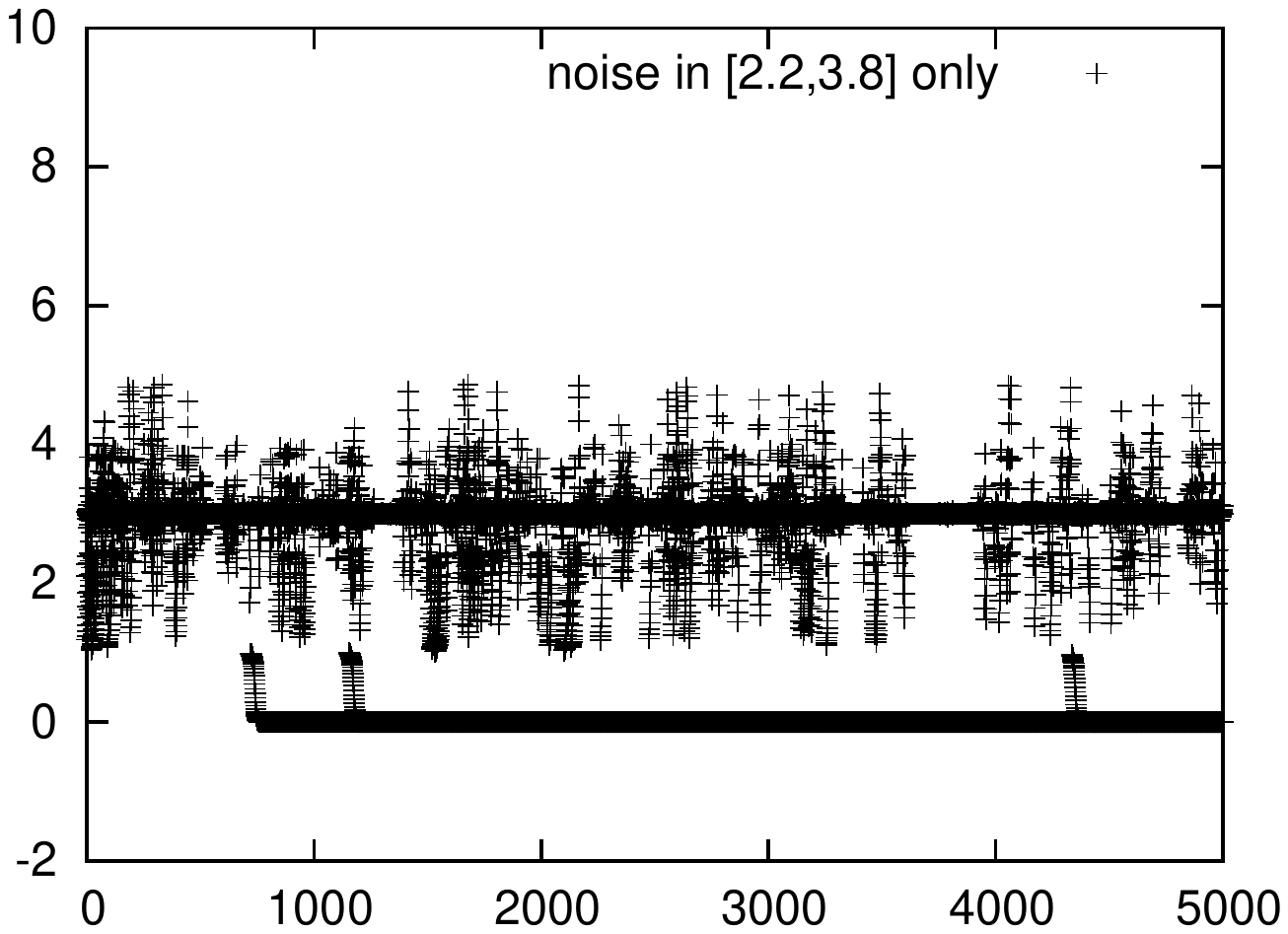}
\includegraphics[height=.25\textheight]{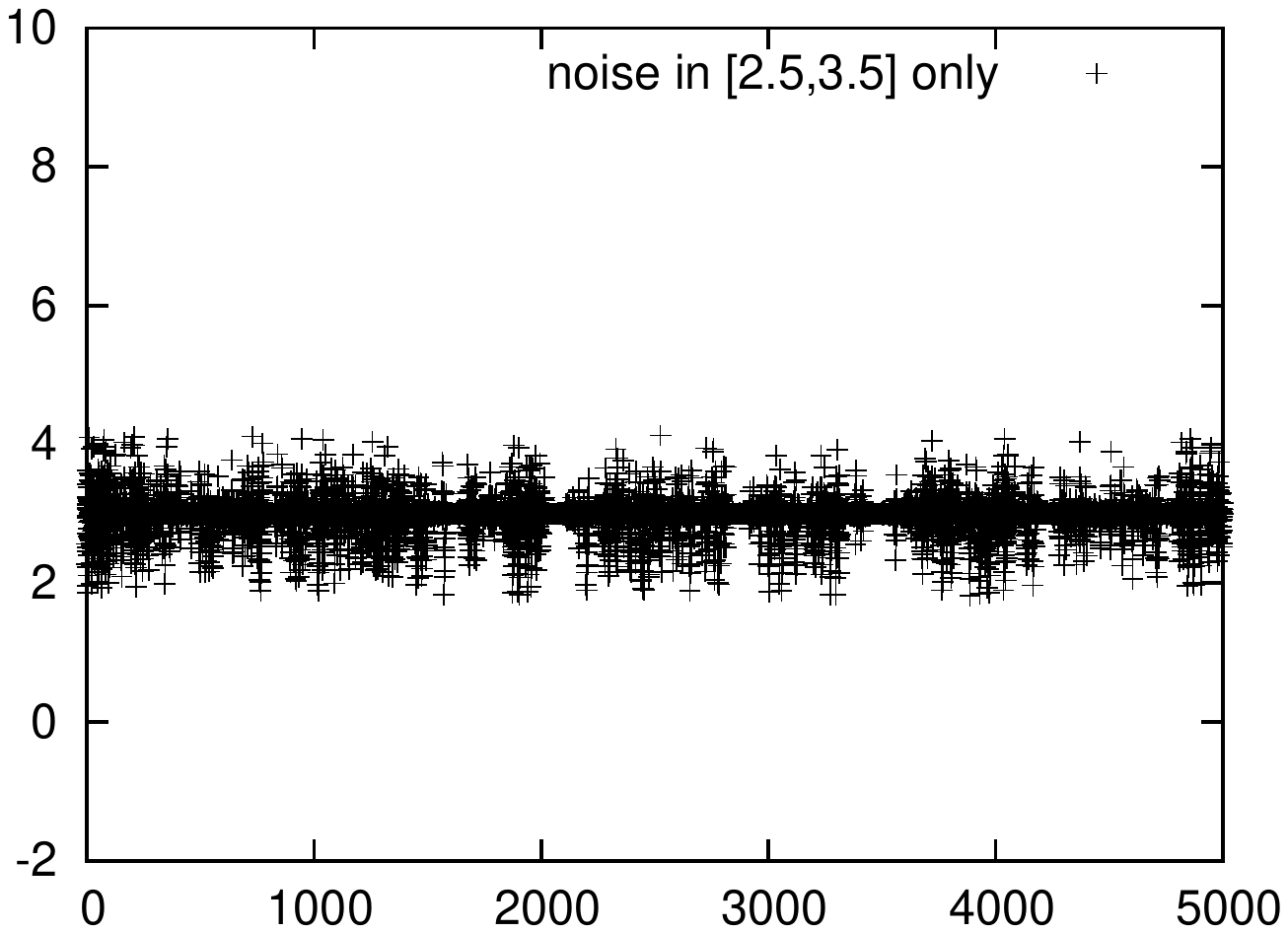}
\caption{
Five runs for the perturbed modified Beverton-Holt equation with $F$ as in \eqref{def:FinstabK3}, noise  \eqref{def:dud} with $l=2$,
$x_0=2$, and zero noise outside $[1,4]$, 
%$[2,4]$, 
$[1.5,3.5]$, $[2.2,3.8]$  and $[2.5,3.5]$.
}
\label{figure11a}
\end{figure}
\end{example}

\section{Summary}
\label{sec:sum}

In the present paper, we obtained sufficient conditions for stabilization of an unstable equilibrium, as well as for destabilization of 
an otherwise stable equilibrium. Estimates are closely connected to a class of noise distributions. Numerical simulations not only 
illustrate applicability of obtained noise lower bounds but also evaluate sharpness of results: in certain cases, sufficient estimates are close to necessary,
in other cases for significantly smaller noise amplitudes, either stabilization or destabilization is achieved in a reasonable time.

Two possible further directions of this study include, first, exploring transient behaviour of solutions and, second, investigating equations with more 
than two equilibrium points. The first problem is very important in practical applications, for example, to develop explicit estimates of 
time when a solution with a certain $x_0$ enters a given neighbourhood of an equilibrium with a probability exceeding 
$1-\gamma$, for any 
prescribed $\gamma \in (0,1)$. Some pilot simulations  in the second direction are provided in Example~\ref{ex:forlena7}, where among the 
zero and four positive equilibrium points, two (0 and 3) were stable. Destabilization of the stable positive equilibrium, when achieved, 
either attracted a solution to the zero equilibrium or took a solution out of the domain of attraction of the positive equilibrium. 

Also, so far we considered stabilization and destabilization of equilibrium points, though cyclic behaviour is quite common in observations of population dynamics, and treating these cycles rather than just fixed points is an essential direction of further research.

Another important extension of our results is stabilization of otherwise unstable fixed points in models described by systems of difference equations. 

%%%%%%%%%%%%%%%%%%

\section*{Acknowledgments}

The authors are grateful to Prof. Stefan Siegmund who attracted our attention to this problem, as well as to organizers of the two 
International Conferences on Difference Equations and Applications (ICDEA-2017 in Timi\c{s}oara, Romania 
and ICDEA-2018 in Dresden, Germany) who supported both authors and provided a collaboration opportunity to start and progress the current  research. The first author was also partially supported by the NSERC research grant RGPIN-2015-05976. The authors are also grateful to the anonymous reviewers whose valuable comments significantly contributed to a better presentation of the results of the paper.

%%%%%%%%%%%%
%%%%%%%%%%%%%%%%%%

%%%%%%%%%%%%%%%%%%%%%

%%%%%%%%%%%%%
\section{Appendix}
\label{sec:Ap}

\begin{proof} [Proof of Lemma \ref{lem:topor}]
Consider a sequence of events $(A_k)_{k\in \mathbb N}$, where
\[
A_k=\left\{\xi_{k J+i}\in [a_i, b_i], \quad  i=1, 2,  \dots, J\right\}.
\]

The events $A_k$ are independent due to independence of $\xi_n$ and since $A_k$ and $A_m$, for $k\neq m$, 
contain $\xi$ with different indices.  
Also, by independence of $\xi_n$ and $p_i>0$, $i=1, 2,  \dots, J$, for each $k\in \mathbb N$,
$\displaystyle
\mathbb P\{A_k\}=\prod_{i=1}^{ J}p_i>0$, and thus $\displaystyle \sum_{k=1}^\infty\mathbb P\{A_k\}=\infty$.
So, by Borel-Cantelli Lemma~\ref{lem:BC}, \,\, 
$\displaystyle
\mathbb P\{A_k \,\,  \text{occurs infinitely often}\}=1$.
Denoting
$
\displaystyle
\mathcal N:=\min\left\{ k: A_k \,\, \text{occurs} \right\}
$,
we conclude the proof.
\end{proof}

\begin{proof}[Proof of Lemma \ref{lem:genstab0}]
A solution of \eqref{eq:addstab} can be presented as 
  \begin{equation*}
  \label{eq:stabproductrepr}
  x_{n+1}=x_0\prod_{i=0}^{n}f(x_i)\prod_{i=0}^{n}\left( 1+\sigma \xi_{i+1} \right), 
   \end{equation*}
  so, by Assumption \ref{as:fBound}, we have
   \begin{equation}
  \label{eq:stabproductest11}
  |x_{n+1}|= |x_0|\prod_{i=0}^{n}|f(x_i)|e^{\sum_{i=0}^{n}\ln |1+\sigma \xi_{i+1}|}\le |x_0|H^{n+1}e^{\sum_{i=0}^{n}\ln |1+\sigma \xi_{i+1}|}. 
   \end{equation}  
 By Kolmogorov's Law of Large Numbers, Lemma~\ref{thm:Kolm}, with $v_n=\ln |1+\sigma \xi_{n+1}|$, we obtain that for any 
$\varepsilon\in (0, \eta)$ there exists an a.s. finite $\mathcal N=\mathcal N(\varepsilon)$ such that, for all $n\ge \mathcal N$, ~
$-(n+1)(\eta+\varepsilon)\le \sum_{i=0}^{n}\ln |1+\sigma \xi_{i+1}|\le -(n+1)(\eta-\varepsilon)$.

Substituting this into \eqref {eq:stabproductest11}, we conclude that, for all $n\ge \mathcal N$, a.s.,
 \begin{equation*}
  \label{eq:stabproductest2}
  |x_{n+1}|\le |x_0|H^{n+1}e^{ -(n+1)(\eta-\varepsilon)}=|x_0|\left[He^{ -(\eta-\varepsilon)}\right]^{n+1}  . 
  \end{equation*}  
 Choosing  $\varepsilon< \varepsilon_1:= \eta - \ln H$
we get, for all $n\ge \mathcal N$, a.s.,
\[
 |x_{n+1}|\le |x_0|\left[e^{ \ln H-\eta+\varepsilon}\right]^{n+1}= |x_0|e^{ -(\varepsilon_1-\varepsilon)(n+1)},
  %< |x_0|\left[e^{ \ln H-\eta+\frac {\eta- \ln H}2}\right]^{n+1}=
% |x_0|\left[e^{-\frac {\eta- \ln H}2}\right]^{n+1}.
  \]
which implies  the necessary result.
\end{proof}

%%%%%%%%%%%%%%%%%%%
\begin{proof}[Proof of Lemma~\ref{lem:Eln123}]
(i) We have
\begin{equation*}
%\label{calc:Ekxicont}
\begin{split}
&\mathbb E\ln (1+\xi_{i+1})=\frac{2s+1}2\int_{-1}^1\ln(1+x)x^{2s}dx\\= &\frac{2s+1}2\left[\ln (1+x) \frac{x^{2s+1}}{2s+1}\biggr|_{-1}^1-\frac1{2s+1}\int_{-1}^1\frac{x^{2s+1}+1-1}{1+x}dx\right]\\= &
\frac{1}2\left[2\ln 2-\lim_{x\to -1}\ln (1+x)(x^{2s+1}+1)-\int_{-1}^1\left(x^{2s}-x^{2s-1}+\dots +1\right]dx\right)
%\\&=
%\frac{1}2\left[2\ln 2-\left[\frac{x^{2s+1}}{2s+1}-\frac{x^{2s}}{2s}+\dots +x\right]\biggr|_{-1}^1\right]\\&=
%\ln 2-\frac 12\left[ \frac{1}{2s+1}-\frac{1}{2s}+\dots +1-\left(-\frac{1}{2s+1}-\frac{1}{2s}+\dots -1 \right)\right]
\\= &
\ln 2-\left[ \frac{1}{2s+1}+\frac{1}{2s-1}+\dots +1\right],
\end{split}
\end{equation*}
since $\lim_{x\to -1}[(x+1)\ln (1+x)]=0$ and, therefore, 
\[
\lim_{x\to -1}\ln (1+x)(x^{2s+1}+1)=\lim_{x\to -1}[\ln (1+x)(x+1)]\left[x^{2s}-x^{2s-1}+\dots +1\right]=0.
\]
As
$\displaystyle
\sum_{s=0}^\infty\frac 1{2s+1}=\infty$,
for each $H$ we can find $s\in n\in {\mathbb N}_0$ such that
\[
\frac{1}{2s+1}+\frac{1}{2s-1}+\dots +1>\ln (2H),
\]
which implies
\[
e^{-\mathbb E\ln (1+\xi_{i+1})}=e^{-\ln 2+\left[ \frac{1}{2s+1}+\frac{1}{2s-1}+\dots +1\right]}>\frac 12e^{\ln (2H)}=H.
\]
%%%%%%%%%%%%%%%%%%%%%%%%%%%

(ii) (a) Consider discrete distribution \eqref{def:dud} with $2l$ states. For some $\sigma\in (0, 1)$, %we get
\begin{equation}
\label{calc:Ekxidiscr}
\begin{split}
-\eta= & \mathbb E\ln (1+\sigma \xi_{i+1})=\frac1{2l}\sum_{i=0}^{2l-1}\ln \left[1+\sigma\left(-1+\frac{2i}{2l-1} \right)\right]\\=
%& \frac1{2l} \left[\ln (1-\sigma^2)+\sum_{i=1}^{2l-2}\ln \left[1+\sigma\left(-1+\frac{2i}{2l-1} \right)\right]\right]
%\\=
& \frac1{2l} \left[\ln (1-\sigma^2)+\sum_{j=1}^{l-1}\ln \left[1-\sigma^2\left(-1+\frac{2j}{2l-1} \right)^2\right]\right],
\end{split}
\end{equation}
as 
$\displaystyle
\ln \left[1+\sigma\left(-1+\frac{2j}{2l-1} \right)\right]+\ln \left[1-\sigma\left(-1+\frac{2j}{2l-1} \right)\right]$ 
$\displaystyle =\ln \left[1-\sigma^2\left(-1+\frac{2j}{2l-1} \right)^2\right]
$
for all $j=1, \dots, l-1$. Note that all the expressions under the logarithms are positive since $\sigma \in (0,1)$ and
$\displaystyle
1+\sigma\left(-1+\frac{2i}{2l-1}\right)>1-\sigma >0$.
Every term of  the sum in  the last line in \eqref{calc:Ekxidiscr} is negative, so $
-\eta=\mathbb E\ln (1+\sigma \xi_{i+1})< \frac1{2l} \ln (1-\sigma^2). $
Since 
$\displaystyle 
\lim_{\sigma\to 1^-}\frac1{2l} \ln (1-\sigma^2)=-\infty$,
for each $H>1$ we can choose $\sigma=\sigma_{l, H}$  sufficiently close to 1 such that $\eta>\ln H. $ Indeed,
$\sigma>\sqrt{1-H^{-2l}}$ implies $\eta>\ln H$.
%\[
%\frac1{2l} \ln (1-\sigma^2)<-\ln H \Leftrightarrow \ln (1-\sigma^2)<-\ln H^{2l}\Leftrightarrow 1-\sigma^2<H^{-2l}\Leftrightarrow
% \sigma>\sqrt{1-H^{-2l}}.
%\]
So we can set $\sigma_{l, H}\in \left(\sqrt{1-H^{-2l}}, 1\right)$. Here without loss of generality we assume that $H>1$, see Remark \ref{rem:Hlambda}.

\medskip

%%%%%%%%%%%%%%%%%%%%%%%%%
%%%%%%%%%%%%%%%%%%%%%%%%%%%

(ii) (b) Consider now piecewise continuous uniform distribution \eqref{def:contdud} with the parameters 
$l\in \mathbb N$ and $\delta>0$. We assume again  that $\sigma\in (0, 1)$. Further,   
\begin{equation}
  \label{calc:Ealphacontdud}
  \begin{split}
\mathbb E \ln [1+\sigma \xi_{n+1}]  = & \frac 1{2\delta(2l-1)}\sum_{i=1}^{2l-2}\int\limits_{-1+\frac 
{2i}{2l-1}-\delta}^{  -1+\frac 
{2i}{2l-1}+\delta}\ln (1+\sigma s)ds
\\&
+\frac1{2\delta (2l-1)}\int\limits_{-1}^{-1+\delta}\ln(1+\sigma s) ds
%\\&
+\frac1{2\delta (2l-1)}\int\limits_{1-\delta}^{1}\ln( 1+\sigma s) ds.
\end{split}
\end{equation}
%Note that $ \int \ln xdx=x(\ln x-1). $ 
By the Mean Value Theorem,
\\
$\displaystyle \frac1{2\delta }\int\limits_{-1}^{-1+\delta}\ln(1+\sigma s) ds
=\frac{\ln(1-\sigma +\sigma \delta \theta_1)}{2}$
~and~ 
$\displaystyle \frac1{2\delta}\int\limits^{1}_{1-\delta}\ln(1+\sigma s) ds
= \frac{\ln(1+\sigma -\sigma \delta \theta_2)}{2}$,
for some $\theta_1, \theta_2\in [0,1]$. The sum of the two  above terms equals
\begin{equation*}
\frac{\ln [(1+\sigma -\sigma \delta \theta_2)(1-\sigma +\sigma \delta \theta_1)] }{2\delta}
=\frac{\ln\biggl[1-\sigma^2 +\delta \tau\left(\delta, \sigma, \theta_1, \theta_2 \right)\biggr]}{2}
\end{equation*} 
and, since $\delta, \sigma, \theta_1, \theta_2\in (0, 1)$,  we have the estimate 
$\displaystyle
|\tau(\delta, \sigma, \theta_1, \theta_2)|\le 3$.

Thus, for each $H>0$, by making  $\sigma$  close to 1 and $\delta$ small, we can get 
\[
\frac{\ln(1-\sigma^2 +\delta \tau)}{2(2l-1)}<-\ln H.
\]
In a similar way, by applying the Mean Value Theorem, grouping  the sum terms of the integrals in \eqref{calc:Ealphacontdud} and making $\delta$ small enough, we  can guarantee that each group  is negative:
$$
\ln \left[1+\sigma\left(-1+\frac{2j}{2l-1} +\delta \theta_{1,j}\right)\right]+\ln \left[1-\sigma\left(-1+\frac{2j}{2l-1}+\delta \theta_{2,j} \right)\right]$$ $$=\ln \left[1-\sigma^2\left(1-\frac{2j}{2l-1} \right)^2+\delta \tau_{j}\right]<0,
$$
where $|\tau_j|=|\tau_j(\delta, \sigma, \theta_{1,j}, \theta_{2,j})| \leq 5$. Thus, for $1-\sigma$ and $\delta$ small enough, we have
\[
\mathbb E \ln (1+\sigma \xi_{n+1}) \le \frac{\ln(1-\sigma^2 +\delta \tau)}{2(2l-1)}<-\ln H.
\]
\end{proof}

%%%%%%%%%%

\begin{proof}[Proof of Lemma ~\ref{lem:destab}]
We have, for $\alpha\in (0, 1]$, for all $n\in n\in {\mathbb N}_0$,
\begin{equation}
\label{calc:unstab2}
%\begin{split}
|x_{n+1}|^{\alpha}
%=|x_n|^{\alpha}|f(x_n)+\sigma(x_n) \xi_{n+1}|^{\alpha} %\\
%=&|x_n|^{\alpha}\frac{\mathbb E\left[ |f(x_n)+\sigma(x_n) \xi_{n+1}|^{-\alpha}\biggr| \mathcal F_n\right]}{\mathbb E\left[ |f(x_n)+\sigma(x_n) \xi_{n+1}|^{-\alpha}\biggr| \mathcal F_n\right]  |f(x_n)+\sigma \xi_{n+1}|^{-\alpha}}\\
%=&
=|x_0|^{\alpha}\frac 1{M_n} \prod_{i=0}^n\frac 1{\mathbb E\left[ |f(x_i)+\sigma(x_i) \xi_{i+1}|^{-\alpha}\biggr| \mathcal F_i\right] },
%\end{split}
\end{equation}
where
\begin{equation}
\label{def:mart}
M_n:=\prod _{i=0}^n\frac { |f(x_i)+\sigma(x_i) \xi_{i+1}|^{-\alpha}}{\mathbb E\left[ |f(x_i)+\sigma(x_i) \xi_{i+1}|^{-\alpha}\biggr| \mathcal F_i\right] }.
\end{equation}
By the assumptions  of the lemma all the expressions in \eqref {calc:unstab2} and \eqref {def:mart} are well defined, 
by Lemma~\ref{lem:martprod}, $(M_n)_{n\in \mathbb N}$ is a nonnegative $\mathcal F_n$-martingale. So Lemma \ref{lem:nnM}  implies that $M_n$ converges to an a.s. finite limit.

By condition \eqref{cond:destab22},  each factor 
 $
  \left(\mathbb E\left[ |f(x_i)+\sigma(x_i) \xi_{i+1}|^{-\alpha}\biggr| \mathcal F_i\right] \right)^{-1}
   $
exceeds 1, and  $\frac 1{M_n}$ converges to an a.s. nonzero limit, thus, the solution $x_n$  can be estimated from below by an a.s. 
finite positive random variable.

\end{proof}

%%%%%%%%%%%%%%%%%%%%%%%%
\begin{proof}[Proof of Theorem~\ref{lem:Ealpha123}]
 
 (i).   Note that since $\xi_n$  are continuous noises and, for each $n\in \mathbb N$,  $\xi_{n+1}$  is independent of $x_n$, the probability that 
$|f(x_n)+\sigma(x_n) \xi_{n+1}|=0$ is zero, so the expression $|f(x_i)+\sigma(x_i) \xi_{i+1}|^{-\alpha}$ is well defined.

{\it Case $s=0$ (uniform distribution).}  We have 
\begin{equation*}
%\label{calc:unifcont}
\begin{split}
&\mathbb E\left[ |f(x_i)+\sigma(x_i) \xi_{i+1}|^{-\alpha}\biggr| \mathcal F_i\right]=\frac 12\int_{-\frac f\sigma}^1\frac{dv}{(f+\sigma v)^{\alpha}}+\frac 12\int_{-1}^{-\frac f\sigma}\frac{dv}{(-f-\sigma v)^{\alpha}}\\&=
\frac 1{2\sigma}\left[\frac{(f+\sigma)^{1-\alpha}}{1-\alpha} +\frac{(\sigma-f)^{1-\alpha}}{1-\alpha}  \right]=\frac {\sigma^{1-\alpha}}{2\sigma(1-\alpha)}\left[\left(1+\frac f\sigma\right)^{1-\alpha}+\left(1-\frac f \sigma\right)^{1-\alpha}\right],
\end{split}
\end{equation*}
where  we omit $x_i$  for simplicity and set $f:=f(x_i)$,  $\sigma:=\sigma(x_i)$. Applying  the inequality
%\[
$\displaystyle (1+y)^\alpha\le 1+\alpha y$ for $-1<y<\infty$ and $0<\alpha<1$,
%\]
we get
\[
\left(1+\frac f\sigma\right)^{1-\alpha}+\left(1-\frac f \sigma\right)^{1-\alpha}\le1+(1-\alpha)\frac f\sigma+1-(1-\alpha)\frac f\sigma= 2.
\]
This implies that
$\displaystyle
\mathbb E\left[ |f(x_i)+\sigma(x_i) \xi_{i+1}|^{-\alpha}\biggr| \mathcal F_i\right]\le \frac {1}{\sigma^\alpha(1-\alpha)}$. 
Then  \eqref{cond:destab22} holds for 
\[
\sigma(x_i)\ge \frac 1{\sqrt[\alpha]{1-\alpha}}.
\]
As
$\displaystyle
\lim_{\alpha \to 0}(1-\alpha)^{\frac 1\alpha}=e^{-1}$ 
and $f(\alpha)=(1-\alpha)^{\frac 1\alpha}$ decreases in $\alpha$,    
we have $(1-\alpha)^{\frac 1\alpha}<\frac{1}{e}$. Then 
\eqref{cond:destab22} holds if we choose
\begin{equation*}
%\label {cond:fe} 
\sigma(x)>\max \left\{f(x), e\right\}, \quad \forall x\in \mathbb R.
\end{equation*}
Indeed,  if $\inf_{x\in \mathbb R}\sigma (x)>e$ we can find $\alpha\in (0, 1)$ such that 
$\displaystyle
\frac 1{\sqrt[\alpha]{1-\alpha}}\in \left(e, \, \inf_{x\in \mathbb R}\sigma (x)\right)$
and apply Lemma \ref{lem:destab} for this particular $\alpha$.

%%%%%%%%%%%%%%

{\it Case of arbitrary $s\in \mathbb N$.} We have
\begin{equation*}
%\label{calc:unstabs2}
%\begin{split}
%&
\mathbb E\left[ |f(x_i)+\sigma(x_i) \xi_{i+1}|^{-\alpha}\biggr| \mathcal F_i\right]=\frac {2s+1}2\int\limits_{-\frac f\sigma}^1\frac{x^{2s}dx}{(f+\sigma x)^{\alpha}}+\frac {2s+1}2\int\limits_{-1}^{-\frac f\sigma}\frac{x^{2s}dx}{(-f-\sigma x)^{\alpha}}.
%\end{split}
\end{equation*}
Integrating by parts and doing several estimations, we conclude that 
\eqref{cond:destab22} holds if we choose, for some $\alpha\in (0, 1)$,
\begin{equation}
\label{cond:discrcontsigma}
\sigma(x)>\max \left\{f(x), \, \sqrt[\alpha]{\frac{2s+1}{1-\alpha}}\, \right\},~~ \forall x\in \mathbb R.
\end{equation}
Since for each  $s\in \mathbb N$ we have $\lim\limits_{\alpha\to 0} \sqrt[\alpha]{2s+1}=1$, we can decrease $\alpha$ to get 
$
\displaystyle \sqrt[\alpha]{\frac{2s+1}{1-\alpha}} \in \left(e, \, \inf_{x\in \mathbb R}\sigma (x)\right).
$
So condition \eqref{cond:discrcontsigma} holds, if inequality \eqref{cond:fe} is satisfied and $\alpha\in (0, 1)$ is sufficiently small.

%%%%%%%%%%%%%%%%%%%%%%%

 (ii) {\it Case \eqref{def:dud}}.   Assume that \eqref{ineq:sdud} holds, i.e.
%\begin{equation*}
$\displaystyle \sigma(x)> (2l-1)\left[\frac{1+\sqrt{1+4 f^2(x)}}{2}\right]$, ~$\forall x\in \mathbb R$.
%\end{equation*}
%Condition \eqref {ineq:sdud} implies, for each $x\in \mathbb R$  and \frac 1e  
As $\sigma(x)>0$ and $i=0, 1, \dots, l-1$, 
%\begin{equation*}
$\displaystyle \frac {2l-1-2i}{2l-1}\sigma(x)\ge \frac{\sigma(x)}{2l-1}$.
%\end{equation*}
Then,  for each $x\in \mathbb R$, $i=0, 1, \dots, l-1$, and $\xi$ from \eqref{def:dud},
\[
\left| f(x)+\left(1-\frac {2i}{2l-1}\right)\sigma(x)\xi\right|\ge \frac{\sigma(x)}{2l-1}-|f(x)| > \frac{1+\sqrt{1+4 f^2(x)}}{2}-|f(x)|\ge \frac 12.
\]
So  the expression
\begin{equation}
\label{expr:-1}
\left| f(x)+\left(1-\frac {2i}{2l-1}\right)\sigma(x)\xi\right|^{-1}
\end{equation}
is  well defined for each $x\in \mathbb R$, $i=0, 1, \dots, l-1$.   Note that positive solutions of the inequality $z^2-z-f^2(x)>0$ satisfy $z>\frac{1+\sqrt{1+4f^2(x)}}{2}$. 
By \eqref{ineq:sdud} we have that for each $x\in \mathbb R$  and   $i=0, 1, \dots, l-1$,
$$z:=\left(\frac {2l-1-2i}{2l-1}\right)\sigma(x) > (2l-1-2i)\left[\frac{1+\sqrt{1+4 f^2(x)}}{2}\right]\ge \frac{1+\sqrt{1+4 f^2(x)}}{2},
$$ 
which implies that $z=\left(\frac {2l-1-2i}{2l-1}\right)\sigma(x)$ satisfies the  inequality $z^2-z-f^2(x)>0$.
Then,  for each $x\in \mathbb R$ and $i=0, 1, \dots, l-1$  
\begin{equation*}
\begin{split}
&\left[ f(x)+\left(\frac {2l-1-2i}{2l-1}\right)\sigma(x)\xi
\right]^{-1}+\left[\left(\frac {2l-1-2i}{2l-1}\right)\sigma(x)\xi- f(x)\right]^{-1}\\=&\frac {2\left(\frac {2l-1-2i}{2l-1}\right)\sigma (x)}{\left(\frac {2l-1-2i}{2l-1}\right)^2\sigma^2(x)- f^2(x)}<2.
\end{split}
\end{equation*}
And, therefore, for each $n\in \mathbb N$,  
\begin{equation}
\begin{split}
\label{calc:Edud}
&\mathbb E\left[|f(x_n)+\sigma(x_n) \xi_{n+1}|^{-1}\bigl|\mathcal F_n\right]\\ =& \frac 1{2l}\sum_{i=0}^{2l-1}
\left|  f(x_n)+\left(1-\frac {2i}{2l-1}\right)\sigma(x_n)\xi_{n+1}\right|^{-1}
%\frac 1{2l}\sum_{j=0}^{l-1}\left(\left[f(x_n)+\left(\frac {2l-1-2j}{2l-1}\right)\sigma(x_n)\xi_{n+1}\right]^{-1}+ \left[\left(\frac {2l-1-2j}{2l-1}\right)\sigma (x_n)\xi_{n+1} -f(x_n)\right]^{-1}  \right)
<\frac 1{2l}\sum_{j=0}^{l-1} 2=\frac {2l}{2l}=1,
\end{split}
\end{equation}
which proves the result.

%%%%%%%%%%%%%%%%%%

{\it Case \eqref{def:contdud}. } Let $u:[0, H] \to {\mathbb R}$ be defined as
\begin{equation*}
%\label{def:u(f)}
u(f):=\frac {1+\sqrt{1+4f^2}}2,  \quad f\in  [0, H],
\end{equation*}
which is a continuous and increasing on $[0, H]$ function, so its maximum %, $u(H)$, 
%\begin{equation}
%\label{def:u(H)}
$\displaystyle u(H) =\frac {1+\sqrt{1+4H^2}}2$
%\end{equation}
is reached at the right end of the interval  $[0, H]$. Let also
\begin{equation*}
%\label{def:calB}
\mathcal B:=\left\{(f, u): \frac {1+\sqrt{1+4f^2}}2\le u \le 2u(H), \quad f \in [0, H] \right\},%1+\sqrt{1+4f^2}
\end{equation*}
\begin{equation*}
%\label{def:calA}
\mathcal A:=\mathcal B\times [-0. 25, \, 0.25]\times [-0.25,\, 0.25] \subset  \mathbb R^4.
\end{equation*}
Note that, for $f\in [0, H]$ and  $u\ge u(f)$,  we have
\begin{equation}
\label{est:u2f2}
\begin{split}
&u^2-f^2\ge \left[\frac {1+\sqrt{1+4f^2}}2\right]^2-f^2
%=\frac 14+\sqrt{1+4f^2}+\frac 14 +f^2-f^2\\&
=\frac 12+\frac{\sqrt{1+4f^2}}{2} \ge 1.
\end{split}
\end{equation}
Consider the function $G:\mathcal A\to \mathbb R^1$,  defined by 
\begin{equation}
\label{def:G}
G(f,u, y, z):=\frac{u+y}{u^2-f^2+z}, \quad (f,u, y, z)\in \mathcal A.
\end{equation}
By \eqref{est:u2f2} the denominator of $G$ is always positive on $\mathcal A$ and 
\[
G(f,u(f), 0, 0)=1 \quad \mbox{since} \quad  u^2(f)-u(f)-f^2=0.
\]
Since $u(f)$ is the largest  of the two solutions of the equation $u^2-u-f^2=0$, for each $f\in [0, H]$ and $u>u(f)$ we have 
\[
u^2-u-f^2>0, \quad \mbox{which implies} \quad G(f,u, 0, 0)<1.
 \]
Fix some $\lambda\in (0, u(H))$ and consider 
\begin{equation}
\label{def:calBAlambda}
\begin{split}
\mathcal B_\lambda:=&\left\{(f, u): \frac {1+\sqrt{1+4f^2}}2+\lambda \le u\le 2u(H), \quad f\in [0, H] \right\},\\ \mathcal A_\lambda:=&\mathcal B_\lambda\times [-0.25, \, 0.25]\times [-0.25,\, 0.25], \\
 \mathcal C_\lambda:=&\mathcal B_\lambda\times \{ 0\}\times \{0\}\subset \mathcal A_\lambda.
\end{split}
\end{equation}
Based on the above, we conclude that $G(f,u, 0, 0)<1$ on $\mathcal B_\lambda$. Define
\begin{equation*}
%\label{def:eps0}
\varepsilon_{\lambda}:=\min_{(f, u)\in \mathcal B_\lambda}\left\{1-  G(f,u, 0, 0) \right\} >0,
\end{equation*}
since $\mathcal B_\lambda$ is a closed (and therefore compact) set, thus
$\displaystyle
G(f,u, 0, 0)<1-\varepsilon_{\lambda}, \quad (f, u)\in \mathcal B_\lambda$.  Since $A_\lambda$  is also a compact set, the function $G$ is uniformly continuous on $A_\lambda$. 
%This in particular means that 
Thus there exists $\varsigma\in (0, \, 0.25)$ such that
\begin{equation}
\label{est:Glambda1}
G(x, t, y, z)-G(f, u, 0, 0)<\frac {\varepsilon_{\lambda}}2 \mbox{ ~ for ~ } \max\left\{ |x-f|, |t-u|, |y|, |z| \right\} < \varsigma.
\end{equation}
Then, for each $(f, u, y, z)\in \mathcal A_\lambda$, ~
$\displaystyle
G(f, u, y, z)<G(f, u, 0, 0)+\frac {\varepsilon_{\lambda}}2<1-\frac {\varepsilon_{\lambda}}2,
$
as soon as 
$
\max\{ |y|, |z| \} < \varsigma.
$
Let function $\sigma$ satisfy
\begin{equation}
\label{ineq:sdudcont}
2u(H)\ge \frac {\sigma(x)}{2l-1}\ge \frac{1+\sqrt{1+4 f^2(x)}}{2}+\lambda.
\end{equation}
Inequalities \eqref{ineq:sdudcont} and \eqref{def:calBAlambda}
imply  $\left( |f(x)|, \frac {\sigma(x)}{2l-1}\right)\in \mathcal B_\lambda$ for all $x\in \mathbb R$.

Assume first that 
$\displaystyle \delta< \frac1{2l-1}$,
which guarantees that intervals do not overlap.
Now we seek $\delta>0$ so small that on each of the intervals
\[
[-1, -1+\delta], \, [1-\delta, 1], \, \left[-1+\frac {2i}{2l-1}-\delta, \,  -1+\frac {2i}{2l-1}+\delta \right], \quad i=1, \dots l-1,
\]
we have $|f+\sigma \xi|>0$, where $\sigma$ satisfies \eqref{ineq:sdudcont}. Since the minimum value which $|\xi|$ takes is  $\frac 1{2l-1}-\delta$,
we have
\[
|f+\sigma \xi|\ge |\sigma \xi|-|f|\ge (2l-1)\left[ \frac{1+\sqrt{1+4 f^2(x)}}{2}+\lambda\right]\left[ \frac 1{2l-1}-\delta \right]-|f|.
\]
Applying the inequality $\frac{1+\sqrt{1+4 f^2(x)}}{2}>|f|+\frac 12$, we arrive at 
\begin{equation*}
\begin{split}
&\left[ \frac{1+\sqrt{1+4 f^2(x)}}{2}+\lambda\right]\left[ 1-\delta (2l-1)\right]-|f|
%\ge \left[|f|+\frac 12+\lambda\right]\left[ 1-\delta (2l-1)\right]-|f|\\&
%=-|f| \delta (2l-1)+ \left[\frac 12+\lambda\right]\left[ 1-\delta (2l-1)\right]
\\
\ge & -H\delta (2l-1)+ \left[\frac 12+\lambda\right]\left[ 1-\delta (2l-1)\right].
\end{split}
\end{equation*}
From the above we get an estimation on $\delta$ which guarantees positivity of $|f+\sigma \xi|$:
\begin{equation*}
\begin{split}
&\delta <\frac{\frac 12+\lambda}{(2l-1)\left[H+\frac 12+\lambda\right]}=\frac 1{(2l-1)\left[H( \frac 12+\lambda)^{-1}+1\right]}.
\end{split}
\end{equation*}
When $x$ is a solution to \eqref{eq:sdnoise}, for each $n\in \mathbb N$ we have
\begin{equation}
\begin{split}
\label{calc:Edudcont}
&\mathbb E\left[|f(x_n)+\sigma(x_n) \xi_{n+1}|^{-1}\bigl|\mathcal F_n\right]= \frac 1{2\delta(2l-1)}\sum_{i=1}^{2l-2}\int_{-1+\frac {2i}{2l-1}-\delta}^{  -1+\frac {2i}{2l-1}+\delta} \left|  f+\sigma s\right|^{-1}ds\\&+\frac1{2\delta (2l-1)}\int_{-1}^{-1+\delta}\left|  f+\sigma s\right|^{-1} ds+\frac1{2\delta (2l-1)}\int_{1-\delta}^{1} \left|  f+\sigma s\right|^{-1}ds.
\end{split}
\end{equation}
Here we set again $f:=f(x_n)$, $\sigma:=\sigma(x_n)$ and get by the Mean Value Theorem
\begin{equation*}
\begin{split}
&\frac1{2\delta (2l-1)}\left[\int_{-1}^{-1+\delta}\left|  f+\sigma s\right|^{-1} ds+\int_{1-\delta}^{1} \left|  f+\sigma s\right|^{-1}ds\right]\\
= &\frac\delta{2\delta (2l-1)}\left[\left[ - f-\sigma (-1+\delta \theta_1)\right]^{-1}+\left[  f+\sigma (1-\delta \theta_2)\right]^{-1} \right]\\
%=\frac 1{2(2l-1)}\,\frac{2\sigma (1-\delta  [\theta_1+\theta_2])}{\sigma^2 (1-\delta \theta_2)^2-f^2+\delta f[\delta+\delta_2]}
= & \frac 1{2(2l-1)}\, \frac{2\sigma -\delta \sigma [\theta_1+\theta_2]}{\sigma^2 -f^2 +\delta (-\sigma^2[\theta_1+\theta_2]+\sigma^2\delta\theta_2\theta_1+ f\sigma(\theta_2-\theta_1))}\\ = &
\frac 1{2(2l-1)}\, \frac{2\left[\sigma +\delta \tau_{01}\right]}{\sigma^2 -f^2 +\delta \tau_{02}}=\frac{1}{ 2l-1}G\left(f, \, \sigma , \, \delta  \tau_{01}, \,\delta  \tau_{02}\right).
\end{split}
\end{equation*}
Here $G$ is defined by \eqref{def:G},  $|\theta_1|\leq 1$, $|\theta_2|\le 1$,
\begin{equation*}
%\label{est:12}
\begin{split}
&\tau_{01}:=-\frac 12\sigma  [\theta_1+\theta_2], \quad \tau_{02}:=-\sigma^2[\theta_1+\theta_2]+\sigma^2\delta\theta_2\theta_1+f\sigma  [\theta_2-\theta_1],\\&
|\tau_{01}|\le u(H)=: L_1(H), \quad  |\tau_{02}|\le 3u^2(H)+2Hu(H) =: L_2(H).
\end{split}
\end{equation*}
Analogously, for $i=1, \dots, l-1$,
\begin{equation*}
\begin{split}
&\frac1{2\delta (2l-1)}\left[\int_{-1+\frac {2i}{2l-1}-\delta}^{  -1+\frac {2i}{2l-1}+\delta} \left|  f+\sigma s\right|^{-1}ds+
\int_{1-\frac {2i}{2l-1}-\delta}^{  1-\frac {2i}{2l-1}+\delta} \left|  f+\sigma s\right|^{-1}ds\right]\\
%&=\frac{2\delta}{2\delta (2l-1)}\left[\left[ - f-\sigma \left(-1+\frac {2i}{2l-1}+\delta \theta_{i1}\right)\right]^{-1}+\left[  f+\sigma \left(1-\frac {2i}{2l-1}-\delta \theta_{i2}\right)\right]^{-1} \right]\\
= &
\frac{1}{ 2l-1}\frac{2\left[\sigma \left(1-\frac {2i}{2l-1}\right)+\delta \tau_{i1}\right]}{\sigma^2 
\left(1-\frac {2i}{2l-1}\right)^2-f^2+\delta  \tau_{i2}}=\frac{2}{ 2l-1}G\left(f, \, \sigma \left(1-\frac {2i}{2l-1}\right), \, \delta  \tau_{i1}, \,\delta  \tau_{i2}\right).
\end{split}
\end{equation*}
Here $|\theta_{i1}|\leq 1$, $|\theta_{i2}|\le 1$,
%\begin{equation*}
%\label{est:i12}
%\begin{split}
%& 
$\displaystyle \tau_{i1}:=-\frac 12\sigma  [\theta_{i1}+\theta_{i2}]$, $\displaystyle \tau_{i2}:=-\sigma^2 \left(1-\frac 
{2i}{2l-1}\right)[\theta_{i1}+\theta_{i2}]+\sigma^2\delta\theta_{i2}\theta_{i1}+f\sigma [\theta_{i2}-\theta_{i1}]$,
%& 
$\displaystyle |\tau_{i1}|\le u(H) =: L_1(H)$, $\displaystyle |\tau_2|\le 3u^2(H)+2Hu(H) =: L_2(H)$. 
%\end{split}
%\end{equation*}

Based on \eqref{est:Glambda1}, we claim that for each 
\begin{equation*}
%\label{def:delta}
\delta<\min\left\{ \frac1{2l-1}, \, \frac 1{(2l-1)\left[H( \frac 12+\lambda)^{-1}+1\right]}, \,\, \frac {\varsigma}{\max\{L_1(H), \, L_2(H) \}}  \right\}
\end{equation*}
and $i=0, 1, \dots, l-1$, we have 
$
\left(f, \, \sigma \left(1-\frac {2i}{2l-1}\right), \, \delta  \tau_{i1}, \,\delta  \tau_{i2}\right)\in \mathcal A_\lambda,
$
which yields that 
\\
$
G\left(f, \, \sigma \left(1-\frac {2i}{2l-1}\right), \, \delta  \tau_{i1}, \,\delta  \tau_{i2}\right)\le 1-\frac {\varepsilon_{\lambda}}2$.
The above argument implies
\begin{equation*}
\begin{split}
&\mathbb E\left[|f(x_n)+\sigma(x_n) \xi_{n+1}|^{-1}\bigl|\mathcal F_n\right]\\
= & \frac{1}{ 2l-1}G\left(f, \, \sigma , \, \delta  \tau_{01}, \,\delta  \tau_{02}\right)
+\frac{2(l-1)}{ 2l-1}G\left(f, \, \sigma \left(1-\frac {2i}{2l-1}\right), \, \delta  \tau_{i1}, \,\delta  \tau_{i2}\right)
\\ \le & \left[\frac{2l-2}{ 2l-1}+\frac{1}{ 2l-1}\right]\left( 1-\frac{\varepsilon_{\lambda}}2\right)= 1-\frac{\varepsilon_{\lambda}}2<1.
 \end{split}
\end{equation*}
\end{proof}
%%%%%%%%%%%%

\begin{proof} [Proof of Theorem  \ref{thm:destabtrunc}]
First, we prove that whenever $x_0\in (0, b)$, there exists some random $\mathcal N_0$ such that  $x_{\mathcal N_0}>b$, a.s.  
Assume the contrary:  for all $n\in \mathbb N$ we have $x_n<b$ on some $\Omega_1\subseteq \Omega$, where $\mathbb P(\Omega_1)>0$.

 {\it Case (a). } Reasoning as in Theorem~\ref{lem:Ealpha123}, (ii),  we conclude that for $x\in [0, b]$, expression 
\eqref 
{expr:-1} is well defined. By \eqref{ineq:sdud}, for $\sigma_b$ defined as in \eqref{def:modsigmal} and $x<0$, a.s.,
\[
|f(x)+\sigma_b(x) \xi |\ge |\sigma_b(x) \xi |-|f(x)|\ge \frac{1+\sqrt{1+4f^2(x)}}{2}-f(x)>\frac 12.
\]
In addition, for  $x\ge b$ expression \eqref {expr:-1} is also well defined since, a.s., $
f(x)+\sigma_b(x) \xi_{n+1}=f(x)>0$. Further,
for each $n\in \mathbb N$  we have
 \[
|x_{n}|=x_{0}\prod_{i=0}^{n-1}\left| f(x_i)+\sigma_{b}(x_i)\xi_{i+1} \right|=x_0\frac 1{M_n}\frac 1{\prod_{i=1}^{n-1}\mathbb E\left[ \left| f(x_i)+\sigma_{b}(x_i)\xi_{i+1} \right|^{-1}\bigl |\mathcal F_{i}\right]},
\]
where $M_n$ is defined by \eqref{def:mart} with $\alpha=1$.  By assumption,  $x_i<b$ on $\Omega_1$ for all $i\in \mathbb N$. Then,  by the choice of $\sigma_{b}$ in \eqref{def:modsigmal},  we have $\mathbb E\left[ \left| f(x_i)+\sigma_{b}(x_i)\xi_{i+1} \right|^{-1}\bigl |\mathcal F_{i}\right]\le 1$ on $\Omega_1$, see \eqref{calc:Edud}. Also the non-negative martingale $(M_n)_{n\in \mathbb N}$ converges to an a.s. finite limit.
This implies  that there exists a random $\mathcal N\in \mathbb N$ such that, 
on $\Omega_1$,  $|x_{k}|\ge \frac{x_{0}}{M_{k}}\ge \varsigma(\omega)>0$ for some random a.s. positive $\varsigma(\omega)$ and $k>\mathcal N$. So there exists 
$\Omega_2\subseteq \Omega_1$ with $\mathbb P(\Omega_2)>0$, and a nonrandom $c>0$ such that $|x_{k}|\ge c$ on $\Omega_2$ for $k\ge 
\mathcal N$.  

If for  a certain $k\ge \mathcal N$, we have $x_k>0$ on some $\Omega_{21} \subseteq \Omega_2$ with $\mathbb P(\Omega_{21})>0$ then  it should be $c\le x_k<b$, i.e. we have $c<b$.   Assume now that $x_{k}<0$  on $\Omega_2$ for   $k\ge \mathcal N$. 
Applying Lemma \ref{lem:topor} with $J=1$, we conclude that there exists $\mathcal N_1>\mathcal N$ such that $\xi_{\mathcal N_1+1}\le -\frac 1{2l-1}$. 
Suppose  that it is the first such moment after $\mathcal N$. 
By \eqref{def:modsigmal} and \eqref{ineq:sdud}, this implies $f(x_{\mathcal N_1})+\sigma_{b}(x_{\mathcal N_1})\xi_{{\mathcal N_1}+1}<0$,
and therefore, $x_{\mathcal N_1+1}>0$, which leads up to the previous case.  So we have $0<c<b$.

Denote
\[
a(c, b) := \inf_{x\in(c, b)}f(x)>0,
\]
where positivity is by   Assumption~\ref{as:Fbd}~(i). Let  
\begin{eqnarray*}
Q_l(c, b):&=&\inf_{x\in  (c, b)}\left[f(x)+\frac 1{2l-1}\sigma_{b}(x)\right] \geq \inf_{x\in  (c, b)}\left[f(x)+\frac{1+\sqrt{1+4f^2(x)}}{2}\right]\\&\ge& 
\left[a(c, b)+\frac{1+\sqrt{1+4a^2(c, b)}}{2}\right]\ge 1+\varepsilon_{a},
\end{eqnarray*}
where $\varepsilon_{a}>0$ and then  $Q(c, b)>1$. Set 
\[
 j_1:=\left\lfloor \frac{\ln\frac bc}{\ln {Q(c, b)}}\right\rfloor +1= \left\lfloor  \log_{Q(c, b)} \frac{b}{c} \right\rfloor +1,
 \]
where $\lfloor t \rfloor$ is the largest integer not exceeding $t$, then 
$\displaystyle
cQ(c, b)^{j_1}>b$.
By Lemma \ref{lem:topor}, for $J=2j_1+1$  we can find $\mathcal N_2>\mathcal N_1$ such that
\begin{equation}
\label{def:xil1}
\begin{split}
& \xi_i=\frac 1{2l-1}, \,\,  i=\mathcal N_2, \dots, \mathcal N_2+j_1-1, ~~ \\ & \xi_{\mathcal N_2+j_1}=-\frac1{2l-1}, \
~\xi_{\mathcal N_2+j_1+i}=\frac 1{2l-1}, \,\,  i=1, \dots, j_1.
\end{split}
\end{equation}
Recall that by assumptions on $\Omega_1$, $\Omega_2\subseteq \Omega_1$ and for $n \ge \mathcal N_2-1$ we have either $x_n\in [c,b]$ if $x_n$ is positive, or $x_n\le -c$ if $x_n$ is negative.

 Assume first that $x_{\mathcal N_2-1}\in [c, b]$, on some  $\Omega_3\subseteq \Omega_2$ with  $\mathbb P(\Omega_3)>0$.  So $f(x_i)\ge a(c, b)$ for $i=\mathcal N_2-1, \dots,  \mathcal N_2+j_1-2$, and, applying  \eqref{ineq:sdud} and \eqref{def:xil1}, we get 
\begin{equation*}
\begin{split}
x_{\mathcal N_2+j_1-1}&= x_{\mathcal N_2-1}\prod_{i=\mathcal N_2-1}^{\mathcal N_2+j_1 -2}\left[ f(x_i)+\sigma(x_i)\xi_{i+1}\right]
%= x_{\mathcal N_2-1}\prod_{i=\mathcal N_2-1}^{\mathcal N_2+j_1 -2}\left[ f(x_i)+(2l-1)\frac{1+\sqrt{1+4f^2(x_i)}}{2}\xi_{i+1}\right]
\\&> x_{\mathcal N_2-1}\prod_{i=\mathcal N_2-1}^{\mathcal N_2+j_1 -2}\left[ f(x_i)+\frac{1+\sqrt{1+4f^2(x_i)}}{2}\right]\ge cQ^{j_1}>b,
\end{split}
\end{equation*}
which contradicts to our assumption  that $x_n\le b$ on all $\Omega_1$ for each $n\in \mathbb N$.  

Now assume that $x_{\mathcal N_2-1}<0$ on some $\Omega_4\subseteq \Omega_2$ with  $\mathbb P(\Omega_4)>0$.  By \eqref{def:xil1}, we get on $\Omega_4$ for $m=\mathcal N_2-1, \, \mathcal N_2, \dots,  \mathcal N_2+j_1-2$,
\[
f(x_m)+(2l-1)\frac{1+\sqrt{1+4f^2(x_m)}}{2}\xi_{m+1}\ge \frac 12>0\quad \mbox{so} \quad x_{\mathcal N_2+j_1-1}<0.
\]
However, for $m=\mathcal N_2+j_1-1$, 
\[
f(x_m)+(2l-1)\frac{1+\sqrt{1+4f^2(x_m)}}{2}\xi_{m+1}=f(x_m)-\frac{1+\sqrt{1+4f^2(x_m)}}{2}\le -\frac 12<0, \]
so $x_{\mathcal N_2+j_1}>0$. 

If $x_{\mathcal N_2+j_1}\in (c, b)$ on some $\Omega_5\subseteq \Omega_2$ with  $\mathbb P(\Omega_5)>0$,
we can apply the same  reasoning  as above for $x_{\mathcal N_2+j_1}$ and obtain 
\[
x_{\mathcal N_2+2j_1}= x_{\mathcal N_2+j_1}\prod_{i=\mathcal N_2+j_1}^{\mathcal N_2+2j_1-1 }\left| f(x_i)+\frac{1+\sqrt{1+4f^2(x_i)}}{2}\right|\ge cQ^{j_1}>b,
\]
which again contradicts to our assumption.  The same holds if $x_{\mathcal N_2+j_1}>b$.

{\it Case (b). }   Choosing $\sigma$ such that $\sigma-f$ is continuous on $[-b,b]$, from condition  \eqref{cond:fe} 
we conclude that there exists $v>0$ such that
\begin{equation}
\label{est:sigmafv}
\sigma(x) >f(x)+v, \quad x\in [-b, b].
\end{equation}
Define 
\begin{equation}
\label{def:varepsilon}
\varepsilon < \min\left\{1, \, 1-\frac {1}{H} \sup_{x \in [-b,b]} f(x), \,  \frac{v}{H} \right\},
\end{equation}
and
%\begin{equation*}
$\displaystyle \varsigma_1:=\mathbb P\{\xi\in (-1, \, -1+\varepsilon/2) \}>0$, $\displaystyle \varsigma_2:=\mathbb P\{\xi\in 
(1-\varepsilon/2, 1)>0
 \}$.
%\end{equation*}
Note also that   \eqref{est:sigmafv} and \eqref{def:varepsilon} imply
\begin{equation}
\label{est:sigmafeH}
\sigma(x) > f(x)+v\ge f(x)+\varepsilon H, \quad x\in [-b, b].
\end{equation}
The beginning of the proof  is the same as in (a): we assume that $x_n<b$ for all $n\in \mathbb N$
and get  that $|x_k| \geq c>0$ on $\Omega_2$, $k>\mathcal N$.  To show that 
 $x_k<0$ for all  $k>\mathcal N$ is impossible,  we note that 
there exists $\mathcal N_1>\mathcal N$ such that $\xi_{\mathcal N_1+1}<-1+\varepsilon/2$.  Then, by \eqref{est:sigmafeH},  
$$
\begin{array}{ll}
\displaystyle f(x_{\mathcal N_1})+\sigma(x_{\mathcal N_1})\xi_{\mathcal N_1+1} & 
\displaystyle < f(x_{\mathcal N_1})+(f(x_{\mathcal N_1})+\varepsilon H)(-1+\varepsilon/2) \\ &
\displaystyle =\varepsilon f(x_{\mathcal N_1})/2-\varepsilon H(1-\varepsilon/2)<\frac{\varepsilon H}2(-1+\varepsilon)<0,
\end{array}
$$
which implies $x_{\mathcal N_1+1}>0$. So, in the same way as in (a), we conclude that $0<c<b$.
 
Since $\sigma(x)>e$, $f$ is positive on $(c, b)$, and due to the choice of $\varepsilon$, we can define %and estimate 
\begin{equation}
\label{def:Qs}
\begin{split}
Q_s(c, b):= & \inf_{x\in  (c, b)}\left[f(x)+\sigma_{b}(x)\left( 1-\frac{\varepsilon}{2} 
\right) \right] \\ > & a(b,c)+e\left( 1-\frac{\varepsilon}{2} \right) \geq 
e\left( 1-\frac{\varepsilon}{2} \right)>1.
\end{split}
\end{equation}
Set  
$\displaystyle
 j_1:=\left\lfloor \frac{\ln\frac bc}{\ln {Q(c, b)}}\right\rfloor +1= \left\lfloor \log_Q \frac{b}{c} \right\rfloor +1,
$ 
 then $\displaystyle cQ_s^{j_1}>b$. By Lemma~\ref{lem:topor} for $J=2j_1+1$  we can find $\mathcal N_2>\mathcal N_1$ such that
\begin{equation}
\label{def:xis}
\begin{split}
&
\xi_i\in (1-\varepsilon/2, 1), ~ i=\mathcal N_2, \dots, \mathcal N_2+j_1-1, ~ \xi_{\mathcal N_2+j_1}\in (-1, -1+\varepsilon/2), 
\\
&
\xi_{\mathcal N_2+j_1+i}\in (1-\varepsilon/2, 1), ~~i=1, \dots, j_1.
\end{split}
\end{equation}
Recall that by assumptions on $\Omega_2$ and for $n \ge \mathcal N_2-1$, we have either $x_n\in [c,b]$ if $x_n$ is positive, or 
$x_n\le -c$ if $x_n$ is negative.

 Assume first that $x_{\mathcal N_2-1}\in [c, b]$, on some  $\Omega_3\subseteq \Omega_2$ with  $\mathbb P(\Omega_3)>0$.  So, by  \eqref{def:xis}, 
 $x_i\in(c, b)$  for $i=\mathcal N_2-1, \dots \mathcal N_2+j_1-2$.
Then, by \eqref{def:Qs} and \eqref{def:xis}, we have 
\begin{equation*}
\begin{split}
x_{\mathcal N_2+j_1-1}&= x_{\mathcal N_2-1}\prod_{i=\mathcal N_2-1}^{\mathcal N_2+j_1 -2}\left[ f(x_i)+\sigma_b(x_i)\xi_{i+1}\right]\ge cQ_s^{j_1}>b,
\end{split}
\end{equation*}
which contradicts to our assumption  that $x_n\le b$ on all $\Omega_1$.  

Now assume that $x_{\mathcal N_2-1}<0$
on some $\Omega_4\subseteq \Omega_2$ with  $\mathbb P(\Omega_4)>0$.  By \eqref{def:xis}, we get on $\Omega_4$ for $m=\mathcal N_2-1, \, \mathcal N_2, \dots,  \mathcal N_2+j_1-2$,
\[
f(x_m)+\sigma_b(x_m)\xi_{m+1}\ge \sigma_{b}(x)(1-\varepsilon/2)>1,\quad \mbox{so} \quad x_{\mathcal N_2+j_1-1}<0.
\]
However, for $m=\mathcal N_2+j_1-1$, also, by \eqref{est:sigmafeH} and \eqref{def:xis},
\begin{align*}
& f(x_m)+\sigma_b(x_m)\xi_{m+1}=f(x_m)+\sigma_{b}(x_m)(-1+\varepsilon/2)\\ \le &-\left(f(x_m)+\varepsilon H \right)(+1-\varepsilon/2)+f(x_m)
\le -\frac{\varepsilon H}2(1-\varepsilon)<0, 
\end{align*}
so $x_{\mathcal N_2+j_1}>0$.  If $x_{\mathcal N_2+j_1}\in (c, b)$ on some $\Omega_5\subseteq \Omega_2$ with  $\mathbb P(\Omega_5)>0$, we can apply the same  reasoning  as above for $x_{\mathcal N_2+j_1}$, use \eqref{def:xis} and obtain that
$x_{\mathcal N_2+2j_1}>b$, which again contradicts to our assumption. 

Thus, in both cases,  (a) and (b), the  solution started in $(0, b)$ will be greater than $b$ at some a.s. finite random moment $\mathcal N_4=\inf\{n: x_n>b \}$.  Then we have either $x_{\mathcal N_4-1}\in (0, b)$, on some on  $\Omega_{11}$,  or $x_{\mathcal N_4-1}\in (-b, 0)$, on some on  $\Omega_{12}$, $\mathbb P[\Omega_{11}], \mathbb P[\Omega_{12}]>0$. In the first case, by  \eqref{cond:Fbd}  we have $x_{\mathcal N_4}\in [b, d]$ and due to  the fact that $\sigma(x)=0$ for $x\ge b$ we get that, on~$\Omega_{11}$, 
\begin{equation}
\label{est:N4k}
x_{\mathcal N_4+k}\in [b, d], \quad \forall k\in \mathbb N.
\end{equation}
 If $x_{\mathcal N_4-1}\in (-b, 0)$ we have 
$\displaystyle
f(x_{\mathcal N_4-1}) +\sigma(x_{\mathcal N_4-1})\xi_{{\mathcal N_4}}\ge f(x_{\mathcal N_4-1}) -\sigma(x_{\mathcal N_4-1}),
$ 
and, by  \eqref{cond:Fbd}, 
\[
x_{\mathcal N_4}=x_{\mathcal N_4-1}\left[f(x_{\mathcal N_4-1}) +\sigma(x_{\mathcal N_4}-1)\xi_{{\mathcal N_4}}\right]\le 
x_{\mathcal N_4-1}\left[f(x_{\mathcal N_4-1}) -\sigma(x_{\mathcal N_4-1})\right]\le d,
\]
which again implies \eqref{est:N4k}.
If  $x_0\in (b, d)$,  we immediately get that $x_1\in (b, d)$ and therefore, $x_n\in (b, d)$ for all $n>1$.  
 \end{proof}
%%%%%%%%%%%%%%%%%

  \begin{remark} 
  \label{rem:Fm}
    Theorem \ref{thm:destabtrunc} remains correct  for distribution \eqref{def:dud} with $l=1$  and $f(x)=0$ for $x<0$, when  instead of truncated equation \eqref{eq:sdnoiseb-b}  we consider 
\begin{equation*}
%\label{eq:sdnoiseb}
x_{n+1}=x_n f(x_n)+\sigma_b(x_n) x_n \xi_{n+1}, \quad x_0\in (0, b), \quad n\in {\mathbb N}_0,
\end{equation*} 
and  take 
$\displaystyle
\sigma(x)=\frac{1+\sqrt{1+4f^2(x)}}2+\frac 1{2H+2}f(x)$, $x\in \mathbb R$.
\end{remark}

\end{document}